\documentclass[10pt,a4paper]{article}
\usepackage[nottoc]{tocbibind}
\usepackage{graphicx}
\usepackage[export]{adjustbox}
\usepackage{slashed}
\usepackage{enumerate}
\usepackage{amsmath}
\usepackage[hidelinks]{hyperref}
{
	
	\newtheorem{assumption}{Assumption}
}
\usepackage{amsfonts}
\usepackage{psfrag}
\usepackage{color}
\usepackage{mathtools}
\usepackage{pgf,tikz}
\usepackage{mathrsfs}
\usetikzlibrary{arrows}
\usepackage{fancyhdr}
\usepackage{amssymb}
\usepackage{slashed}
\usepackage{enumitem}
\usepackage{verbatim}

\usepackage[T1]{fontenc}
\usepackage[utf8]{inputenc}
\usepackage{authblk}

\newtheorem{remark}{Remark}

\newtheorem{proposition}{Proposition}[section]

\newtheorem{theorem}{Theorem}[section]
\newtheorem{corollary}{Corollary}[section]
\newtheorem{lemma}{Lemma}[section]
\newtheorem{conjecture}{Conjecture}[section]

\numberwithin{equation}{section}
\allowdisplaybreaks[4]

\newenvironment{proof}{\smallskip\noindent\emph{Proof.}\hspace{1pt}}%
{\hspace{-5pt}{\nobreak\quad\nobreak\hfill\nobreak$\square$\vspace{8pt}%
		\par}\smallskip\goodbreak}
\newcommand{\csigma}{\check{\sigma}}
\newcommand{\Lbar}{\underline{L}}
\newcommand{\Hbar}{\underline{H}}
\newcommand{\dal}{\delta^{\frac{1}{2}}}
\newcommand{\uprime}{u^{\prime}}

\newcommand{\upr}{u^{\prime}}

\newcommand{\omegabar}{\underline{\omega}}
\newcommand{\psibar}{\underline{\psi}}

\newcommand{\inftySu}[1]{\lVert{#1} \rVert_{L^{\infty}(S_{u,\ubar})}}

\newcommand{\inftySuubarprime}[1]{\lVert{#1} \rVert_{L^\infty(S_{u,\ubar^\prime})}}
\newcommand{\twoSudelta}[1]{\lVert {#1} \rVert_{L^2(S_{u,\delta})} }
\newcommand{\TwoSudelta}[1]{\big\lVert {#1} \big \rVert_{L^2(S_{u,\delta})} }
\newcommand{\inftySudelta}[1]{\lVert {#1} \rVert_{L^{\infty}(S_{u,\delta})} }

\newcommand{\enmu}{\end{multline}}

\newcommand{\chihat}{\hat{\chi}}
\newcommand{\g}{\gamma}
\newcommand{\al}{a^{\frac{1}{2}}}
\newcommand{\chibar}{\underline{\chi}}
\newcommand{\chibarhat}{\underline{\hat{\chi}}}
\newcommand{\ubar}{\underline{u}}
\newcommand{\e}{\mathrm{e}}
\newcommand{\be}{\begin{equation}}
\newcommand{\ee}{\end{equation}}

\newcommand{\bemu}{\begin{multline}}

\newcommand{\dubarprime}{\hspace{.5mm} \text{d}\ubar^{\prime}}

\newcommand{\tru}{\text{tr}\chi - \frac2u + \frac{4m_0}{u^2}}
\newcommand{\trubar}{\text{tr}\chibar + \frac2u}
\newcommand{\hsp}{\hspace{.5mm}}
\newcommand{\hs}{\hspace{.5mm}}

\newcommand{\tildetr}{\widetilde{\tr \chibar}}

\newcommand{\kmax}{k_{\text{max}}}

\newcommand{\twoSu}[1]{\lVert{#1} \rVert_{L^2(S_{u,\ubar})}}
\newcommand{\TwoSu}[1]{\big\lVert{#1} \big\rVert_{L^2(S_{u,\ubar})}}

\newcommand{\twoSuprime}[1]{\lVert{#1} \rVert_{L^2(S_{u^\prime,\ubar})}}
\newcommand{\twoSubarprime}[1]{\lVert{#1} \rVert_{L^2(S_{u,\ubar^{\prime}})}}

\newcommand{\twoSuubarprime}[1]{\lVert{#1} \rVert_{L^2(S_{u,\ubar^\prime})}}

\newcommand{\alphabar}{\underline{\alpha}}
\newcommand{\betabar}{\underline{\beta}}
\newcommand{\etabar}{\underline{\eta}}

\newcommand{\Hutwo}[1]{\lVert #1 \rVert_{L^2(H_u)}}
\newcommand{\Hbarubartwo}[1]{\lVert #1 \rVert_{L^2(\Hbar_{\ubar})}}

\newcommand{\duprime}{\hspace{.5mm} \text{d}u^{\prime}}

\newcommand{\Hodge}[1]{\prescript{*}{}{#1}}

\newcommand{\chihatbar}{\hat{\chibar}}

\def\crho {\check{\rho}}

\def\f {\frac}
\def\i {\infty}
\def\l {\bigg(}
\def\r {\bigg)}
\def\S {S_{u,\underline{u}}}

\def\K{K-\frac{1}{|u|^2}}

\renewcommand{\div}{\mbox{div }}
\newcommand{\curl}{\mbox{curl }}

\newcommand{\tr}{\mbox{tr}}
\newcommand\restr[2]{{
	\left.\kern-\nulldelimiterspace 
	#1 
	\vphantom{\big|} 
	\right|_{#2} 
}}

\newcommand{\uubarSuu}[4]{\lVert #1 \rVert_{L_{u}^{#2}L_{\ubar}^{#3} L^{#4}(S_{u,\ubar}) }}

\newcommand{\biguubarSuu}[4]{\Big \lVert #1 \Big \rVert_{L_{u}^{#2}L_{\ubar}^{#3} L^{#4}(S_{u,\ubar}) }}
\newcommand{\bigubaruSuu}[4]{\Big \lVert #1 \Big \rVert_{L_{\ubar}^{#2}L_{u}^{#3} L^{#4}(S_{u,\ubar}) }}

\newcommand\restri[2]{{
	\left.\kern-\nulldelimiterspace 
	#1 
	\right|_{#2} 
}}

\definecolor{ffqqqq}{rgb}{1.,0.,0.}
\definecolor{uuuuuu}{rgb}{0.26666666666666666,0.26666666666666666,0.26666666666666666}

\begin{document}

\title{Construction of Cauchy data for the dynamical formation of apparent horizons and the Penrose Inequality}
\author[1]{Nikolaos Athanasiou}
\author[2]{Martin Lesourd}
\affil[1]{Department of Mathematics, Oxford University, Oxford}
\affil[2]{Black Hole Initiative, Harvard University, Cambridge MA}

\maketitle

\begin{abstract}
Based on scale critical initial data, we construct smooth asymptotically flat Cauchy initial data for the Einstein vacuum system that does not contain Marginally Outer Trapped Surfaces (MOTS) but whose future evolution contains a trapped region, which itself is bounded by an apparent horizon (a smooth hypersurface foliated by MOTS).  \\ \indent 
Although the long time behaviour of these solutions is unknown, a statement of Kerr Stability would yield a dynamical, scale critical, non-spherically symmetric class of vacuum examples for the conjectures of Weak Cosmic Censorship and Final State.  \\ \indent
Owing to estimates obtained for the ADM mass of the data and the area of the MOTS foliating the apparent horizon, this construction yields a dynamical setting in which to test the conjectured spacetime Penrose Inequality. We show that the inequality holds in an open region in the future of the initial data, which itself can be controlled by the parameters of the initial data. 
\end{abstract}

\section{Introduction}
General relativity provides a framework to describe the structure of Lorentzian $(n+1)$-manifolds $(\mathcal{M},\overline{g})$ obeying the Einstein field equations 
\begin{equation}
\text{Ric}_{\overline{g}} - \frac{1}{2}\overline{g}\: R_{\overline{g}}=T
\end{equation}
where $R_{\overline{g}}$ is the scalar curvature of $\overline{g}$ and $T$ is a source term for possible matter fields in $\mathcal{M}$. Here we are concerned with the vacuum case of (1.1) where $T=0$ and the equations become
\begin{equation}
\text{Ric}_{\overline{g}} =0
\end{equation} 
Viewing (1.2) from the perspective of the Cauchy problem, the task is to specify suitable initial data and generate a spacetime by Cauchy evolution of this data. By `suitable' we mean that the initial data must satisfy the constraint equations. For $T=0$ these are
\begin{gather}
R_g- \lvert k \rvert^2 + (\tr k)^2 =0 \label{constraint1}, \\{\text{div}_g}k  -\text{d} (\tr k) =0\label{constraint2}
\end{gather}
where $M$ is an $n$-dimensional Riemannian manifold with metric $g$ and $k$ is a symmetric two-tensor on $M$ with $M$ isometrically embedding as a spacelike hypersurface of a spacetime $(\mathcal{M},\overline{g})$ satisfying (1.2).\\ \\
Once a solution $(M,g,k)$ to (1.3-4) is given, the field equations describe how it generates a spacetime $(\mathcal{M},\overline{g})$ by Cauchy evolution. This Cauchy problem is well posed in the sense that, for any such $(M,g,k)$, there exists a continuously and uniquely determined (up to isometry) spacetime called the maximal Cauchy development of $(M,g,k)$. Obtaining spacetimes from general initial data sets and understanding their properties is the focus of many outstanding conjectures in mathematical general relativity.   \\ \\
In the language of PDE, these conjectures take the form of global existence and uniqueness questions, and a binding theme among them is to understand the development of singularities. Only in the specific setting of the spherically symmetric scalar field, due to Christodoulou \cite{C91}, \cite{C99}, is this understood. \\ \\
The family of Kerr black hole solutions provide an explicit class of vacuum, axisymmetric, asymptotically flat, singular spacetimes. This family is thought to be archetypal in the class of asymptotically flat vacuum black hole spacetimes, and indeed there are a number of conjectures aiming to provide a better understanding of how these fit within the space of \textit{generic} asymptotically flat vacuum solutions. A central problem is to determine whether and how, in the course of black hole formation, the spacetimes that are generated share the properties of this family. This is the so-called \textit{Final State Conjecture} and it states that the generic outcome of gravitational collapse is a black hole spacetime whose exterior geometry approaches some member of the family.   \\ \\
The purpose of this paper is to construct a class of Cauchy initial data sets that can serve as models for the formation of black holes for the Einstein vacuum system. The three main results are as follows. 
\begin{enumerate}
    \item Based on the \textit{scale critical} initial data of \cite{AL17}, we construct smooth Cauchy initial data without trapped surfaces or MOTS, whose future evolution contains both such surfaces. These MOTS are strung out in a smooth hypersurface, i.e. \textit{apparent horizon}, whose approach to a null hypersurface can be controlled by the initial data. This gives the first apparent horizon formation result from Cauchy initial data. 
    \item The estimates we obtain bring about a connection between Kerr stability and the conjectures of Weak Cosmic Censorship and Final State for \textit{scale critical} data. In particular, if a certain form of Kerr stability holds, then this construction would yield the first scale critical dynamical vacuum examples of the conjectures of Weak Cosmic Censorship and Final State. 
    \item Owing to estimates for the ADM mass of our initial data sets and the area of the MOTS produced by evolution, the construction yields the first dynamical (non-spherically symmetric) setting in which to test the conjectured general spacetime Penrose inequality. We prove that the inequality holds for an open region in the evolution, whose size that can be controlled by the initial data.
\end{enumerate}

\subsection{Previous Work}
The first trapped surface existence result is due to Schoen-Yau \cite{SY83}. The result was refined by Yau \cite{Y01}, and extended into a \textit{hoop conjecture} statement by Alaee-Lesourd-Yau \cite{ALY19}. These results are purely at the level of the initial data: one formulates conditions on a Riemannian manifold with boundary which imply the existence of a trapped surface within. In a landmark contribution, Christodoulou \cite{C09} found a way of forming a trapped surface \textit{dynamically} for the vacuum Einstein system. 
\begin{theorem}[Christodoulou 2009]\label{C09}
Consider the characteristic initial value problem for (1.2) such that $\underline{H}_0$ coincides with a backwards lightcone in Minkowski space for $0\leq u \leq 1$. For every $B>0$ and $u_*\leq 1 $, there exists $\delta=\delta(B,u_*)>0$ sufficiently small such that if the initial $\chihat_0$, prescribed on $H_1$ for $0\leq u\leq \delta$, satisfies 
\begin{equation}
    \sum_{i\leq 5,j\leq 3} \delta^{\frac{1}{2}+j}||\nabla^i \nabla^j_4 \chihat_0||_{L^\infty_{\ubar}  L^2(S_{u,\ubar})} \leq B
\end{equation}
then the solution to (1.2) remains regular in $u_*\leq u\leq 1$, $0\leq \ubar \leq \delta$. Moreover, if the initial data satisfies the lower bound 
\begin{equation}
    \inf_{\omega\in S_{1,0}}\int_0^\delta |\chihat_0(\ubar',\omega)|^2d\ubar'\geq M_*\geq 2(1-u_*)
\end{equation}
then after choosing $\delta$ sufficiently small (depending on $B$, $u_*$, and $M_*$) if necessary, the sphere $S_{u_*,\delta}$ is a trapped surface. 
\end{theorem}	
Christodoulou's argument relies on identifying a certain hierarchy among quantities that is preserved under the non-linear evolution of the Einstein vacuum system. Identifying and maintaining this hierarchy makes existence possible whilst permitting certain quantities to grow large, in particular those needed for trapped surface formation. \\ \\
Shortly thereafter, Klainerman-Rodnianski \cite{KR12} found a simplified and more direct argument for the formation of trapped surfaces, which reduced the number of derivatives needed from two of curvature to one.  \\ \\
Another major contribution was brought by Li-Yu \cite{LY15}, who found a way to re-express a version of Theorem 1.1 in the language of Cauchy initial data. Their idea was to improve the estimates of \cite{C09} in order to use the local deformation result of Corvino-Schoen \cite{C00}, \cite{CS05} and glue an asymptotically flat slice (isometric to Kerr outside a compact set) onto the dynamical slab in \cite{C09}. To achieve this extra control, they imposed that the initial shear specified in Theorem 1.1 satisfy
\begin{equation}
    m_0=\frac{1}{4}\int_0^\delta |\chihat_0(\ubar',\omega)|^2d\ubar'
\end{equation}
for some constant $m_0$, so that the total shear along $\ubar\in[0,\delta]$ is independent of $\omega$. With this assumption, they eventually obtained the following. 
\begin{theorem}[Li-Yu 2015]
Let $\Sigma$ be a $3$-dimensional differential manifold diffeomorphic to $\mathbb{R}^3$ and seperated into four concentric regions \[\Sigma=\Sigma_M\cup \Sigma_C \cup \Sigma_S \cup \Sigma_K \] 
with $\Sigma_M$ diffeomorphic to the $3$-ball, $\Sigma_C$ and $\Sigma_S$ diffeomorphic to the $3$-annulus, and $\Sigma_K$ diffeomorphic to $\mathbb{R}^3\backslash B^3$. Then for any $\epsilon>0$ sufficiently small, there is a Riemannian metric $g$ and a symmetric two tensor $k$ on $\Sigma$ satisfying (1.3-4) such that 
\begin{enumerate}
    \item $\Sigma_M$ is a constant time slice in Minkowski spacetime $(g,k)=(\delta_{ij},0)$,
    \item $\Sigma_K$ is isometric to a constant time slice all the way to spacelike infinity in a Kerr spacetime with mass $m$ and angular momentum $\textbf{a}$ satisfying $|m-m_0|+|\textbf{a}|\lesssim \epsilon$, 
    \item $\Sigma$ is free of trapped surfaces, 
    \item there are trapped surfaces in the development of $\Sigma$.
\end{enumerate}
\end{theorem}
Note here that $\Sigma_C$ is a spacelike hypersurface traversing the dynamical spacetime slab arises from the characteristic initial value problem of Theorem 1.1 with (1.7). \\ \\
At around the same time, Klainerman-Luk-Rodnianski \cite{KLR14} were able to find an anisotropic mechanism to form trapped surfaces. 
\begin{theorem}[Klainerman-Luk-Rodnianski 2015]
Take as starting point the characteristic initial value problem of Theorem 1.1. If (1.6) is replaced with 
\begin{equation}
    \sup_{\omega\in S_{1,0}} \int_0^\delta |\chihat_0(\ubar',\omega)|^2d\ubar'\geq M_* >0
\end{equation}
then, after choosing $\delta$ smaller if necessary, a compact trapped surface can be guaranteed to form to the future of the initial data, within the domain in which the solution remains regular. 
\end{theorem}
In replacing `inf' with `sup', they only require the initial shear to be large in the neighborhood of a \textit{single} geodesic, thus yielding an \textit{isotropic} formation result. \\ \\
In a different direction, An-Luk \cite{AL17} proved a trapped surface formation under `scale critical' data for the Einstein vacuum system. Their result made use of techniques developed by Luk-Rodnianski \cite{LR14}, \cite{LR17}, which those authors had developed to study interacting impulsive waves. 
\begin{theorem}[An-Luk 2017] \label{AL17}
Consider the following characteristic initial value problem the Einstein vacuum system. The initial incoming hypersurface $\underline{H}_0$ is required to coincide with a backwards lightcone in Minkowski space with $0\leq u \leq 1$. On the initial outgoing hypersurface $H_1$, the initial $\chihat_0$ satisfies 
\begin{equation}
    \sum_{i\leq 7} ||\nabla^i \chihat_0||_{L^\infty_{\ubar} L^2(S_{u,\ubar})} \leq a^{1/2}
\end{equation}
for $0\leq \ubar \leq \delta$. There exists a universal large constant $b_0$ such that if $b_0\leq b\leq a$ and $\delta a^{1/2}b<1$, then the unique solution to (1.2) remains regular in the region $\delta a^{1/2}b\leq u\leq 1$, $0\leq \ubar,\leq \delta$. Moreover, if the initial data also verify the lower bound 
\begin{equation}
    \inf_{\omega\in S_{1,0}}\int_0^\delta |\chihat_0(\ubar',\omega)|^2d\ubar'\geq 4a^{1/2}b\delta
\end{equation}
then the sphere $S_{b\delta a^{1/2},\delta}$ is trapped. 
\end{theorem}
Note that after choosing $a=B^2\delta^{-1}$ and $b=b_0$ one basically recovers a version of Theorem 1.1 as a corollary.
\begin{corollary}[An-Luk 2017]
Replace (1.9) with
\begin{equation}
    \sum_{i\leq 7} \delta^{1/2}||\chihat_0||_{L^\infty_{\ubar} L^2(S_{u,\ubar})}\leq B
\end{equation}
for $0\leq \ubar \leq \delta$. Then there exists a universal large constant $b_0$ such that the solution to (1.2) remains regular in $u_*\leq u\leq 1$, $0\leq \ubar\leq \delta$ for $u_*=b_0B\delta^{1/2}$. Moreover, if the initial data also verify the lower bound 
\begin{equation}
    \inf_{\omega\in S_{1,0}}\int_0^\delta |\chihat_0(\ubar',\omega)|^2d\ubar' \geq 4b_0B\delta^{1/2}
\end{equation}
then the sphere $S_{u_*,\delta}$ is a trapped surface.
\end{corollary}

\par \noindent The significance of Theorem 1.4 lies in the fact that the size of the incoming radiation, given by
\[ \inf_{\theta \in S_{1,0}} \int_{0}^{\delta} \lvert \chihat \rvert^2 (\ubar^{\prime}, \theta) \dubarprime, \]can be of the same order of magnitude as the length scale $\delta$. In particular, there exist initial data satisfying the conditions of Theorem 1.4 for which the metric is only large in $H^{\f32}$ and small in $H^s$ for all $s < \f32$. More precisely, one can construct initial data satisfying the conditions of Theorem \ref{AL17} in which
\[ \lVert \gamma \rVert_{H^s} \approx \al \delta^{\f32 -s}. \]This is in contrast to Theorem 1.1, in which the data are large in $H^s$ for all $s>1$. The significance of the $H^{\f32}$ space is that it is a critical space in terms of scaling considerations for the Einstein vacuum equations. It is in this sense that the data are termed \textit{mild}. Broadly speaking, therefore, scale-critical data can be thought of as the smallest initial data, in terms of size, known to produce a trapped surface in evolution.

\vspace{3mm} 
\par \noindent More recently, based on the spacetime whose existence was shown in \cite{AL17}, An \cite{A17} found an additional initial condition permitting to prove the existence of a \textit{apparent horizon}, i.e., a hypersurface in the spacetime foliated by MOTS. An's argument extends the ideas set forth in \cite{KLR14}, and yields the following.
\begin{theorem}[An 2018]
Consider a characteristic initial value problem for the Einstein system, where the initial data satisfies 
\begin{equation}
    \sum_{0\leq i <\infty,0\leq j<\infty} \delta^{j}a^{-\frac{1}{2}}||\nabla^j_{e_4} \nabla^i \chihat_0||_{L^\infty_{\ubar} L^2(S_{1,\ubar})} \leq B
\end{equation}
Require in addition that $a^{1/2}>b$ and
\begin{equation}
    \int_0^{\ubar} |\chihat_0(\ubar',\omega)|^2 d\ubar'=f(\ubar,\omega)\ubar a \hspace{0.2in} \emph{for each} \hspace{0.2in} \delta b a^{-1/2}\leq \ubar\leq \delta
\end{equation}
where $f(\ubar,\omega)$ is a smooth function such that $\frac{20}{21}\leq f(\ubar,\omega)\leq \frac{22}{21}$ and $\partial^{i}_{\omega} f(\ubar,\omega)\lesssim 1$ for all $i \in \mathbb{N}$ and $\omega\in S^2$, then, along every $\underline{H}_{\ubar}$ with $\ubar\in [b \delta a^{-1/2},\delta ]$, there exists a unique MOTS $M_{\ubar}$ and these join to make a smooth hypersurface.
\end{theorem}
The paper \cite{A17} contains other results, but we will only invoke the above. As in \cite{KLR14}, the argument involves studying an elliptic PDE that is singled out after computing the null expansion of spheres in the development. 
\vspace{3mm} 
\par \noindent
Most recently, Li-Mei \cite{LM20} extended \cite{LY15} by performing the gluing procedure of \cite{LY15} inside the black hole region of a pre-existing spacetime that is isometric to the Kerr solution in a neighborhood of future timelike infinity. By invoking \textit{Cauchy stability}, they can extend this interior region to a region outside the event horizon. In doing so, they thus obtain the following.
\begin{theorem}[Li-Mei 2020]
There exists a class of solutions $(M,g)$ to the Einstein vacuum equations that are isometric to the Kerr spacetime in a neighborhood of future timelike infinity such that
\begin{enumerate}
    \item there is a spacelike slice $\Sigma$ that does not intersect the black hole region $\mathcal{B}\subset M$,
    \item there is a trapped surface in the development of $\Sigma$.
\end{enumerate}
\end{theorem}
We describe below how Theorems 1.2 and 1.6 differ from the current work.

\subsection{Statement of Theorems}
The following results are based on the initial data described in Section 1.3. This data yields two main theorems, whose applications we describe in Sections 1.4-5. The first theorem is a straightforward extension of \cite{A17}, Theorem 1.5 above. 
\begin{theorem}[Formation of Apparent Horizon]
\label{main2} Consider the following characteristic initial value problem for the Einstein vacuum equations. The initial incoming null hypersurface $\Hbar_0$ is required to coincide with a backwards light cone in Minkowski space with $0\leq u \leq 1$, and assume that, along $H_0$, the initial shear $\chihat_0$ satisfies the assumptions described in Section 1.3. Then the Einstein vacuum equations admit a unique solution in the region given by $\ubar\in [0,2\delta]$ and $1\leq u\leq ba^{1/2}\delta$. Moreover, along each $\Hbar_{\ubar}$ with $\gamma \frac{a^{1/2}}{b}\delta \leq \ubar \leq 2\delta$ where $\gamma$ is a free $o(1)$ parameter, there exists a unique smooth marginally outer trapped surface $M_{\ubar}$, and for different $\ubar$, the union \[\bigsqcup_{\gamma \frac{a^{\frac{1}{2}}}{b} \delta \leq \ubar \leq 2\delta} M_{\ubar}   \] forms a $3$-dimensional smooth hypersurface. Moreover, with an additional condition on the initial data, this horizon can be shown to be spacelike, and it tends to a null hypersurface in a way that can be controlled by the initial data. 
\end{theorem}
\begin{figure}
\adjincludegraphics[height=9cm,trim={0cm 12.5cm 7.5cm 4cm},clip]{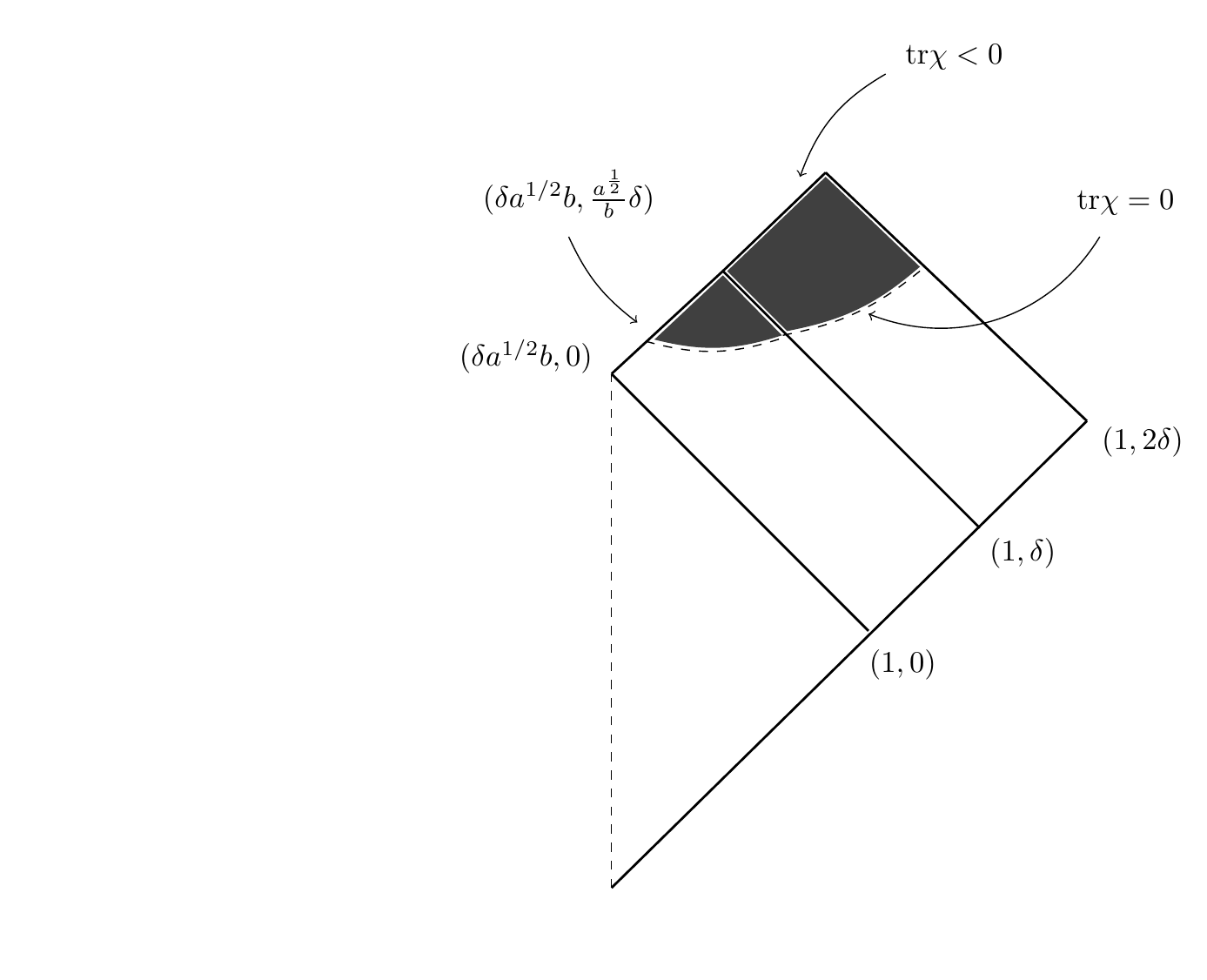}
\caption{Theorem 1.7.} 
\end{figure} To state the second theorem we briefly describe how to construct Cauchy initial data which can be glued onto the spacetime obtained in Theorem 1.7. Since the gluing takes place outside $\ubar = \delta$, the resulting spacetime will differ from that above in the region $\ubar\in [\delta,2\delta]$ above, and indeed this is why we cannot, as we did above for the region $\ubar\in [\delta,2\delta]$, use \cite{AL17} to infer existence.  \\ \\ 
Let $\Sigma$ be a $3-$dimensional differentiable manifold diffeomorphic to $\mathbb{R}^3$ and let $(x_1,x_2,x_3)$ be the standard coordinate system. Let $\lvert \cdot \rvert$ denote the usual radius function. Let $r_0>1$ be a given number. We divide $\Sigma$ into four concentric regions $\Sigma = \Sigma_M \cup \Sigma_{AL} \cup \Sigma_{S} \cup \Sigma_K$, where \[ \Sigma_M = \begin{Bmatrix}
x \hspace{.5mm} \mid \hspace{.5mm} \lvert x \rvert \leq r_0
\end{Bmatrix}, \hspace{2mm} \Sigma_{AL} = \begin{Bmatrix}
x \hspace{.5mm} \mid \hspace{.5mm} r_0 \leq \lvert x \rvert \leq  r_1
\end{Bmatrix},  \] \[ \Sigma_{S} =   \begin{Bmatrix}
x \hspace{.5mm} \mid \hspace{.5mm} r_1 \leq \lvert x \rvert \leq  r_2
\end{Bmatrix}, \hspace{2mm} \Sigma_K =   \begin{Bmatrix}
x \hspace{.5mm} \mid \hspace{.5mm} \lvert x\rvert \geq r_2
\end{Bmatrix}.   \]The numbers $r_1, r_2$ will be fixed later on such that $r_1-r_0 = O(\delta)$ and $r_2-r_1 = O(\epsilon_0)$ for some small positive $\epsilon_0$, $\delta$. \vspace{3mm}

\begin{theorem}[Cauchy Initial Data]
Take as starting point the region of existence covered by $1\geq u \geq \delta a^{1/2}b$ and $\ubar\in [0,\delta]$ in Theorem 1.7. Define
\begin{equation}
    \epsilon\equiv \frac{\delta^{1/2}a^{1/2}}{u^{1/2}}C>0
\end{equation} for some universal constant $C>0$. Then for any sufficiently small $\epsilon$, there exists an $\epsilon_{g}>0$, independent of $\epsilon$ once $\epsilon$ is picked sufficiently small, such that, to any of the spacetimes obtained in Theorem 1.7, we can associate a complete Cauchy initial data set $(\Sigma,g,k)$  
\begin{equation*}
\Sigma=\Sigma_M \cup \Sigma_{AL} \cup \Sigma_{S} \cup \Sigma_K
\end{equation*}
satisfying 
\begin{enumerate}
    \item $(g,k)=(\delta_{ij},0)$ on $\Sigma_M$,
    \item $\Sigma_{AL}$ is a spacelike hypersurface with boundary at $\ubar =0$ and $\ubar=\delta$ that traverses the region $\ubar\in [0,\delta]$ in Theorem 1.7,
    \item $\Sigma_K$ is isometric to a constant time slice all the way to spacelike infintiy in a Kerr spacetime with mass $m$ and angular momentum $\textbf{a}$ satisfying $|m-m_0|+|\textbf{a}|\lesssim \epsilon$,
    \item there is a unique solution to the Einstein vacuum equations in the region connecting 3 and 4 
    \[ (\ubar,u) \hspace{.5mm} \mid \hspace{.5mm}
\delta \leq \ubar \leq  \delta+\epsilon_{gr}, \hspace{.5mm} 1-\epsilon_{gr} < u \leq 1 \] traversed by $\Sigma_S$,
\item there are no trapped surfaces or MOTS on $\Sigma$, its future domain of dependence contains both such surfaces, and the MOTS are exactly those arising in Theorem 1.7 in the region $\ubar\in [\gamma\frac{a^{\frac{1}{2}}}{b} \delta,\delta]$.
\end{enumerate}
\end{theorem}
\begin{figure}
\adjincludegraphics[height=9cm,trim={2cm 10cm 0cm 5cm},clip]{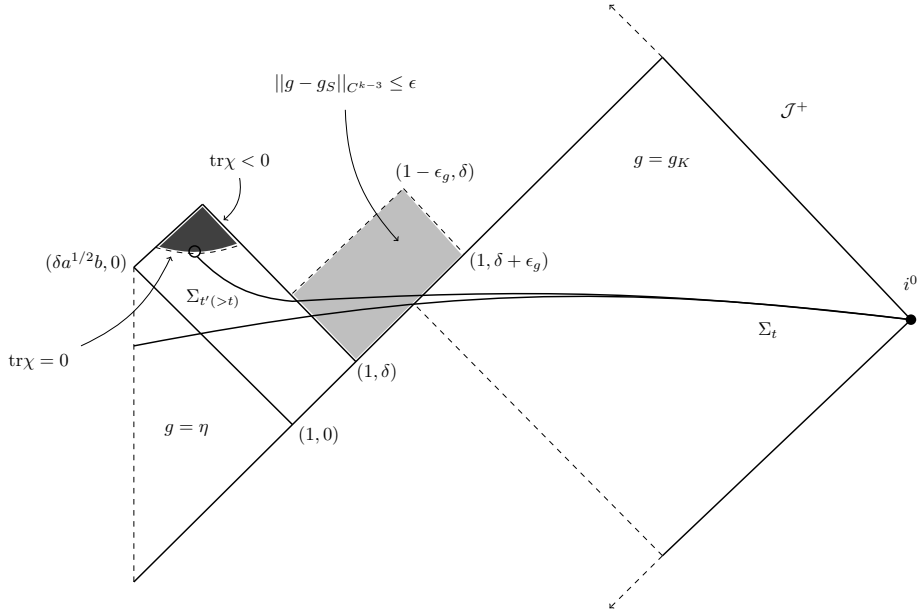}
\caption{Theorem 1.8.} 
\end{figure} 
In Figure 2 the dark region is trapped and bounded by MOTS which foliate a smooth hypersurface tending to a null hypersurface. The light gray region is close to Schwarzschild, the exterior region up to $i^0$ (and thus a portion of $\mathcal{J}^+$ is isometric to Kerr), and everything to the future of these regions would fall under the scope of Kerr Stability. Should stability hold for this data, $\Sigma_t$ would yield a non-spherically symmetric example of initial data realizing the conjectures of Weak Cosmic Censorship and Final State (that is not already isometric to Kerr in a neighborhood of future timelike infinity). $\Sigma_{t'>t}$ are the kind fo slices for which we give a dynamical test of the spacetime Penrose Inequality. \\ \\
We now note various differences with \cite{LY15} and \cite{LM20}, which might seem similar at first glance. In \cite{LY15} the goal is to translate \cite{C09} in the language of Cauchy data, and in this sense one could say that Theorem 1.8 is a translation of the scale critical result of \cite{AL17} into the language of Cauchy initial data. In doing so however, we find that scale critical data allows for greater control on the evolution and the following applications:
\begin{enumerate} 
    \item the dynamical formation of MOTS and of an \textit{apparent horizon},
    \item a different set of estimates which would connect Kerr Stability and the conjectures of Final State and Weak Cosmic Censorship,
    \item a quantitative understanding of how small $\delta$ needs to be relative to the other parameters, i.e., a theorem of the form `` for any $\delta<\cdot\cdot\cdot $ '' as opposed to `` for some $\delta$ sufficiently small '' as in \cite{C09}, \cite{KR12}, \cite{KLR14}, \cite{LY15}, \cite{LM20}. 
     \item a dynamical collapse setting in which to study and verify the conjectured spacetime Penrose inequality.
    \end{enumerate} 
We explain these points in more detail in Sections 1.5-6.

\subsection{The Initial Data}
Since the initial incoming null hypersurface $\Hbar_0$ is required to coincide with a backwards light cone in Minkowski space with $0\leq u \leq 1$, and since on $[\ubar=\delta,\delta+1]$ we ask for trivial data $\chihat=0$, the only non-trivial data is located on the null hypersurface $\ubar \in [0,\delta]$, henceforth $\chihat_0$. \\ \\
Our first requirement is that the data be smooth.
\begin{itemize}
    \item \textbf{Smooth.} Require that $\chihat_0 = 0$ on $\ubar=0$ and $\ubar=\delta$ and moreover that $\chihat \to 0$ smoothly, both in $\ubar$ and $\theta$, on approach of $\ubar=0$ and $\ubar=\delta$.
\end{itemize} 
We also require to be scale critical in the following sense.
\begin{itemize}
\item \textbf{Scale Critical.} From $0 \leq \ubar \leq \delta$, prescribe $\chihat_0$ such that
\begin{equation}
\label{AnLukansatz} \sum_{\substack{0\leq i < \infty, \\ 0 \leq j < \infty}} \delta^j a^{-\frac{1}{2}} \lVert \nabla_{\e_4}^j \nabla^i \chihat_0 \rVert_{L_{\ubar}^\infty L^2(S_{0,\ubar})} \leq B
\end{equation}
\end{itemize}
We also require a certain degree of angular independence.
\begin{itemize}
    \item \textbf{Averaged Angular Independence}. Once integrated along $\ubar\in [0,\delta]$,  the total shear is given by
    \begin{equation}
        \int_0^\delta |\chihat_0(\ubar,\omega)|^2d\ubar'=4m_0
    \end{equation}
    for a constant $m_0$ that is independent of $\omega$.
\end{itemize}
We now impose a condition on the $\ubar$ dependence of $|\chihat_0(\ubar,\omega)|$, where $\omega$ here denotes angular co-ordinates on the spheres along the null hypersurface $H_{u=1}$. This condition is key in An's work on the existence of a MOTS, and its specific form comes from the \cite{AL17} estimates in the dynamical slab.
\begin{itemize}
\item \textbf{$\ubar$-Dependence.} Let $\lambda$ be a free parameter such that $1>\lambda=1-o(1)$ for some $o(1)$ quantity, $\gamma$ a free $o(1)$ parameter, and $\mu$ a free parameter $>1$. Then require that for all $\ubar \in \left[\gamma \frac{a^{\frac{1}{2}}}{b} \delta, \lambda \delta \right]$ there holds
\begin{equation}
\int_0^{\ubar} | \chihat_0|^2 (\ubar', \theta) \hspace{.5mm} \text{d}\ubar' = a^{\frac{1}{2}}b^\mu f(\ubar,\omega) \ubar 
\end{equation} 
for a smooth, in both $\ubar$ and $\omega$, function $f(\ubar,\omega)$ satisfying $1-\frac{1}{c_1} \leq f(\ubar,\theta) \leq 1+\frac{1}{c_1}$ with $c_1>20$, and moreover $| \partial_{\omega}^i f(\ubar,\omega)| \lesssim 1$ for all $i \in \mathbb{N}$ and all $\omega \in \mathbb{S}^2$. \\ \\
For $\ubar \in [\lambda \delta, \lambda' \delta]$ where $\lambda'$ is a constant $\lambda <\lambda'<1-o(1)$, require that
\begin{equation}
    \int_0^{\ubar}|\chihat_0(\ubar',\omega)|^2 d\ubar' = a^{\frac{1}{2}} b^\mu  f(\ubar,\omega)\zeta(\ubar,\omega)\ubar + (1-\zeta(\ubar,\omega))4m_0
\end{equation}
where $\zeta(\ubar,\omega)$ is a smooth (in both $\omega$ and $\ubar$) cut-off function $=1$ for $\ubar \leq \lambda \delta$ and $=0$ for $\ubar \geq \lambda'\delta$, and moreover which satisfies $\zeta(\ubar) (1-\frac{1}{c_2})\leq \zeta(\ubar,\omega)\leq \zeta(\ubar)(1+\frac{1}{c_2}) $ for $c_2>20$, $|\partial^i_\omega \zeta(\ubar,\omega)|\lesssim 1$ for $ i \in \mathbb{N}$ and $\omega\in S^2$ where $\zeta(\ubar)$ is a smooth cut-off function in $\ubar$ which is $1$ at $\ubar=\lambda \delta$, and $0$ at $\ubar=\lambda'\delta$. 
\end{itemize}
For reasons which will become clear, we also require the data on $\ubar \in [\lambda \delta, \delta]$ to be small by an $o(1)$ factor compared to that on $\ubar\in [ \gamma  \frac{a^{\frac{1}{2}}}{b}\delta,\lambda \delta]$. 
\begin{itemize}
    \item \textbf{Dominant Contribution}. The contribution from $\ubar \in [0,\lambda \delta]$ is 
\begin{equation}
\begin{split}
\int_0^{\ubar} |\chihat_0(\ubar',\omega)|^2 d\ubar' = \int_0^{\lambda \delta} |\chihat_0(\ubar',\omega)|^2d\ubar' + \int_{\lambda \delta}^{\lambda'\delta} |\chihat_0(\ubar',\omega)|^2d\ubar' + \int_{\lambda'\delta}^{\ubar}|\chihat_0(\ubar',\omega)|^2d\ubar'\\
= M^*(\omega) +N(\lambda'\delta,\omega)+0 
\end{split}
\end{equation}
where $M^*(\omega)=\int_0^{\lambda \delta}|\chihat_0(\ubar',\omega)|^2d\ubar'$ and $M^*(\omega)\sim d_0N(\lambda'\delta,\omega)$ for some universal large constant $d_0$.
\end{itemize}
There turns out to be a topological fact which makes certain naive guesses for possible $|\chihat_0|$ unsuitable.
\begin{itemize}
    \item \textbf{A Topological Fact}. Due to the non-existence of non-vanishing linearly independent vector fields on $S^{2}$, a traceless two tensor on $S^{2}$ must vanish at least at one point. Thus we cannot simply choose  $|\chihat_0|$ to lie in some non-zero interval everywhere on each sphere. Moreover, in order to give $\chihat_0$ freedom to vary, its zero set ought not be fixed - for instance we do not wish to impose $\chihat_0=0$ along on a fixed null geodesic generator of the outgoing null hypersurface.
\end{itemize}
To see that initial data satisfying these examples can be constructed, it suffices to produce a symmetric example. \\ \\
\textbf{An Example}. Consider co-ordinates $(\omega,\ubar)$ on a $S^2\times [0,\delta]$ and separate $\ubar\in[0,\delta]$ into three intervals $I=[0,\frac{a^{\frac{1}{2}}}{b} \delta]$, $II=[\frac{a^{\frac{1}{2}}}{b}\delta,\lambda \delta]$ and $III=[\lambda\delta,\delta]$. \\ \indent
For $\omega \in \Omega$ for some solid angle $\Omega$ covering all but a region of size $o(1)$ of the sphere\footnote{$\Omega$ is such that $\int_{\Omega^2}\omega =4\pi-o(1)$.}, let  $\chihat_0(\omega,\ubar)=\mathfrak{f}(\ubar)$ be smooth and independent of $\omega$ so that $\int_0^{\ubar} |\chihat_0(\omega,\ubar')|^2d\ubar'$ grows monotonically in I, is linear in $\ubar$ in II, and reaches a constant $C$ at $\ubar=\delta$, and that the dominant contribution (by a factor given by a universal large constant) occurs in $I$ and $II$. \\ \indent
Throughout $I$, $II$, $III$, suppose that $\chihat_0(\ubar,\omega)=0$ only at a single point for each $\ubar\in (0,\delta)$, and that the zero set is constant in $\phi$ and visits an interval of size $o(1)$ in the $\theta \in [0,\pi]$ direction. Suppose that the zero set visits each point (say from $\frac{\pi}{2}- o(1)$ to $\frac{\pi}{2}+o(1)$) only once, and that for any $\ubar \in (0,\delta)$, $|\chihat_0|$ is constant outside of a neighborhood of $\omega_0$ where $\chihat_0(\omega_0)=0$ for that $\ubar$, and described by a smooth symmetric cut-off function around $0$. Require that
\begin{equation}
    \sum_{0\leq i< \infty}\delta^{j}||\nabla_{\e_4}^j \nabla^i\chihat_0(\ubar,\omega)||_{L^2(S_{\ubar})} \leq C_1 
\end{equation}
for some constant $C_1$ eventually chosen to satisfy (1.16). Clearly such a choice is possible.
\\ \indent 
Finally, to compensate for having $\chihat_0$ having been $0$ in the region $\omega\not\in \Omega$, we suppose that for $\omega\not\in \Omega$, $|\chihat_0(\ubar,\omega)|\in [0,\mathfrak{f}(\ubar)(1+o(1))]$ such that $\int_0^{\ubar} |\chihat_0(\ubar,\omega)|^2 d\ubar$ lies within $o(1)$ of $\int_0^{\ubar} |\chihat_0(\omega,\ubar)|^2$ and reaches the constant $C$ at $\ubar=\delta$.\\ \indent 
Note that since all the relevant quantities are $o(1)$, there is no potentially harmful `squeezing' of $\chihat_0$ in the angular directions.

\subsection{Outline of the Argument}

The argument combines ideas and results in \cite{CS05}, \cite{C09}, \cite{LY15}, \cite{LR14}, \cite{LR17}, \cite{AL17}, \cite{A17}, \cite{A19} and the basic strategy is as in \cite{LY15}.
\begin{enumerate}
\item Assume an averaged angular dependence on the initial data which permits obtaining greater control on the region of existence in \cite{AL17}.
\item Construct a transition region by posing initial data on a null hypersurface extending beyond the outgoing null cone in \cite{AL17}. 
\item Tune the initial data in a smooth and controlled fashion to obtain geometric closeness to Schwarzschild in the transition region. 
\item Invoke the gluing argument of \cite{LY15} and \cite{CS05} to study the Penrose inequality for the dynamically obtained MOTS. 
\end{enumerate}
Given that the strategy of the proof is broadly in the spirit of \cite{LY15}, let us note some crucial differences.
\begin{itemize}

\item{Since the initial data comes from \cite{AL17} and not \cite{C09}, we cannot resort, as in \cite{LY15}, to the exhaustive term-by-term analysis of \cite{C09} and must instead obtain certain key estimates directly.}
\item{Although our estimates are for renormalized quantities (following \cite{LR14}, \cite{LR17} and \cite{AL17}), we found that the most convenient way of constructing the transition region (especially the existence part) is to go back to the standard setting and employ \cite{L12}, rather than the more general \cite{LR17}. Going back and forth between the renormalized and standard setting is a useful method that can be exploited in other circumstances.} 
\item{Similar to how \cite{AL17} contains a version of \cite{C09} in Corollary 1.1 above, we can think of a version of \cite{LY15} as essentially being obtainable from Theorem 1.8. Moreover, the scale critical data allows for much more quantitative control on the size and relation of the parameters $a,b,\mu,\kappa, \delta,\lambda,\gamma$ which are involved in the problem. To say it another way, in \cite{C09} (and consequently \cite{LY15}, \cite{LM20}) the data does not permit (to our knowledge) obtaining quantitative control on $\delta$ relative to the other parameters involved. } 
\item{Due to the scale critical data and our use of renormalized quantities, our construction is readily applicable to more singular kinds of initial data.}
\end{itemize}
As for the structure of what follows, the rest of Section 1 describes the connection with Kerr Stability and the Penrose Inequality. In 2 we set up the relevant tools and notation. In 3 we carry out the higher order energy estimates. In 4 we construct the transition region, and in 5 we describe the formation of an apparent horizon.

\subsection{Kerr Stability, Weak Cosmic Censorship, and Final State}
\textbf{Kerr Stability} A rough statement of the conjecture reads as follows. \\ \\ \indent
\textbf{Conjecture}. \textit{Given vacuum initial data - characteristic or Cauchy - that is close, in some sense, to an initial data set for the Kerr spacetime - including part of its exterior and a neighborhood of its event horizon, its maximal Cauchy development will asymptote to a member of the Kerr family.} \\ \\ 
\textbf{Weak Cosmic Censorship}. The conjecture originates in the work of Penrose and a few terms are needed to express it. Given a spacetime $(\mathcal{M},\overline{g})$ that embeds conformally, ${\overline{g}\to \tilde{g}\equiv \Omega^2\overline{g}}$, into a spacetime with boundary, $(\mathcal{M},\overline{g}) \hookrightarrow (\tilde{\mathcal{M}},\tilde{g})$, one can define an event horizon, $\mathcal{M}\cap \partial I^-(\mathcal{J}^+;\tilde{\mathcal{M}})$ where $\mathcal{J}^+$ is the \textit{future conformal boundary} in $(\mathcal{\tilde{M}},\tilde{g})$. Both event horizons and conformal boundaries are well defined in the context of explicit black hole solutions, and the conjecture asserts that these notions are suitable much more generally.\\ \\ \indent
\textbf{Conjecture}. \textit{Given complete, generic vacuum initial data, its maximal Cauchy development admits a complete $\mathcal J^+$.} \\ \\
This conjecture is a fundamental problem in general relativity and the only result available, due to Christodoulou \cite{C91}, \cite{C99}, is for the spherically symmetric scalar field, for which the conjecture is true under a certain rendition of the term \textit{generic}. \\ \\
\textbf{Final State}. If we combine the Weak Cosmic Censorship, Kerr Stability, and Kerr Rigidity\footnote{Kerr Rigidity is the conjecture that the family of Kerr solutions is the only asymptotically flat vacuum black hole stationary solution. It is still open in full, and the most recent result in this direction is due to Alexakis-Ionescu-Klainerman \cite{AKI15}.}, one obtains the following. \\ \\ \indent 
\textbf{Conjecture}. \textit{Given complete, generic, vacuum initial data lacking trapped or marginally trapped surfaces, its maximal Cauchy development admits a complete $\mathcal J^+$, and if a trapped surface forms dynamically, then an event horizon will form $\mathcal{M}\cap \partial I^-(\mathcal{J}^+;\tilde{\mathcal{M}}) \neq \emptyset$, and the region exterior to it asymptotically approaches a member of the Kerr family.}\\ \\  
At present, we do not possess non-spherically symmetric dynamical examples of black hole formation for the Einstein vacuum system. The examples in \cite{LM20}, in virtue of being isometric to Kerr in a neighborhood of future timelike infinity, circumvent Weak Cosmic Censorship and Final State. The results in \cite{LY15} and Theorem 1.8 bring about a connection between Kerr Stability, Weak Cosmic Censorship, and Final State in the spirit of Dafermos-Luk \cite{DL17}, who show that Kerr Stability would imply the failure of the $C^0$ version of Strong Cosmic Censorship. \\ \\
Here the connection follows from the estimates obtained in Section 4. One can set-up a new characteristic initial value problem, where the outgoing one is exactly Kerr and the incoming cone $\ubar=\delta$, truncated at a MOTS locating at $u=\delta a^{1/2}b^\mu(1+o(1))$, satisfies the following estimates. \\ \\ 
For all integers $k\geq0$ we have
	\begin{gather} 
	\lVert u^k \nabla^k \chihat \rVert_{L^2(S_{u,\delta})} \lesssim \frac{\delta a}{u},\\ 
	\twoSudelta{ u^{i+1} \nabla^{i} \csigma}+  \twoSudelta{ u^{i+1} \nabla^{i} \sigma} \lesssim \frac{\delta \al}{u},\\
	\lVert u^{k+1} \nabla^k \beta \rVert_{L^2(S_{u,\delta})} \lesssim \frac{\delta a}{u} , \\ 
	\big\lVert u^{k+1}\nabla^k \left(\tr\chi - \frac2u +\frac{4m_0}{u^2}\right)\big\rVert_{L_{\ubar}^{\infty}L_{u}^{\infty} L^{2}(S_{u,\delta})} \lesssim \frac{\delta a}{u^{\f12}} + \frac{\delta \al}{u},  \\ 
	\big\lVert u  \left(\omega + \frac{m_0}{2u^2} \right)\big\rVert_{L^2(S_{u,\delta})} \lesssim \frac{\delta^{\frac{1}{2}}\al}{u^{\frac{1}{2}}},\\
	\big\lVert u^{k+1} \nabla^{k+1} \left(\omega + \frac{m_0}{2u^2} \right)\big\rVert_{L^2(S_{u,\delta})} \lesssim \frac{\delta^{\frac{1}{2}}\al}{u^{\frac{1}{2}}},\\  \TwoSudelta{u^{k+2} \nabla^k \left(\rho + \frac{2m_0}{u^3}\right)} \lesssim \frac{\dal\al}{u^{\f12}},  \\\twoSudelta{u^{k+1}\nabla^k \alphabar} \lesssim \frac{\dal \al}{u^{\f12}}, \hspace{2mm}  \twoSudelta{u^{k+1}\nabla^k \alpha} \lesssim \frac{\delta a}{u}. 
	\end{gather}
	where these estimates are with reference to their values in a Schwarzschild spacetime of mass $m_0$.\\ \\ 
Expressing the quantities in terms of $a$ 
\begin{equation*}
    b=a^{\kappa}, \hspace{0.2in} \delta=a^{-y} 
\end{equation*}
the initial data conditions $a^{1/2}/b<1$ and $\delta a^{1/2}b^\mu<1$ become
\begin{equation}
    \frac{1}{2}< \kappa < 1, \hspace{0.2in} 1<\mu, \hspace{0.2in} \kappa\mu  -y +\frac{1}{2} < 0
\end{equation}
and the worst estimate on the incoming cone is
\begin{equation}
\lesssim    \frac{\delta^{1/2}a^{1/2}}{u^{1/2}} 
\end{equation}
which for the MOTS located at $u=\delta a^{1/2}b^\mu(1+o(1))$ becomes
\begin{equation}
    \lesssim \frac{\delta^{1/2}a^{1/2}}{b^{\mu/2}a^{1/4}\delta^{1/2}}=a^{\frac{1}{4}-\frac{\kappa \mu}{2}}=a^{-C}
\end{equation}
Seeking to make $C=\frac{\kappa \mu}{2}-\frac{1}{4}$ large, (1.21) restricts us to $C<\frac{1}{2} y-\frac{1}{2}$, and so the estimates for the incoming cone are 
\begin{equation}
\leq C \delta^{1/2} \: \: \text{and} \: \: \leq C\delta, 
\end{equation} 
for a constant $C$ independent of $\delta$. \\ \\
Note that in \cite{LY15} the estimates obtained on the incoming cone $\underline{C}_{\ubar=\delta}$ up to the trapped surface (which lies at $u=u_*$ in the language of Theorem 1.1 above) are for a different set of quantities and are of the form 
\begin{equation}
    \lesssim \frac{\delta^{1/2}}{u^{k>0} }
\end{equation}  
for $k\in \mathbb{N}^+$ with $k$ different for various components. The estimates here are uniform in $u$ and there is an extra $\delta^{1/2}$ for some components. More importantly, the scale critical data of \cite{AL17} permits controlling the size of $\delta$ relative to the other parameters.

\subsection{The Penrose Inequality}
The Penrose inequality is a conjectured inequality between the ADM mass of asymptotically flat initial data sets $(M,g,k)$ and the area of the data set's boundary, $\partial M$, when the latter is taken to be a MOTS, see \cite{M09} for a review.\\ \\
The inequality was motivated by Penrose \cite{P73} which he derived by a heuristic argument representing the ``establishment viewpoint'' on gravitational collapse. His idea was to propose a test for the difficult evolutionary conjectures that were (and still are) far beyond reach, i.e. a counterexample to the inequality would put the conjectures in doubt. Penrose's heuristic was developed and its modern form is as follows. 
\begin{conjecture}
Let $(M,g,k)$ be an initial data set satisfying the dominant energy condition\footnote{The dominant energy condition (trivially satisfied in vacuum) says that $\mu \geq |J|$ where $\mu$ and $J$ appear on the right hand side of (1.3-4).}. Suppose that $\partial M$ has boundary given by a MOTS whose strictly minimizing hull has area $A$. 
Then \begin{equation*}
    m_{ADM}\geq \sqrt{\frac{A}{16\pi}}
\end{equation*} 
with equality if and only if $(M,g,k)$ belongs to the Schwarzschild spacetime.
\end{conjecture}
The special case when $k=0$, known as the Riemannian Penrose Inequality, was proved by Huisken-Ilmanen \cite{HI01} and Bray \cite{B01}. \\ \\ 
In the general case, outside of spherical symmetry, the only known spacetime Penrose-type inequality like Conjecture 1.1 that does not have unwanted constants on the right hand side is the recent result of Alaee-Lesourd-Yau \cite{ALY19}, which holds for a certain class of \textit{admissible} initial data sets. There is also a null Penrose Inequality, where the Bondi mass along $\mathcal{J}^-$ replaces the ADM mass. This inequality is in a sense stronger than the one above and it has been shown to hold by Roesch \cite{R16} in a certain setting. \\ \\
The current construction offers a dynamical test of the Penrose inequality. More precisely, in Section 5 we show that the MOTS $M_{\ubar}$ obtained in Theorem 1.8 in the region $\ubar\in [\gamma \frac{a^{\frac{1}{2}}}{b} \delta, \delta]$ satisfy the following area bounds
\begin{equation}
    (\frac{1}{4}-o(1))b^\mu a^{1/2}\ubar \leq  \sqrt{\frac{|M_{\ubar}|}{16\pi}} \leq (\frac{1}{4}+o(1))b^\mu a^{1/2}\ubar 
\end{equation}
for some $o(1)\ll 1$. By assumption, the initial data satisfies 
\begin{equation}
    b^\mu a^{1/2}\lambda \delta (1+o(1))=4m_0
\end{equation}
for a quantity $o(1)\ll 1$. The gluing procedure described in [CS05] and [LY12] yields
\begin{equation}
    m_{ADM}=m_0\pm \epsilon
\end{equation}
Combining these leads to  
\begin{equation}
    m_{ADM}-\sqrt{\frac{|M_{\ubar}|}{16\pi}} \geq b^\mu a^{1/2}(\frac{1}{4}\pm o(1))\left( \lambda \delta -\ubar \right) \pm \epsilon
\end{equation}
Evaluating the inequality at $\ubar =\gamma \frac{a^{\frac{1}{2}}}{b}\delta$ gives
\begin{equation}
    m_{ADM}-\sqrt{\frac{|M_{\ubar}|}{16\pi}} \geq \frac{1}{4}b^{\mu}a^{1/2}\delta (\lambda -\gamma \frac{a^{\frac{1}{2}}}{b}) \pm \epsilon
\end{equation}
with $\epsilon\lesssim \delta^{1/2} a^{1/2}$. To show that this is positive, we express quantities in terms of $a$ as above to get
\begin{equation}
    \frac{1}{4}b^{\mu}a^{1/2}\delta (\lambda -\gamma \frac{a^{\frac{1}{2}}}{b}) \pm \epsilon= \left(a^{\kappa \mu-\frac{y}{2}}o(1)\pm c_2\right) a^{\frac{1}{2}-\frac{y}{2}}
\end{equation}
where $c_2$ is an unknown constant $>0$ that is independent of $\delta=a^{-y}$. Our initial data requires  $\kappa \mu +\frac{1}{2} <y$, $\frac{1}{2}<\kappa<1$. \\ \\
So upon choosing for any $0<t<\frac{1}{2}$
\begin{equation}
    \kappa \mu +\frac{1}{2} =(\frac{1}{2}+t)y
\end{equation} 
we obtain
\begin{equation}
   m_{ADM}-\sqrt{\frac{|M_{\gamma \delta}|}{16\pi}} \geq (a^{ty-\frac{1}{2}}o(1)\pm c_2)a^{\frac{1}{2}-y}
\end{equation}
This means that, for any $0<t<\frac{1}{2}$, choosing $y$ sufficiently large, we obtain 
\begin{equation}
    m_{ADM}-\sqrt{\frac{|M_{\gamma \delta}|}{16\pi}} >0
\end{equation}
and thus there is an open region in the spacetime for which the Penrose inequality is satisfied. \\ \\
Note that this argument becomes inconclusive when $\ubar$ approaches $\lambda \delta$. In particular, when $\ubar-\lambda \delta \lesssim \delta^\frac{3}{2}$, then (1.38) becomes
\begin{equation}
    (a^{\kappa\mu+\frac{1}{2}-y-\frac{1}{2}}\pm c_2)a^{\frac{1}{2}-\frac{y}{2}}
\end{equation}
but since $\kappa \mu +\frac{1}{2}<y$ the sign of (1.44) now depends on the unknown constant $c_2$. \\ \\
Note that this test of the Penrose Inequality is made possible by \cite{A17} and the data being scale critical. In particular, there are no dynamically obtained MOTS or indeed dynamical horizons in either \cite{LY15} in \cite{LM20}. There are pre-existing MOTS in \cite{LM20} but these are not dynamically formed: they are there by virtue of the solutions being isometric to Kerr in a neighborhood of future timelike infinity, where the Penrose inequality is immediate. \\ \\
Note also that this argument does not permit violating the Penrose inequality since the error in $m-m_0$ is not known.

\subsection*{Acknowledgements} Both authors would like to acknowledge helpful discussions with Professors Xinliang An, Lydia Bieri, Mihalis Dafermos, Igor Rodnianski, Yakov Shlapentokh-Rothman, and Shing-Tung Yau. N.Athanasiou\footnote{nikolaos.athanasiou@maths.ox.ac.uk}. is supported by the Engineering and Physical Sciences Research Council grant [EP/L015811/1] and would like to acknowledge Harvard University's Black Hole Initiative for its hospitality during part of this work. M.Lesourd\footnote{mlesourd@fas.harvard.edu} would like to acknowledge the Gordon Betty Moore Foundation and Harvard University's Black Hole Initiative.

\section{Setting, equations and notations}

In this section, we will introduce the geometric setup and the double null foliation gauge. We then write the Einstein equations as a system of equations for the Ricci coefficients and curvature components adapted to this gauge. After that, we introduce the necessary notations and the norms that we will use.

\subsection{Double null foliation}

The concept of a double null foliation is introduced here. Given a spacetime $(\bar{\mathcal{M}},\bar{g})$ solving the Einstein vacuum equations $R_{\mu \nu}=0,$ we define a double null foliation by solving the eikonal equations

\[ (g^{-1})^{\mu \nu}\hspace{.5mm} \partial_\mu u \hspace{.5mm} \partial_\nu u = 0,   (g^{-1})^{\mu \nu}\hspace{.5mm} \partial_\mu \ubar \hspace{.5mm} \partial_\nu \ubar = 0,   \]for $u$ and $\ubar$ such that $u=1$ on $H_1$ and $\ubar=0$ on $\underline{H}_0$. According to our convention, $\ubar$ is increasing towards the future whilst $u$ is decreasing towards the future. 

\vspace{3mm}

\par \noindent 	Define the future-directed, null geodesic vector fields

\[L^{\prime \mu} = - 2\hspace{.4mm} (g^{-1})^{\mu\nu}\hspace{.4mm} \partial_\nu u, \hspace{2mm} \underline{L}^{\prime \mu} = 2\hspace{.4mm} (g^{-1})^{\mu\nu} \partial_\nu \ubar   \]and define the null lapse function $\Omega$ by 

\[ 2 \Omega^{-2} = -g(L^\prime, \Lbar^\prime).     \]Define the normalized vector fields

\[ e_3 = \Omega \hspace{.4mm}\Lbar^\prime, \hspace{2mm} e_4 = \Omega\hspace{.4mm} L^\prime, \]which satisfy, by construction, \[ g(e_3,e_4) =-2.\] These are the frames that we will use to decompose the Ricci coefficients and the curvature components. Define also the equivariant vector fields

\[      \Lbar= \Omega^2\hspace{.4mm} \Lbar^\prime, \hspace{2mm} L= \Omega^2\hspace{.4mm} L^\prime.\]We fix the gauge on the initial hypersurfaces such that

\[ \Omega=1, \hspace{2mm} \text{on $H_1$ and $\underline{H}_0$.}     \]\vspace{3mm}

\par \noindent Denote by $H_u$ the level sets of $u$ and by $\Hbar_u$ the level sets of $\ubar$. By the eikonal equations, $H_u$ and $\Hbar_u$ are in fact null hypersurfaces. The topology of the intersections $H_u \cap \Hbar_{\ubar}$ is that of a $2-$sphere. Denote these $2-$spheres by $S_{u,\ubar}$.

\subsection{The coordinate system} We define a coordinate system $(u,\ubar,\theta^1, \theta^2)$ in the spacetime. On the standard sphere $S_{1,0}$, define a coordinate system $(\theta^1, \theta^2)$ such that on each coordinate patch the metric $\g$ is smooth, bounded and positive definite. We then define the coordinates on the initial hypersurfaces by requiring $\theta^A$ to be constant along null generators of the initial hypersurface. In the spacetime, we define $u$ and $\ubar$ to be solutions to the eikonal equations, as described in the previous subsection. Moreover, we naturally extend $\theta^1, \theta^2$ by 

\[ \slashed{\mathcal{L}}_L \theta^A = 0,       \]where $\slashed{\mathcal{L}}$ denote the restriction of the Lie derivative to $TS_{u,\ubar}$. Relative to the coordinate system $(u,\ubar, \theta^1,\theta^2)$, the vector fields $e_3$ and $e_4$ can be written as

\[ e_3 = \Omega^{-1}\left( -\frac{\partial}{\partial u} + d^A \frac{\partial}{\partial \theta^A}\right), \hspace{2mm} e_4= \Omega^{-1}\frac{\partial}{\partial \ubar},    \]for some functions $d^A$ such that $d^A=0$ on $\Hbar_0$ and the metric $g$ takes the form

\[ g= -2\Omega^2 (du \otimes d\ubar + d\ubar \otimes du)+ \gamma_{AB}(d\theta^A - d^A du)\otimes(d\theta^B- d^B du).      \]

\subsection{The Einstein vacuum equations relative to a double null foliation}

We provide a decomposition of the Ricci coefficients and the null curvature components with respect to a null frame $e_3, e_4$ defined above and a frame $e_1, e_2$ tangent to the spheres $S_{u,\ubar}$. \vspace{3mm} Using the indices $A,B$ taking values in the set $\begin{Bmatrix} 1,2\end{Bmatrix}$, we define the Ricci coefficients relative to the null frame:

\begin{equation}
\label{2.1}
\begin{split}\
\chi_{AB} = g(D_A e_4, e_B),\hspace{2mm} \chibar_{AB} =g(D_A e_3, e_B), \\  \eta_A = -\frac{1}{2}\hspace{.5mm}g(D_3 e_A , e_4), \hspace{2mm} \etabar_A = -\frac{1}{2}\hspace{.5mm}g(D_4 e_A, e_3),  \\ \omega = - \frac{1}{4} \hspace{.5mm}g(D_4 e_3, e_4), \hspace{2mm} \omegabar = -\frac{1}{4} \hspace{.5mm}g(D_3 e_4, e_3), \\\zeta_A = \frac{1}{2}\hspace{.5mm} g(D_A e_4, e_3).
\end{split}
\end{equation}

\vspace{3mm}

\begin{equation}
\label{2.2}
\begin{split}
\alpha_{AB} = W(e_A, e_4, e_B, e_4), \hspace{2mm} \alphabar_{AB} = W(e_A, e_3, e_3), \\ \beta_A= \frac{1}{2}\hspace{.5mm} W(e_A, e_4, e_B, e_4), \hspace{2mm} \betabar_A = \frac{1}{2}\hspace{.5mm} W(e_A, e_3, e_B, e_3), \\ \rho= \frac{1}{4} \hspace{.5mm}W_(e_3, e_4,e_3,e_4), \hspace{2mm} \sigma = \frac{1}{4} \hspace{.5mm}\Hodge{W}(e_3,e_4,e_3,e_4).
\end{split}	 
\end{equation}Here $\Hodge{W}$ denotes the the Hodge dual of $W$, the Weyl curvature tensor. Let $\nabla$ be the induced covariant derivative operator on $S_{u,\ubar}$ and $\nabla_3, \nabla_4$ be the projections to $S_{u,\ubar}$ of the covariant derivatives $D_3,\hspace{.5mm}D_4$. We note the following useful identities regarding the Ricci coefficients:

\begin{equation} \label{2.3}
\begin{split}
\omega = -\frac{1}{2}\hspace{.5mm} \nabla_4 (\text{log} \hspace{.5mm} \Omega), \hspace{2mm} \omegabar = - \frac{1}{2}\hspace{.5mm} \nabla_3(\text{log} \hspace{.5mm} \Omega),          \\ \eta_A = \zeta_A + \nabla_A(\text{log}\Omega), \hspace{2mm} \etabar_A = -\zeta_A+ \nabla_A(log\Omega).
\end{split}
\end{equation}\\ Define the following contractions of the tensor product of $\phi^{(1)}$ and $\phi^{(2)}$ with respect to the metric $\gamma_{AB}$. For symmetric $2-$tensors $\phi_{AB}^{(1)}, \hspace{0.5mm} \phi_{AB}^{(2)},$ define 

\be \begin{split} \label{2.4} \phi^{(1)} \cdot \phi^{(2)} := (\gamma^{-1})^{AC} (\gamma^{-1})^{BD} \phi^{(1)}_{AB}\hspace{.5mm} \phi^{(2)}_{CD}, \\ \phi^{(1)} \wedge \phi^{(2)} := \slashed{\epsilon}^{AB}(\gamma^{-1})^{CD} \phi^{(1)}_{AB} \hspace{.5mm} \phi^{(2)}_{CD},  \end{split} \ee where $\slashed{\epsilon}$ is the volume form of the metric $\gamma$. Continuing our introduction of notation,  define for two 1-forms $\phi_A^{(1)}, \phi^{(2)}_A$, the following:

\be \begin{split}         
\label{2.5}        
\phi^{(1)} \cdot \phi^{(2)} := (\gamma^{-1})^{AB} \phi_A^{(1)} \phi_B^{(2)}, \\ \phi^{(1)}\wedge \phi^{(2)} := \slashed{\epsilon}_{AB} \phi_A^{(1)}\phi_B^{(2)}, \\ \big(\phi^{(1)}\hspace{.5mm} \widehat{\otimes} \hspace{.5mm} \phi^{(2)}\big)_{AB}:= \phi_A^{(1)} \phi_B^{(2)} +\phi_B^{(1)} \phi_A^{(2)} - \gamma_{AB} (\phi^{(1)} \cdot \phi^{(2)}).
\end{split}                        
\ee For a symmetric $2-$tensor $\phi_{AB}^{(1)}$ and a $1-$form $\phi_A^{(2)}$, define the contraction \be \label{2.6} \left( \phi^{(1)} \cdot \phi^{(2)}\right)_A := (\gamma^{-1})^{BC}\hspace{.5mm} \phi_{AB}^{(1)}\hspace{.5mm} \phi_C^{(2)}. \ee We also define the operation $^*$ for $1-$forms $\phi^{(1)}$ and symmetric $2-$tensors  $\phi^{(2)}$ respectively as follows:

\be  \label{2.7} \begin{split} \Hodge{\phi}^{(1)}_A := \gamma_{AC}\hspace{.5mm} \slashed{\epsilon}^{CB} \hspace{.5mm} \phi^{(1)}_B, \\ \Hodge{\phi}_{AB}^{(2)} := \gamma_{BD} \hspace{.5mm} \slashed{\epsilon}^{DC} \hspace{.5mm} \phi^{(2)}_{AC}.  \end{split}      \ee For totally symmetric tensors, the divergence and curl operators are defined by the
formulae

\be \label{2.8}\begin{split} ( \text{div} \hspace{.5mm} \phi)_{A_1 \dots A_r}  := \nabla^B \phi_{BA_1\dots A_r}, \\ (\text{curl} \hspace{.5mm} \phi)_{A_1 \dots A_r} :=  \slashed{\epsilon}^{BC} \nabla_B \phi_{CA_1\dots A_r}.  \end{split}    \ee Define the operator $\nabla \widehat{\otimes}$ on a $1-$form $\phi_A$ by

\be \label{2.9}  (\nabla \widehat{\otimes} \phi) := \nabla_A \phi_B + \nabla_B \phi_A -\gamma_{AB} \hspace{.5mm} \text{div} \phi.        \ee Finally, define the trace to be

\be (tr \hspace{.5mm} \phi)_{A_1\dots A_{r-1}} := (\gamma^{-1})^{BC} \phi_{BCA_1 \dots A_{r-1}}. \ee

\vspace{3mm}

\par \noindent Let $\chihat, \chibarhat$ be the traceless parts of $\chi$ and $\chibar$ respectively, so that

\[  \chi = \chihat + \frac{1}{2} tr\chi \gamma, \hspace{2mm} \chibar = \chibarhat + \frac{1}{2} tr\chibar \gamma.  \]	The components $\chi$ and $\chihat$ obey the following transport equations:

\begin{gather}
\nabla_4 \hspace{.5mm} tr\chi + \frac{1}{2}\hspace{.5mm}  (tr\chi)^2 = -\lvert \chihat \rvert^2- 2 \omega \hspace{.5mm} tr\chi \label{trchieq},\\ \label{chihateq}
\nabla_4\hspace{.5mm}  \chihat + tr\chi \hspace{.5mm} \chihat = -2\hspace{.5mm} \omega \hspace{.5mm} \chihat -\alpha,\\ \nabla_3\hspace{.5mm}  tr\chibar + \frac{1}{2}\hspace{.5mm}  (tr\chibar)^2 = -\lvert \chibarhat \rvert^2 -2\hspace{.5mm}  \omegabar \hspace{.5mm} tr\chibar, \\ \nabla_3 \hspace{.5mm}  \chibarhat + tr\chibar \hspace{.5mm} \chibarhat = -2\omegabar \hspace{.5mm} \chibarhat- \alphabar, \\ \nabla_4 \hspace{.5mm} tr\chibar + \frac{1}{2}\hspace{.5mm} tr\chi \hspace{.5mm} tr\chibar = 2\hspace{.5mm} \omega \hspace{.5mm} tr\chibar + 2\hspace{.5mm} \rho- \chihat\cdot \chibarhat +2 \hspace{.5mm} \text{div}\hspace{.5mm} \etabar +2 \lvert \etabar \rvert^2,\\ \nabla_4 \chibarhat + \frac{1}{2}\hspace{.5mm}  tr\chi\hspace{.5mm}  \chibarhat = \nabla \widehat{\otimes} \etabar + 2\hspace{.5mm} \omega \hspace{.5mm} \chibarhat -\frac{1}{2}\hspace{.5mm} tr\chibar \hspace{.5mm} \chihat +\etabar \widehat{\otimes} \etabar,\\ \nabla_3 tr\chi + \frac{1}{2}\hspace{.5mm} tr\chi \hspace{.5mm} tr\chibar = 2\hspace{.5mm} \omegabar \hspace{.5mm} tr\chi +2 \hspace{.5mm} \rho - \chihat\cdot \chibarhat + 2 \hspace{.5mm} \text{div}\hspace{.5mm}  \eta + 2\hspace{.5mm} \lvert \eta \rvert^2,\\ \nabla_3 \chihat + \frac{1}{2} tr\chibar \hspace{.5mm}\chihat = \nabla \widehat{\otimes} \eta + 2\hspace{.5mm} \omegabar\hspace{.5mm} \chihat - \frac{1}{2}\hspace{.5mm} tr\chibar \hspace{.5mm} \chihat + \eta \widehat{\otimes} \eta. \label{chihat3eq}
\end{gather}The other components satisfy the following transport equations:

\begin{gather}
\nabla_4 \eta = - \chi \cdot (\eta - \etabar) - \beta, \\ \nabla_3 \etabar = -\chibar \cdot (\etabar - \eta) - \betabar, \\ \nabla_4 \omegabar = 2\omega \omegabar -\eta\cdot \etabar + \frac{1}{2} \lvert \eta \rvert^2 +\frac{1}{2} \rho, \\ \nabla_3 \omega = 2\omega \omegabar -\eta\cdot \etabar + \frac{1}{2} \lvert \etabar \rvert^2 +\frac{1}{2} \rho,
\end{gather}as well as the constraint equations        

\begin{gather} 
\text{div}\hspace{.5mm}\chihat =  \frac{1}{2} \nabla tr\chi - \frac{1}{2}(\eta - \etabar) \cdot (\chihat - \frac{1}{2} tr\chi) -\beta, \\\text{div}\hspace{.5mm}\chibarhat =  \frac{1}{2} \nabla tr\chibar + \frac{1}{2}(\eta - \etabar) \cdot (\chibarhat - \frac{1}{2} tr\chibar) + \betabar,\\ \text{curl} \hspace{.5mm} \eta = - \text{curl}\hspace{.5mm} \etabar = \sigma + \frac{1}{2} \chibarhat \wedge \chihat, \\ K = -\rho- \frac{1}{4} tr\chi\hspace{.5mm} tr\chibar +\frac{1}{2} \hspace{.5mm}\chihat \cdot \chibarhat.
\end{gather}Here $K$ is the Gauss curvature of the spheres $S_{u,\ubar}$. The null curvature components satisfy the null Bianchi equations

\begin{gather}
\nabla_4 \beta + 2tr\chi \beta = \text{div} \alpha - 2\omega \beta +\eta\cdot \alpha, \\ \nabla_3 \beta + tr\chibar \beta = \nabla \rho + \prescript{*}{}{\nabla \sigma} +2\chihat\cdot \beta+2\omegabar \beta +3(\eta \rho + \prescript{*}{}{\eta}\sigma),\\ \nabla_4 \betabar +tr\chi \betabar = -\nabla \rho + \prescript{*}{}{\nabla \sigma} +2\chibarhat\cdot \beta + 2\omega \betabar - 3(\etabar \rho - \Hodge{\etabar}\sigma),\\\nabla_3 \betabar +2tr\chibar \hspace{.5mm}\betabar= -\text{div}\alphabar  -2 \omegabar \hspace{.5mm}\betabar +\etabar \alphabar ,\\ \nabla_4 \alphabar +\frac{1}{2}tr\chi \alphabar = -\nabla\widehat{\otimes} \betabar +4\omega \alphabar -3(\chibarhat\rho - \Hodge{\chibarhat}\sigma)+(\zeta - 4\etabar)\widehat{\otimes}\betabar, \\\nabla_3 \alpha + \frac{1}{2}tr\chibar \alpha = \nabla \widehat{\otimes}\beta+ 4 \omegabar \alpha -3(\chihat \rho +\Hodge{\chihat}\sigma)+(\zeta + 4 \eta) \widehat{\otimes} \beta,\\ \nabla_4 \rho + \frac{3}{2} tr\chi \rho = \text{div} \beta - \frac{1}{2}\chibarhat \cdot \alpha +\zeta\cdot  \beta + 2\etabar \cdot \beta,\\ \nabla_3 \rho + \frac{3}{2}tr\chibar \rho = -\text{div} \betabar - \frac{1}{2}\chihat\cdot\alphabar+ \zeta\cdot \betabar - 2\eta\cdot \betabar,\\ \nabla_4 \sigma +\frac{3}{2}tr\chi \sigma = -\text{div}\Hodge{\beta} + \frac{1}{2}\hspace{.5mm}\hat{\chibar}\cdot \Hodge{\alpha}- \zeta \cdot \Hodge{\beta} -2 \etabar \cdot \Hodge{\beta} , \\ \nabla_3 \sigma + \frac{3}{2} tr\chibar \sigma  = -\text{div} \Hodge{\betabar} + \frac{1}{2} \chihat \cdot \Hodge{\alphabar} -\zeta \cdot \Hodge{\betabar} - 2\eta \cdot \Hodge{\betabar}.\end{gather}
Defining $\check{\sigma}=\sigma+\frac{1}{2}\hat{\chibar}\wedge \hat{\chi}$, the null Bianchi equations (2.27)-(2.36) can be re-written so as to replace $\rho$, $\sigma$, $\alpha$ with $K$ and $\check{\sigma}$, yielding the following null structure equations for $\chi$ and $\chibar$:
\begin{gather*}
\nabla_3 \beta + \tr \chibar \beta = -\nabla K+ \Hodge{\nabla} \check{\sigma}+2
\omegabar \beta-3(\eta K-\Hodge{\eta} \check{\sigma})+\frac{1}{2}(\nabla(\chihat  \cdot \chibarhat)\\+\Hodge{\nabla}(\hat{\chi}\wedge \chibarhat))-\frac{3}{4}\eta \tr \chi \tr\chibar+ \frac{3}{2}(\eta \chihat \cdot \chibarhat+\Hodge{\eta} \chihat \wedge \chibarhat)-\frac{1}{4}(\nabla \tr \chi \tr \chibar + \tr \chi \nabla \tr \chibar), \\
\nabla_4 \check{\sigma} +\frac{3}{2}\tr{\chi}\check{\sigma} =-\div \Hodge{\beta} -\zeta \wedge \beta -2\eta \wedge \beta -\frac{1}{2}\chihat \wedge (\nabla \widehat{\otimes} \etabar)-\frac{1}{2}\chihat \wedge (\etabar \widehat{\otimes} \etabar)\\ 
\nabla_4 K+\tr{\chi}K=-\div \beta-\zeta\cdot \beta -2\etabar\cdot \beta+\frac{1}{2}\chihat\cdot \nabla \widehat{\otimes} \etabar +\frac{1}{2}\chihat\cdot(\etabar \widehat{\otimes} \etabar)\\ 
-\frac{1}{2}\tr{\chi}\div{\etabar} -\frac{1}{2}\tr{\chi}|\etabar|^2,\\
\nabla_3 \check{\sigma}+\frac{3}{2}\tr{\chibar}\check{\sigma}=-\div \Hodge{\betabar}+\zeta\wedge\betabar -2\etabar\wedge \betabar +\frac{1}{2}\chibarhat \wedge(\nabla\widehat{\otimes}\eta) +\frac{1}{2}\chibarhat \wedge (\eta \widehat{\otimes} \eta),\\
\nabla_3 K+\tr{\chibar}K=\div \betabar -\zeta\cdot \betabar +2\eta\cdot \betabar +\frac{1}{2}\chibarhat \cdot \nabla \widehat{\otimes} \eta +\frac{1}{2}\chibarhat \cdot (\eta \widehat{\otimes} \eta) -\frac{1}{2}\tr \chibarhat \div \eta -\frac{1}{2}\tr{\chibar}|\eta|^2,\\
\nabla_4 \betabar+\tr{\chi}\betabar =\nabla K+\Hodge{\nabla \check{\sigma}} +2\omega\betabar +2\chibarhat \cdot \beta+3(-\etabar K+\Hodge{\etabar}\check{\sigma})-\frac{1}{2}(\nabla(\chihat\cdot \chibarhat) -\Hodge{\nabla}(\chihat \wedge \chihat)) \\ 
+\frac{1}{4}(\nabla \tr{\chi}\tr{\chibar}+\tr{\chi}\nabla \tr{\chibar})-\frac{3}{2}(\etabar \chihat \cdot \chihat -\Hodge{\etabar} \chihat \wedge \chibarhat)
+\frac{3}{4}\etabar\tr{\chi}\tr{\chibar}.
\end{gather*}

\subsection{The norms}

Fix a positive integer $N \geq 4$. First we define the norms for the curvature components:

\begin{multline}
\mathcal{R}_N = \sum_{i \leq N} \Bigg( \sup_u  \Big( \frac{1}{\delta^{\frac{1}{2}} a^{\frac{1}{2}}}  \lVert u^{i+1}\nabla^i \beta \rVert_{L^2\left(H_{u}\right)}  \Big)    \\ +\sup_{\ubar} \Big( \frac{1}{\delta^{\frac{1}{2}}a^{\frac{1}{2}}} \Big\lVert u^{i+1}\nabla^i   \left(K- \frac{1}{\lvert u \rvert^2}, \check{\sigma} \right)\Big\rVert _{L^2(\Hbar_{\ubar})}\Big)\Bigg)\\ + \sum_{1\leq i \leq N} \Bigg(   \sup_u \Big( \frac{1}{\delta^{\frac{3}{2}} a^{\frac{3}{4}}} \Big\lVert u^{i+2}\nabla^i   \left(K- \frac{1}{\lvert u \rvert^2}, \check{\sigma} \right)\Big\rVert_{L^2(H_u)} \Big)                    \\  +\sup_{\ubar} \Big(  \frac{1}{\delta^{\frac{3}{2}} a^{\frac{3}{4}}}\lVert u^{i+2} \nabla^i \betabar \rVert_{L^2(\Hbar_{\ubar})}   \Big)  \Bigg)\\ + \sup_u \Big( \frac{1}{\delta^{\frac{3}{2}} a^{\frac{3}{4}}} \Big\lVert u^2\Big(K- \frac{1}{\lvert u \rvert^2}\Big) \Big\rVert_{L^2(H_u)} \Big). 
\end{multline}We then define the norms for the Ricci coefficients. We begin with those for the highest order derivatives:

\begin{multline}  \widetilde{\mathcal{O}}_{N+1,2}  = \sup_u \Big( \frac{1}{\delta^{\frac{1}{2}} a^{\frac{1}{2}} } \Hutwo{u^{N+1} \nabla^{N+1} (\chihat, tr\chi, \omega) }\Big) +\sup_{\ubar} \Big(\frac{1}{\delta^{\frac{1}{2}} a^{\frac{1}{2}} } \Hbarubartwo{u^{N+1} \nabla^{N+1} \eta}  \Big)\\ + \sup_u \Big( \frac{1}{\delta^{\frac{3}{2}} a^{\frac{3}{4}} } \Hutwo{u^{N+2} \nabla^{N+1} (\eta,\etabar) }\Big) \\ + \sup_{u,\ubar} \Bigg( \frac{\lvert u \rvert^{\frac{1}{2}}}{\delta\hspace{.5mm} a^{\frac{1}{2}} } \Hbarubartwo{u^{N+1} \nabla^{N+1} (\chibarhat, tr\chibar, \omegabar) }\Bigg).  \end{multline}For $i \leq N$ we define the following $L^2-$norms:

\[ \mathcal{O}_{i,2} = \sup_{u,\hspace{.5mm} \ubar}\left( \frac{1}{a^{\frac{1}{2}}} \lVert u^i \nabla^i (\chihat, \omega, tr\chi) \rVert_{L^2(S_{u,\hspace{.5mm}\ubar})}+ \frac{\lvert u \rvert}{\delta a^{\frac{1}{2}}} \Big \lVert  u^i \nabla^i \left( \eta, \etabar, \nabla \hspace{.5mm} \text{log} \Omega, \chibarhat, tr\chibar + \frac{2}{u}, \omegabar \right)           \Big\rVert_{L^2(S_{u,\ubar})} \right), \]while for $i\leq N-2$ we define the following $L^\infty-$norms: 
\[    \mathcal{O}_{i,\infty} = \sup_{u,\hspace{.5mm} \ubar}  \left( \frac{\lvert u \rvert}{a^{\frac{1}{2}}} \lVert u^i \nabla^i (\chihat, \omega, tr\chi) \rVert_{L^{\infty}(S_{u,\hspace{.5mm}\ubar})}+ \frac{\lvert u^2 \rvert}{\delta a^{\frac{1}{2}}} \Big \lVert  u^i \nabla^i \left( \eta, \etabar, \chibarhat, tr\chibar + \frac{2}{u}, \omegabar \right)           \Big\rVert_{L^{\infty}(S_{u,\ubar})} \right) . \]As a shorthand, we denote \[ \mathcal{O}_N= \sum_{i \leq N-2}\mathcal{O}_{i,\infty} + \sum_{i \leq N}\mathcal{O}_{i,2}.\]

\subsection{The theorem of An-Luk}
Based on the norms defined above, An-Luk were able to show that if $N=4$ in the section above we have the following:

\begin{theorem}
Consider the following characteristic initial value problem
for the Einstein vacuum equations. The initial incoming hypersurface $\Hbar_0$
is required to coincide with a backwards light cone in Minkowski space with $0 \leq u \leq 1$.  On the initial outgoing hypersurface $H_1$, the initial data are smooth and the shear $\chihat_0$ satisfies the following bounds \[ \sum_{i \leq 7} \lVert \nabla^i \chihat_0 \rVert_{L_{\ubar}^\infty L^2(S_{u,\ubar})} \leq \al,      \]uniformly in $0\leq \ubar \leq \delta$. Then there exists a universal large constant $b_0$ and two large numbers $b,a$ such that if $b_0 \leq b \leq a$ and $\delta \al b <1$, the unique solution to the Einstein vacuum equations exists and obeys the following estimates in the region $\delta \al \beta \leq u \leq 1, 0 \leq \ubar \leq \delta:$

\[ \mathcal{O} +\tilde{\mathcal{O}}_{5,2}+\mathcal{R} \lesssim 1.     \]Here the implicit constant is independent of $a,b$ and $\delta$.
\end{theorem}

\section{Higher order energy estimates}\label{section3}
The goal of this section is to show that the desired higher order energy estimates go through using the hierarchy of \cite{AL17}. More precisely, we will establish the following.

\begin{theorem}[Higher order energy estimates]
Fix a natural number $N \geq 4$ and consider the following characteristic initial value problem
for the Einstein vacuum equations. The initial incoming hypersurface $\Hbar_0$
is required to coincide with a backwards light cone in Minkowski space with $0 \leq u \leq 1$.  On the initial outgoing hypersurface $H_1$, the initial data are smooth and the shear $\chihat_0$ satisfies the following bounds \[ \sum_{i \leq N+3} \lVert \nabla^i \chihat_0 \rVert_{L_{\ubar}^\infty L^2(S_{u,\ubar})} \leq \al,      \]uniformly in $0\leq \ubar \leq \delta$. Then there exists a universal large constant $b_0$ and two large numbers $b,a$ such that if $b_0 \leq b \leq a$ and $\delta \al b <1$, the unique solution to the Einstein vacuum equations exists and obeys the following estimates in the region $\delta \al b \leq u \leq 1, 0 \leq \ubar \leq \delta:$

\[ \mathcal{O}_N + \tilde{\mathcal{O}}_{N+1,2} + \mathcal{R}_N \lesssim 1    \]
\end{theorem}
The proof of the theorem is carried out in Section 3 and can be summarized by the following observation: estimates for higher orders of derivatives in the transport equations are done using the commutation formula to obtain transport equations for $\nabla^i \psi$. A similar philosophy is followed in the elliptic and the energy estimates. The point here is that the commutation formula holds for arbitrary natural numbers, i.e. for any number of angular derivatives.

\subsection{The extended bootstrap assumptions}

In analogy to the lower order estimates, the higher order estimates require similar bootstrap assumptions on more derivatives.  We make the following bootstrap assumptions on the first $N$ derivatives
of the Ricci coefficients:

\begin{gather}
\sum_{i \leq N}  \frac{1}{\delta \al} \twoSu{u^{i+1} \nabla^i \underline{\psi}} + \sum_{i \leq N-2}  \frac{1}{\delta \al}\inftySu{u^{i+2} \nabla^i \underline{\psi}} \leq b^{\frac{1}{4}}, \label{bootstrap1}\\ \sum_{i \leq N}  \frac{1}{\al} \twoSu{u^{i} \nabla^i \psi} + \sum_{i \leq N-2}  \frac{1}{ \al}\inftySu{u^{i+1} \nabla^i \psi} \leq b^{\frac{1}{4}}. \label{bootstrap2}
\end{gather}
We also make the following bootstrap assumptions on the top order norms of the Ricci coefficients and the curvature norms
\begin{gather}
\widetilde{\mathcal{O}}_{N+1,2} + \mathcal{R}_N \leq b^{\frac{1}{4}}. \label{bootstrap3}
\end{gather}
We make the final bootstrap assumption, useful when developing elliptic estimates:

\begin{equation}
\sum_{i \leq N-1} \Big\lVert u^{i+1} \nabla^i \left(K- \frac{1}{\lvert u \rvert^2} \right) \Big\rVert_{L_u^\infty L_{\ubar}^\infty L^2(S_{u,\ubar})}\leq 1. \label{bootstrap4}
\end{equation}

\par \noindent In particular, using the same arguments as in \cite{AL17}, we can give the following improvements on \eqref{bootstrap1}-\eqref{bootstrap4}:

\begin{gather}
\sum_{i \leq N}  \frac{1}{\delta \al} \twoSu{u^{i+1} \nabla^i \underline{\psi}} + \sum_{i \leq N-2}  \frac{1}{\delta \al}\inftySu{u^{i+2} \nabla^i \underline{\psi}} \leq 1, \label{bootstrap1improved}\\ \sum_{i \leq N}  \frac{1}{\al} \twoSu{u^{i} \nabla^i \psi} + \sum_{i \leq N-2}  \frac{1}{ \al}\inftySu{u^{i+1} \nabla^i \psi} \leq 1. \label{bootstrap2improved}\\\widetilde{\mathcal{O}}_{N+1,2} + \mathcal{R} \leq 1. \label{bootstrap3improved} \\ \sum_{i \leq N-1} \Big\lVert u^{i+1} \nabla^i \left(K- \frac{1}{\lvert u \rvert^2} \right) \Big\rVert_{L_u^\infty L_{\ubar}^\infty L^2(S_{u,\ubar})}\leq \frac{1}{b^{\frac{3}{4}}}. \label{bootstrap4improved}
\end{gather}We will rigorously argue, in the remainder of this section, that these higher order bootstrap assumptions can be improved following the same philosophy/approach as in \cite{AL17}.

\subsection{Metric Components}
The following straightforward estimates hold, cf. Section 5.1 of \cite{AL17}.
\begin{proposition}
Under the assumptions of Theorem 2.1 and the bootstrap assumptions (3.1)-(3.4), we have \[||\Omega^{-1}-1||_{L^\infty(S_{u,\ubar})} \lesssim \frac{\delta a^{\frac{1}{2}}b^{\frac{1}{4}}}{|u|}\]
\end{proposition}
\begin{proposition}
Under the assumptions of Theorem 2.1 and the bootstrap assumptions (3.1)-(3.4), we have 
\[\sup_{u,\ubar} |\text{Area}(S_{u,\ubar})-\text{Area}(S_{u,0})|\lesssim \frac{\delta a^{\frac{1}{2}}b^{\frac{1}{4}}}{|u|} \]
\end{proposition}

\subsection{Transport Equations}
Next, we introduce two propositions from \cite{AL17} that will be useful in controlling transport equations.

\begin{proposition}\label{transportprop1}
Let $\phi$ be an $S_{u,\ubar}-$tangent tensor of arbitrary rank. Under the assumptions of Theorem \ref{main2} and the bootstrap assumptions \eqref{bootstrap1}-\eqref{bootstrap4} there holds
\[ \twoSu{\phi} \lesssim \lVert \phi \rVert_{L^2(S_{u,\ubar^{\prime}})} + \int_{\ubar^\prime}^{\ubar} \lVert \nabla_4 \phi \rVert_{L^{\infty}(S_{u,\ubar^{\prime\prime}})} \hspace{.5mm} \text{d}\ubar^{\prime\prime}.    \]
\end{proposition}

\begin{proposition}\label{transportprop2}
Under the same assumptions as in Proposition \ref{transportprop1}, let $\phi$ and $F$ be $S_{u,\ubar}$-tangent tensor fields of rank $k$ satisfying the following transport equation:
\[ \nabla_{3}\phi_{A_1\dots A_k} + \lambda_0 \tr\chibar \phi_{A_1\dots A_k} = F_{A_1\dots A_k}.     \]Define $\lambda_1:=2\lambda_0-1$. Then there holds

\[  u^{\lambda_1} \twoSu{\phi} \lesssim \lVert \phi \rVert_{L^2(S_{1,\ubar})}+ \int_{u}^{1} \lvert u^\prime \rvert^{\lambda_1} \lVert F \rVert_{L^2(S_{u^\prime, \ubar})} \hspace{.5mm}\text{d}u^\prime.   \]
\end{proposition}

\subsection{Sobolev Embedding}
We will also need the following Sobolev Embedding statements, cf. Sec 5.3 of \cite{AL17}.\\ \\
Define the isoperimetric constant 
\[ I(S)=\sup_{U, \partial U\in C^1} \frac{\min\{\text{Area}(U),\text{Area}(U^c)\}}{(\text{Perimeter}(\partial U))^2}\]
where $S$ is one of the $2$-spheres $S_{u,\ubar}$ adapted to the double null foliation. 
\begin{proposition}[Lemma 5.6 of \cite{AL17}]
Under the assumptions of Theorem 2.1 and the  bootstrap assumptions (3.1)-(3.8), the isoperimetric constant obeys the bound \[I(S_{u,\ubar})\leq \frac{1}{\pi}\]
for $0\leq \ubar \leq \delta$ and $\delta a^{1/2}b\leq u\leq 1$.
\end{proposition}
We now quote two Sobolev embedding theorems from \cite{C09}.
\begin{proposition}[Lemma 5.1 \cite{C09}]
For any Riemannian $2$-manifold $(S,\gamma)$, we have the estimate 
\[(\emph{Area}(S))^{-\frac{1}{p}}||\phi||_{L^p(S)}\leq C_p \sqrt{\max\{I(S),1\}} \left( ||\nabla \phi||_{L^2(S)}+(\emph{Area}(S))^{-\frac{1}{2}} ||\phi||_{L^2(S)} \right)\]
for any $2<p<\infty$ and for any tensor $\phi$.
\end{proposition}
\begin{proposition}[Lemma 5.2 \cite{C09}]
For any Riemannian $2$-manifold $(S,\gamma)$, we have the estimate 
\[||\phi||_{L^\infty(S)}\leq C_p \sqrt{\max\{I(S),1\}} \times \emph{Area}(S)
^{\frac{1}{2}-\frac{1}{p}} \left( ||\nabla \phi||_{L^p(S)}+(\emph{Area}(S))^{-\frac{1}{2}} ||\phi||_{L^p(S)} \right)\]
for any $2<p$ and for any tensor $\phi$.
\end{proposition}
Combining Propositions 3.1,3.2,3.3, using preliminary estimates for metric components $\text{Area}(S_{u,\ubar})\sim |u|^2$, cf. Prop 5.3 of \cite{AL17}, one obtains the following.
\begin{proposition}
Under the assumptions of Theorem 2.1 and the the bootstrap assumptions (3.1)-(3.8), we have 
\[||\phi||_{L^{\infty}}(S_{u,\ubar})\lesssim \sum_{i\leq 2} ||u^{i-1}\nabla^i\phi||_{L^2(S_{u,\ubar})}\]
\end{proposition}

\subsection{Commutation formula}
We will make repeated use of the following commutation formulae.

\begin{proposition}
For a scalar function $f$, there holds 
\[ [\nabla_4, \nabla]f = \frac{1}{2}(\eta+ \etabar)\cdot \nabla_4 f - \chi \cdot \nabla f, \]\[  [\nabla_3, \nabla] f = \frac{1}{2}(\eta+\etabar) \nabla_3 f - \chibar \cdot \nabla f.      \]
\end{proposition}
\begin{proposition}
For an $S_{u,\ubar}-$tangent $1-$form $U_b$, there holds
\[   [\nabla_4, \nabla_a] U_b = -\chi_{ac} \nabla_c U_b +\epsilon_{ac}\Hodge{\beta}_b U_c + \frac{1}{2}(\eta_a + \etabar_a)\nabla_4 U_b -\chi_{ac}\etabar_b U_c + \chi_{ab} \etabar\cdot U,   \]\[  [\nabla_3, \nabla_a] U_b = -\chibar_{ac} \nabla_c U_b +\epsilon_{ac}\Hodge{\betabar}_b U_c + \frac{1}{2}(\eta_a + \etabar_a)\nabla_3 U_b -\chi_{ac}\etabar_b U_c + \chibar_{ab} \eta\cdot U.     \]
\end{proposition}
\begin{proposition}
For an $S_{u,\ubar}-$tangent $2-$form $V_{bc}$, there holds
\begin{multline*}
[\nabla_4, \nabla_a] V_{bc} =   \frac{1}{2}(\eta_a + \etabar_a)\nabla_4 V_{bc}- \etabar_b V_{dc} \chi_{ad} -\etabar_c V_{bd}\chi_{ad} - \epsilon_{bd}\Hodge{\beta}_a       V_{dc} -\epsilon_{cd} \Hodge{\beta}_c V_{bd}\\ +\chi_{ac}V_{bd} \etabar_{d} + \chi_{ab} V_{dc} \etabar_{d} - \chi_{ad} \nabla_d V_{bc},
\end{multline*}\begin{multline*}
[\nabla_3, \nabla_a] V_{bc} =   \frac{1}{2}(\eta_a + \etabar_a)\nabla_3 V_{bc}- \eta_b V_{dc} \chibar_{ad} -\eta_c V_{bd}\chibar_{ad} - \epsilon_{bd}\Hodge{\betabar}_a       V_{dc} -\epsilon_{cd} \Hodge{\betabar}_c V_{bd}\\ +\chibar_{ac}V_{bd} \eta_{d} + \chibar_{ab} V_{dc} \eta_{d} - \chibar_{ad} \nabla_d V_{bc}.
\end{multline*}
\end{proposition}
\begin{proposition}\label{commutation}
Assume $\nabla_4 \phi = F_0$. Let $\nabla_4 \nabla^i \phi = F_i$. Then \begin{multline*} F_i= \sum_{i_1+ i_2 + i_3 =i} \nabla^{i_1} (\eta+\etabar)^{i_2} \nabla^{i_3}F_0 + \sum_{i_1+ i_2 + i_3 +i_4 =i-1} \nabla^{i_1} (\eta+\etabar)^{i_2} \nabla^{i_3} \beta \nabla^{i_4}\phi   \\    +       \sum_{i_1+ i_2 + i_3 +i_4 =i} \nabla^{i_1} (\eta+\etabar)^{i_2} \nabla^{i_3} \chi \nabla^{i_4}\phi.                \end{multline*}Assume now that $\nabla_3 \phi =G_0$. Let $\nabla_3 \nabla^i \phi= G_i$. Then \begin{multline*}
G_i + \frac{i}{2}\tr \chibar \nabla^i \phi = \sum_{i_1+ i_2 + i_3 =i}\nabla^{i_1} (\eta+\etabar)^{i_2} \nabla^{i_3}G_0 \\ + \sum_{i_1+ i_2 + i_3 +i_4 =i-1} \nabla^{i_1}(\eta+\etabar)^{i_2} \nabla^{i_3} \betabar \nabla^{i_4} \phi \\ + \sum_{i_1+ i_2 + i_3 +i_4 =i-1} \nabla^{i_1}(\eta+\etabar)^{i_2} \nabla^{i_3} (\chibarhat, \widetilde{\tr\chibar})\nabla^{i_4} \phi \\ + \sum_{i_1+ i_2 + i_3 +i_4 =i-1} \nabla^{i_1}(\eta+\etabar)^{i_2+1} \nabla^{i_3} \tr\chibar \nabla^{i_4} \phi 
\end{multline*}
\end{proposition}Finally, we can replace $\beta, \betabar$ by expressions involving Ricci coefficients, under the Codazzi equations:

\begin{align*}   \beta = - \text{div} \chihat + \frac{1}{2}\nabla \tr\chi - \frac{1}{2}(\eta-\etabar)\cdot (\chihat  - \frac{1}{2}\tr \chi), \\ \betabar=  \text{div} \chibarhat - \frac{1}{2}\nabla \tr \chi - \frac{1}{2}(\eta-\etabar)(\chibarhat - \frac{1}{2}\tr\chibar) .        \end{align*}That way, we arrive at the following:

\begin{proposition}\label{commutationformulaeprop}
Suppose $\nabla_4 \phi = F_0$. Let $\nabla_4 \nabla^i \phi = F_i$. Then \[ F_i = \sum_{i_1+ i_2 + i_3 =i}   \nabla^{i_1} \psi^{i_2} \nabla^{i_3}F_0 +   \sum_{i_1+ i_2 + i_3 +i_4 =i} \nabla^{i_1}\psi^{i_2} \nabla^{i_3}(\psi, \chihat) \nabla^{i_4}\phi.       \]Similarly, suppose $\nabla_3 \phi = G_0$. Let $\nabla_3\nabla^i \phi = G_i$. Then 

\begin{gather*}
G_i + \frac{i}{2}\tr\chibar \nabla^i \phi = \sum_{i_1+ i_2 + i_3 =i} \nabla^{i_1} \psi^{i_2} \nabla^{i_3} G_0 \\ + \sum_{i_1+ i_2 + i_3 +i_4 =i} \nabla^{i_1} \psi^{i_2}\nabla^{i_3}(\psi, \chibarhat, \widetilde{\tr \chibar}) \nabla^{i_4}\phi \\+ \sum_{i_1+ i_2 + i_3 +i_4 =i-1} \nabla^{i_1} \psi^{i_2+1} \nabla^{i_3} \tr \chibar \nabla^{i_4}\phi.
\end{gather*}
\end{proposition}

\subsection{General elliptic estimates for Hodge systems}

We will also make use of the following statements on Hodge systems coming from \cite{AA20}.

\begin{proposition}\label{Hodgeprop1}
Under the assumptions of Theorem \ref{main2} and the bootstrap assumptions \eqref{bootstrap1}-\eqref{bootstrap4}, if $\phi$ is a totally symmetric $(r+1)$-covariant tensor field on a $2-$sphere $(\mathbb{S}^2,\gamma)$ satisfying
\[  \div \phi = f, \curl \phi =g, \tr\phi =h,    \]then for $1\leq i \leq N$ there holds \begin{equation*}
\twoSu{u^i \nabla^i \phi} \lesssim  \sum_{j=0}^{i-1} \left( \twoSu{u^{j+1}\nabla^{j}(f,g)} + \twoSu{u^j \nabla^j h} + \twoSu{u^j \nabla^j \phi }  \right).
\end{equation*}
\end{proposition}
\begin{proof}
 Recall the following identity from Chap. 7 in \cite{C09} that for $\phi$, $f$, $g$, and $h$ as above 
\begin{equation}
    \int_{S_{u,\ubar}}\left( |\nabla \phi|^2+(r+1)K|\phi|^2 \right) = \int_{S_{u,\ubar}} \left( |f|^2+|g|^2+K|h|^2\right) 
\end{equation}
Notice that $||K||_{L^{\infty}_u L^{\infty}_{\ubar} L^{\infty}(S_{u,\ubar})} \lesssim \frac{1}{|u|^2}$ by the bootstrap assumption (3.4). This implies the conclusion for $i=1$ after multiplying by $u^2$. For $i>1$ we recall again from \cite{C09} that the symmetrized angular derivative of $\phi$ defined by 
\begin{gather*}
     (\nabla \phi)^s_{BA_1 \cdot\cdot\cdot A_{r+1}} \equiv \frac{1}{r+2}\left(  \nabla_B \phi_{A_1\cdot\cdot\cdot A_r}+\sum^{r+1}_{i=1}\nabla_{A_i}\phi_{A_1 \cdot \cdot \cdot <A_i>B\cdot \cdot \cdot A_{r+1}} \right)
\end{gather*}
 
satisfies the div-curl system 
\begin{gather*}
    \div (\nabla \phi)^s =(\nabla f)^s-\frac{1}{r+2}(^*\nabla g)^s+(r+1)K\phi -\frac{2K}{r+1}(\gamma \otimes^s h), \\
    \curl (\nabla \phi)^s=\frac{r+1}{r+2}(\nabla g)^s+(r+1)K(^*\phi)^s, \\ 
\text{tr} (\nabla \phi)^s=\frac{2}{r+2}f+\frac{r}{r+2}(\nabla h)^s, 
\end{gather*}
    where 
    \[ (\gamma \otimes^s h)_{A_1 \cdot \cdot \cdot A_{r+1}}\equiv \gamma_{A_i A_j} \sum_{i<j=1,\cdot\cdot\cdot,r+1}h_{A_1\cdot\cdot\cdot<A_i>\cdot\cdot\cdot<A_j>\cdot\cdot\cdot A_{r+1}}
    \]
    and 
    \[ (^*\phi)^s_{A_1\cdot \cdot \cdot A_{r+1}} \equiv \frac{1}{r+1}\sum_{i=1}^{r+1} \slashed{\epsilon}_{A_i}^B\phi_{A_1\cdot  \cdot\cdot <A_i>B\cdot \cdot \cdot A_r }
    \]
Using (3.9) and iterating we get for $i\leq N$
\begin{gather*}
    ||\nabla^i \phi||^2_{L^2(S){u,\ubar})} \lesssim ||\nabla^{i-1}(f,g)||^2_{L^2(S_{u,\ubar})}+||K(|\nabla^{i-2}(f,g)|^2+|\nabla^{i-1}(\phi,h)|^2)||_{L^1(S_{u,\ubar})}\\
    +||K\left( \sum_{i_1+2i_2+i_3=i-3} \nabla^{i_1}K^{i_2+1}\nabla^{i_3}(\phi,h)\right)^2||_{L^1(S_{u,\ubar})} +||K\left(\sum_{i_1+2i_2+i_3=i-4}\nabla^{i_1}K^{i_2+1}\nabla^{i_3}f\right)^2||_{L^1(S_{u,\ubar})}
    \\+ \sum_{i_1+2i_2+i_3=i-2}||\nabla^{i_1}K^{i_2+1}\nabla^{i_3}(\phi,h)||^2_{L^2(S_{u,\ubar})}+\sum_{i_1+2i_2+i_3=i-3}||\nabla^{i_1}K\nabla^{i_2}(f,g)||^2_{L^2(S_{u,\ubar})}
\end{gather*}
where we have adopted the convention that $\sum_{i\leq -1}=0$. Whenever a $K$-term appears with at most $N-4$ derivatives, we estimate it in $L^\infty$. Whenever a $K$-term contains between $N-3$ and $N-2$ derivatives we shall estimate it in $L^2$ and the rest of the terms in $L^\infty$, noting that we can estimate terms of the form $||\nabla^i(f,g,h)||_{L^\infty}$ with $i\leq N-4$ by the corresponding norms in $L^2$ through the standard Sobolev embedding. Using the argument of Lemma 6.1 in \cite{AA20}, after translating back to standard $L^p$ norms, there holds
\[
\sum_{i\leq N-4} || |u|^i\nabla^iK||_{L^\infty(S_{u,\ubar})}+\sum_{j\leq N-2}|| |u|^j \nabla^j K||_{L^2(S_{u,\ubar})} \lesssim 1
\]
Therefore, for $i\leq N$, we have 
\[
|| |u|^i \nabla^i \phi||^2_{L^2(S_{u,\ubar})} \lesssim \sum_{j\leq i-1}\left( || |u|^{j-1}\nabla^j(f,g)||^2_{L^2(S_{u,\ubar})}+|| |u|^j\nabla^j(\phi,h)||^2_{L^2(S_{u,\ubar})}\right) 
\]
which finishes the proof.
\end{proof}

\begin{proposition}\label{Hodgeprop2}
Suppose $\phi$ is a symmetric traceless $2-$tensor satisfying \[ \div \phi =f.     \]Then, under the assumptions of Theorem \ref{main2} and the bootstrap assumptions \eqref{bootstrap1}-\eqref{bootstrap4}, there holds \[  \twoSu{u^i \nabla^i \phi} \lesssim \sum_{j=0}^{i-1} \left( \twoSu{u^{j+1}\nabla^{j} f} + \twoSu{u^j \nabla^j \phi} \right).    \]
\end{proposition}
\begin{proof}
This is an application of Proposition 3.14 by noticing that $\curl \phi=^*f$. This is a straightforward calculation, using that the $2$-tensor $\phi$ is symmetric and traceless.
\end{proof}

\subsection{Estimates for the Ricci coefficients}
Our goal in this section will be to rigorously verify the inequality

\[ \mathcal{O}_N \lesssim 1+ \widetilde{O}_{N+1,2} + \mathcal{R}_N.     \]
We shall give the proof for an example term, in particular the first term $\chihat$ and the other proofs are done by following \cite{AL17} in the same way as for $\chihat$.

\begin{proposition}\label{riccihelpful}
Under the assumptions of Theorem \ref{main2} and the bootstrap assumptions \eqref{bootstrap1}, \eqref{bootstrap2}, \eqref{bootstrap3} and \eqref{bootstrap4} we have

\[   \sum_{i_1+i_2 \leq N} \twoSu{u^{i_1+i_2+1}\nabla^{i_1} \psibar^{i_2+1}} \lesssim \sum_{i \leq N} \twoSu{u^{i+1}\nabla^i \psibar}     \]and \[ \sum_{i_1+i_2 \leq N-2}  \inftySu{u^{i_1+i_2+2}\nabla^{i_1} \psibar^{i_2+1}} \lesssim \sum_{i \leq N-2} \inftySu{u^{i+2}\nabla^i \psibar}.    \]In particular, using the bootstrap assumption \eqref{bootstrap1} we obtain
\[   \sum_{i_1+i_2 \leq N} \twoSu{u^{i_1+i_2+1}\nabla^{i_1} \psibar^{i_2+1}} +\sum_{i_1+i_2 \leq N-2}  \inftySu{u^{i_1+i_2+2}\nabla^{i_1} \psibar^{i_2+1}} \lesssim \delta \al b^{\frac{1}{4}}.   \]

\end{proposition}
\begin{proof} Notice that we can write the expression $\nabla^{i_1}\psibar^{i_2+1}$ as 

\[  \nabla^{i_1}\psibar^{i_2+1} = \sum_{k=1}^{i_2+1} \nabla^{j_k} \psibar,        \]where $j_1+\dots+j_{i_2+1} = i_1$. Assume, without loss of generality, that the term $j_{i_2+1}$ is the largest. We then write \[ u^{i_1+i_2+1} \nabla^{i_1}\psibar^{i_2+1} = u^{i_2} \cdot(u^{j_{i_2+1}+1} \nabla^{j_{i_{2}+1}} \psibar) \cdot \prod_{k=1}^{i_2}(u\cdot \nabla)^{j_k}\psibar.     \]Our philosophy is to choose the $\psibar$ term with the highest order of derivatives and estimate it in $L^2(S_{u,\ubar})$. The crucial point to notice is that, in the expression above, each of the terms $j_1, \dots, j_{i_2}$ is at most $N-2$. This is easy to see by contradiction. Indeed, assume without loss of generality that $j_1 \geq N-1$. Then since $j_1 \leq j_{i_2+1}$ we have \[ 2N-2 \leq j_1+j_{i_2+1} \leq i_1 \leq i_1+i_2\leq N,     \]a contradiction for $N \geq 3$. What follows is that we can estimate every other term in $L^\infty(S_{u,\ubar})$ using the bootstrap assumptions, since our bootstrap assumptions for $L^\infty$ norms include derivatives up to order $N-2$.
Thus, there holds \begin{multline*}
\sum_{i_1+i_2 \leq N}\twoSu{u^{i_1+i_2+1}\nabla^{i_1}\psibar^{i_2+1}} = \sum_{i_1+i_2 \leq N}\twoSu{u^{i_2} \cdot(u^{j_{i_2+1}+1} \nabla^{j_{i_{2}+1}} \psibar) \cdot \prod_{k=1}^{i_2}(u\cdot \nabla)^{j_k}\psibar} \\ \lesssim u^{i_2} \cdot \twoSu{u^{j_{i_2+1}+1} \nabla^{j_{i_{2}+1}} \psibar}\cdot \prod_{k=1}^{i_2} \inftySu{(u\cdot \nabla)^{j_k}\psibar}. 
\end{multline*}Each term estimated in $L^\infty$ can be bounded by $\frac{\delta \al b^{\frac{1}{4}}}{u^2}$ by the bootstrap assumption \eqref{bootstrap1}. It follows that

\begin{multline*}
\sum_{i_1+i_2 \leq N}\twoSu{u^{i_1+i_2+1}\nabla^{i_1}\psibar^{i_2+1}} \lesssim  u^{i_2} \cdot \twoSu{u^{j_{i_2+1}+1} \nabla^{j_{i_{2}+1}} \psibar}\cdot \left(\frac{\delta \al b^{\frac{1}{4}}}{u^2}\right)^{i_2} \\\lesssim \twoSu{u^{j_{i_2+1}+1} \nabla^{j_{i_{2}+1}} \psibar} \lesssim \sum_{i \leq N} \twoSu{u^{i+1}\nabla^{i}\psibar},
\end{multline*}since $\delta \al b^{\frac{1}{4}}/ u \leq 1/b^{\frac{3}{4}}\lesssim 1.$ The second statement follows similarly.
\end{proof}

\par \noindent We are now in a position to move on and estimate $\chihat$.

\begin{proposition}\label{riccichihat}
Under the assumptions of Theorem \ref{main2} and the bootstrap assumptions \eqref{bootstrap1}-\eqref{bootstrap4}, there holds \[ \sum_{i \leq N} \twoSu{u^i \nabla^i \chihat} \leq \al.     \]	
\end{proposition}
\begin{proof}
There holds \[ \nabla_3 \chihat + \frac{1}{2}\tr\chibar\hspace{.5mm} \chihat = \nabla \eta + (\psi,\psibar)\psibar.     \]Using Proposition \ref{commutation} and commuting with $i$ angular derivatives, we have that for any $i$ there holds \begin{multline*}
\nabla_3 \nabla^i \chihat + \frac{i+1}{2}\tr\chibar \nabla^i \chihat  = \nabla^{i+1}\eta +\sum_{i_1+i_2=i}\nabla^{i_1} \psibar^{i_2+2} + \sum_{i_1+ i_2 + i_3 =i} \nabla^{i_1}\psibar^{i_2+1}\nabla^{i_3}\psi \\+ \sum_{i_1+ i_2 + i_3 =i-1} \frac{1}{u} \nabla^{i_1} \psibar^{i_2+1}\nabla^{i_3}\psi.
\end{multline*}We apply Proposition \ref{transportprop2} which tells us that the quantity $\lVert u^i \nabla^i \chihat \rVert_{L_u^\infty L_{\ubar}^\infty L^2(S_{u,\ubar})}$ can be controlled by the sum of $\lVert u^i \nabla^i \chihat \rVert_{L_u^\infty L_{\ubar}^\infty L^2(S_{1,\ubar})}$ and the $\lVert u^i \cdot\rVert_{L_u^\infty L_{\ubar}^\infty L^2(S_{u,\ubar})}$ norm of the right-hand side in the equation above. We now estimate each of the terms on the right-hand side for $i \leq N$:

\begin{itemize}
	\item We first look at the linear term in $\eta$. For $i \leq N-1$, the $\nabla^{i+1} \eta$ term can be controlled by the bootstrap assumptions on $\mathcal{O}_N$. In particular, we have \begin{equation} \sum_{i \leq N-1}  \lVert u^{i} \nabla^{i+1} \eta \rVert_{L_u^1 L^2(S_{u,\ubar})} \leq\Big \lVert \frac{1}{u^2} \Big \rVert_{L_u^1}   \lVert u^{i+2}\nabla^{i+1}\psibar \rVert_{L_{\ubar}^\infty L^2(S_{u,\ubar})}  \lesssim \frac{\delta \al b^{\frac{1}{4}}}{u},        \end{equation} where we have used the bootstrap assumption \eqref{bootstrap1}. For the top-order term $i=N$ we estimate \begin{equation}   \lVert u^{N} \nabla^{N+1} \eta \rVert_{L_u^1 L^2(S_{u,\ubar})} \leq \Big \lVert \frac{1}{u} \Big \rVert_{L_u^2}\lVert u^{N+1} \nabla^{N+1} \eta \rVert_{L_u^2 L^2(S_{u,\ubar})} \lesssim \frac{\delta^{\frac{1}{2}}\al}{u^{\frac{1}{2}}}\widetilde{O}_{N+1,2} \lesssim \frac{\delta^{\frac{1}{2}} \al b^{\frac{1}{4}}}{u^{\frac{1}{2}}}, \end{equation} where in the last inequality we used the bootstrap assumption \eqref{bootstrap3}.
	\item We then control the second and third terms together. Here, we use the
	estimates derived in Proposition \ref{riccihelpful}. There holds \begin{multline}
	\sum_{i \leq N} \hspace{.5mm} 
	\Big \lVert \sum_{i_1+i_2 +i_3=i} u^i \nabla^{i_1}\psibar^{i_2+1} \nabla^{i_3}(\psi,\psibar) \Big \rVert_{L_u^1 L^2(S_{u,\ubar})}\\ \lesssim  \sum_{i_1+i_2 \leq N} \lVert u^{i_1+i_2-1} \nabla^{i_1}\psibar^{i_2+1} \rVert_{L_u^1 L^2(S_{u,\ubar})}\cdot \sum_{i_3\leq N-2} \lVert u^{i_3+1} \nabla^{i_3}(\psi,\psibar) \rVert_{L_u^\infty L^\infty(S_{u,\ubar})} \\ + \sum_{i_1+i_2 \leq N} \lVert u^{i_1+i_2-1} \nabla^{i_1}\psibar^{i_2+1} \rVert_{L_u^1 L^{\infty}(S_{u,\ubar})}\cdot \sum_{i_3\leq N-2} \lVert u^{i_3+1} \nabla^{i_3}(\psi,\psibar) \rVert_{L_u^\infty L^2(S_{u,\ubar})} \\ \lesssim \frac{\delta \al b^{\frac{1}{4}}}{u}\left( \al b^{\frac{1}{4}} +\frac{\delta \al b^{\frac{1}{4}}}{u} \right) \lesssim \frac{\delta a b^{\frac{1}{2}}}{u}.
	\end{multline}
	
	\item Finally, the last term is estimated as follows
	
	\begin{multline}
	\sum_{i \leq N} \Big \lVert \sum_{i_1+i_2 +i_3=i-1} u^{i-1} \nabla^{i_1}\psibar^{i_2+1} \nabla^{i_3}\psi \Big \rVert_{L_u^1 L^2(S_{u,\ubar})} \\ \lesssim \sum_{i_1+i_2 \leq N-1} \lVert u^{i_1+i_2-1} \nabla^{i_1}\psibar^{i_2+1} \rVert_{L_u^1 L^2(S_{u,\ubar})} \sum_{i_3 \leq 1} \lVert u^{i_3+1}\nabla^{i_3}\psi \rVert_{L_u^{\infty} L^{\infty}(S_{u,\ubar})} \\ +\sum_{i_1+i_2 \leq 1} \lVert u^{i_1+i_2-1} \nabla^{i_1}\psibar^{i_2+1} \rVert_{L_u^1 L^{\infty}(S_{u,\ubar})} \sum_{i_3 \leq N-1} \lVert u^{i_3+1}\nabla^{i_3}\psi \rVert_{L_u^{\infty} L^{2}(S_{u,\ubar})} \\ \lesssim \frac{\delta \al b^{\frac{1}{4}}}{u}\left( \al b^{\frac{1}{4}} +\frac{\delta \al b^{\frac{1}{4}}}{u} \right) \lesssim \frac{\delta a b^{\frac{1}{2}}}{u}.
	\end{multline}Finally, applying the condition that $u > \delta \al b$ it is easy to see that all the terms above are bounded by $\al$. Also, recalling that initially we have \[ \sum_{i \leq N+3}  \lVert \nabla^i \chihat_0 \rVert_{L_u^{\infty} L^{2}(S_{1,\ubar})} \leq \al,     \]we get \[ \sum_{i\leq N} \lVert u^i \nabla^i \chihat \rVert_{L_u^{\infty} L^2(S_{u,\ubar})} \lesssim \al.    \]
\end{itemize}The estimates for the rest of the Ricci coefficients are obtained in the same way.
\end{proof}

The way to generalize the energy estimates to higher order is still of the same philosophy. In particular, the introduction of the crucial functions $\omega^\dagger$ and $\mu$ and their commutation formula is preserved.We assume at this stage the inequality \begin{equation}
\mathcal{O}_N \lesssim 1 + \widetilde{O}_{N+1,2} + \mathcal{R}_N
\end{equation} and give an example of how one can show \[ \widetilde{O}_{N+1,2}\lesssim 1+ \mathcal{R}_N.     \]We are ready to address the elliptic estimates.

\begin{proposition}
Under the assumptions of Theorem \ref{main2} and the bootstrap assumptions \eqref{bootstrap1}-\eqref{bootstrap4}, there holds 
\[ \lVert u^{N+2}\nabla^{N+1}\tr\chi \rVert_{L_{\ubar}^2 L^2(S_{u,\ubar})}  \lesssim \delta^{\frac{3}{2}}ab^{\frac{1}{4}}     \]and \[      \lVert u^{N+1}\nabla^{N+1}\chihat \rVert_{L_{\ubar}^2 L^2(S_{u,\ubar})}  \lesssim \delta^{\frac{1}{2}} \al (1+\mathcal{R}).      \]
\end{proposition}
\begin{proof}
Consider the following equation
\[    \nabla_4 \tr\chi + \frac{1}{2}(\tr\chi)^2 = -   \lvert \chihat \rvert^2 - 2\omega \tr\chi.       \]Commuting $(N+1)$ times  with angular derivatives, we get

\[ \nabla_4 \nabla^{N+1} \tr\chi = \sum_{i_1+i_2 +i_3+i_4=N+1}  \nabla^{i_1}\psibar^{i_2}\nabla^{i_3}\psi \nabla^{i_4}\psi.       \]Using Proposition \ref{transportprop1} we can estimate $\lVert u^{N+2}\nabla^{N+1}\tr\chi \rVert_{L_{\ubar}^\infty L_u^\infty L^2(S_{u,\ubar})}$ by the $\lVert u^{N+2} \cdot \rVert_{L_{\ubar}^\infty L_u^1 L^2(S_{u,\ubar})}$ norm of the right-hand side. It is important here once again to differentiate between the case where all the derivatives fall on a $\psi$ and the other cases.

\begin{multline}
\Big \lVert u^{N+2} \sum_{i_1+i_2 +i_3+i_4=N+1}  \nabla^{i_1}\psibar^{i_2}\nabla^{i_3}\psi \nabla^{i_4}\psi \Bigg \rVert_{L_{u}^\infty L_{\ubar}^1 L^2(S_{u,\ubar})}\\ \lesssim \delta^{\frac{1}{2}} \lVert u \psi \rVert_{L_{\ubar}^\infty L^{\infty}(S_{u,\ubar})} \lVert u^{N+1}\nabla^{N+1}\psi \rVert_{L_u^{\infty} L_{\ubar}^2 L^{2}(S_{u,\ubar})}\\ + \delta \sum_{i \leq N-2} \lVert u^{i+1}\nabla^{i}\psi \rVert_{ L_u^\infty L_{\ubar}^\infty L^\infty(S_{u,\ubar})} \sum_{i_1+i_2\leq N}\lVert u^{i}\nabla^{i}\psi \rVert_{L_u^\infty L_{\ubar}^\infty  L^2(S_{u,\ubar})}\\ +\delta \sum_{i_1+i_2 \leq N-2}\lVert u^{i_1+i_2+2}\nabla^{i_1}\psibar^{i_2+1}\rVert_{ L_u^\infty L_{\ubar}^\infty L^\infty(S_{u,\ubar})}  \\ \times \sum_{i_3\leq N-2}\lVert u^{i_3+1}\nabla^{i_3}\psi \rVert_{L_{\ubar}^\infty L^\infty(S_{u,\ubar})}\cdot \sum_{i_4\leq N} \lVert u^{i_4}\nabla^{i_4}\psi \rVert_{L_u^{\infty} L_{\ubar}^\infty L^{2}(S_{u,\ubar})}\\ +\delta \sum_{i_1+i_2 \leq N-2}\lVert u^{i_1+i_2+2}\nabla^{i_1}\psibar^{i_2+1}\rVert_{ L_u^\infty L_{\ubar}^\infty L^2(S_{u,\ubar})}  \\ \times \sum_{i_3\leq N-2}\lVert u^{i_3+1}\nabla^{i_3}\psi \rVert_{L_{\ubar}^\infty L^\infty(S_{u,\ubar})}\cdot \sum_{i_4\leq N} \lVert u^{i_4}\nabla^{i_4}\psi \rVert_{L_u^{\infty} L_{\ubar}^\infty L^{\infty}(S_{u,\ubar})}\\ \lesssim \delta a b^{\frac{1}{4}}.
\end{multline} Recall that $\nabla^{N+1}\tr\chi = 0$ initially on $\Hbar_0$. Therefore, we have $\lVert u^{N+2}\nabla^{N+1}\tr\chi \rVert_{L_{\ubar}^\infty L^2(S_{u,\ubar})} \lesssim \delta a b^{\frac{1}{4}}$. Taking $L^2$ in $\ubar$ we arrive at \[  \lVert u^{N+2}\nabla^{N+1}\tr\chi \rVert_{L_{\ubar}^2 L^2(S_{u,\ubar})} \lesssim \delta^{\frac{3}{2}} a b^{\frac{1}{4}}.    \]As for $\chihat$, we employ the elliptic estimates for Hodge systems using a starting point the Codazzi equation \[ \text{div} \hspace{.5mm} \chihat = \frac{1}{2}\nabla \tr\chi -  \beta + \psi \psibar.    \]Using Proposition \ref{Hodgeprop2} we obtain

\begin{multline}
\twoSu{u^{N+1}\nabla^{N+1}\chihat} \lesssim \sum_{i \leq N} \twoSu{u^{i+1}\nabla^{i+1}\tr\chi} + \sum_{i \leq N} \twoSu{u^{i+1}\nabla^i \beta} \\ + \sum_{i \leq N}\twoSu{\sum_{i_1+i_2=i} u^{i+1}\nabla^{i_1}\psi \nabla^{i_2}\psibar} + \sum_{i \leq N} \twoSu{u^i \nabla^i \chihat}.
\end{multline}Taking $L^2$ in $\ubar$ and using the bootstrap assumptions, the control on $\psi$ derived from the lower-order Ricci coefficient estimates as well as the bound on $\tr\chi$ just obtained, we have

\begin{multline}
\lVert u^{N+1} \nabla^{N+1}\chihat \rVert_{L_{\ubar}^2 L^2(S_{u,\ubar})}\\  \lesssim  \sum_{1\leq i \leq N} \lVert u^{i+1} \nabla^{i+1}\tr\chi \rVert_{L_{\ubar}^2 L^2(S_{u,\ubar})} + \sum_{i \leq N} \lVert u^{i+1} \nabla^{i+1}\beta \rVert_{L_{\ubar}^2 L^2(S_{u,\ubar})}\\ + \sum_{i \leq N} \sum_{i_1+i_2=i} \lVert u^{i+1} \nabla^{i+1} \psibar \nabla^{i_2}\psi\rVert_{L_{\ubar}^2 L^2(S_{u,\ubar})} + \sum_{i \leq N} \lVert u^i \nabla^i \chihat\rVert_{L_{\ubar}^2 L^2(S_{u,\ubar})}\\ \lesssim \frac{\delta^{\frac{3}{2}}a b^{\frac{1}{4}}}{u}+\delta^{\frac{1}{2}}\al + \delta^{\frac{1}{2}}\al\mathcal{R}_N \lesssim \delta^{\frac{1}{2}}\al (1+\mathcal{R}_N).
\end{multline}
\end{proof}The rest of the elliptic estimates are carried in the same spirit.

\subsection{Energy estimates}

By this stage we have managed to prove the following two inequalities:

\begin{gather} 
\mathcal{O}_N \lesssim 1+ \widetilde{\mathcal{O}}_{N+1,2} + \mathcal{R}_N,    \\  \widetilde{\mathcal{O}}_{N+1,2} \lesssim 1+\mathcal{R}_N. 
\end{gather}The goal of this section will be to explain how we can arrive at the desired \begin{equation}
\mathcal{R}_N \lesssim 1.
\end{equation}We are going to need several preliminary propositions.

\begin{proposition}
Suppose $\phi_1$ and $\phi_2$ are tensorfields. Then,
\begin{equation*}\label{energyprop1}
\iint_{\mathcal{D}(u,\ubar)} \phi_1 \nabla_4 \phi_2 + \iint_{\mathcal{D}(u,\ubar)} \phi_2 \nabla_4 \phi_1 = \int_{\Hbar_{\ubar}^{(1,\ubar)}} \phi_1\phi_2 - \int_{\Hbar_0^{(1,u)}} \phi_1 \phi_2 + \iint_{\mathcal{D}{(u,\ubar)}} (2\omega-\tr\chi)\phi_1 \phi_2.      \end{equation*}
\end{proposition}

\begin{proposition}\label{midenergyprop}
If$ \hspace{2mm} \prescript{(1)}{}{\phi}$ is an $r-$tensorfield and $ \hspace{.5mm} \prescript{(2)}{}{\phi}$ is an $(r-1)-$tensorfield, then

\begin{multline*}
\iint_{\mathcal{D}{(u,\ubar)}} \prescript{(1)}{}{\phi}^{A_1\dots A_r} \nabla_{A_r} \prescript{(2)}{}{\phi}_{A_1\dots A_{r-1}} + \iint_{\mathcal{D}{(u,\ubar)}} \nabla^{A_r}\prescript{(1)}{}{\phi}_{A_1\dots A_r} \prescript{(2)}{}{\phi}^{A_1\dots A_{r-1} } \\= - \iint_{\mathcal{D}{(u,\ubar)}}(\eta+ \etabar)\prescript{(1)}{}{\phi}\prescript{(2)}{}{\phi}.
\end{multline*}
\end{proposition}To control the components satisfying $\nabla_3$-equations, we need an analogue of Proposition \ref{energyprop1}.

\begin{proposition}\label{energyprop2}
Suppose $\phi$ is an $r-$tensorfield and let $\lambda_1 = 2(\lambda_0-\frac{1}{2})$. Then
\begin{multline*}
2 \iint_{\mathcal{D}{(u,\ubar)}} \lvert u \rvert^{2\lambda_1}\phi (\nabla_3+\lambda_0 \tr\chibar)\phi \\ = \int_{H_u^{(0,\ubar)}} \lvert u \rvert^{2\lambda_1} \lvert\phi \rvert^2 - \int_{H_0^{(0,\ubar)}} \lvert u \rvert^{2\lambda_1} \lvert \phi \rvert^2 + \iint_{\mathcal{D}{(u,\ubar)}} \lvert u \rvert^{2\lambda_1} f \lvert \phi \rvert^2,
\end{multline*}where $f$ obeys the estimate \[  \lvert f \rvert \lesssim \frac{\delta \al b^{\frac{1}{4}}}{\lvert u \rvert^2}  .    \]
\end{proposition}

\begin{proposition}
Under the assumptions of Theorem \ref{main2} and the bootstrap assumptions \eqref{bootstrap1}-\eqref{bootstrap4}, there holds
\begin{equation*}
\sum_{1\leq i \leq N} \Big(       \Big \lVert u^{i+2} \nabla^i \left(K- \frac{1}{\lvert u \rvert^2}, \check{\sigma} \right) \Big \rVert_{L_u^{\infty} L_{\ubar}^{2}(S_{u,\ubar})}  + \lVert u^{i+2} \nabla^i \betabar \rVert_{L_{\ubar}^\infty L_u^2 L^2(S_{u,\ubar}) }\Big) \lesssim \delta^{\frac{3}{2}}a^{\frac{3}{4}}.
\end{equation*} 
\end{proposition}

\begin{proof}
We begin with the following schematic Bianchi equations for $K-\frac{1}{\lvert u \rvert^2}, \check{\sigma},\betabar$:

\be \label{checksigmaeq} \nabla_3 \check{\sigma} + \div\Hodge{\betabar} + \frac{3}{2}\tr\chibar \check{\sigma} = \sum_{i_1+i_2=1} \psibar^{i_1+1}\nabla^{i_2}\psibar,        \ee

\begin{multline}\label{Keq} \nabla_3\left(K-\frac{1}{\lvert u \rvert^2} \right)+\text{div} \betabar + \frac{3}{2}\tr\chibar\left(K-\frac{1}{\lvert u \rvert^2}\right)\\ =\sum_{i_1+i_2=1} \psibar^{i_1+1}\nabla^{i_2}\psibar + \frac{1}{\lvert u \rvert}\mu + \frac{1}{\lvert u \rvert^2}(\tr\chibar  +\frac{2}{\lvert u \rvert}) + \frac{\Omega^{-1}-1}{\lvert u \rvert^3} \end{multline}and finally

\begin{equation}\label{betabareq}
\nabla_4 \betabar - \nabla K - \Hodge{\nabla}\check{\sigma} =\psibar(K,\check{\sigma}) + \sum_{i_1+i_2 +i_3=1}\psibar^{i_1} \nabla^{i_2}\left(\tr\chibar+\frac{2}{\lvert u \rvert}, \chibarhat, \omegabar \right)\nabla^{i_3}\psi  + \sum_{i_1+i_2=1} \psibar^{i_1} \nabla^{i_2}\tr\chi.
\end{equation}Commuting \eqref{checksigmaeq}, \eqref{Keq} and \eqref{betabareq} with $i$ angular derivatives, we have

\be \nabla_3 \nabla^i \check{\sigma} + \text{div} + \frac{3+i}{2}\tr\chibar \nabla^i \check{\sigma} = F_{1,i}  \label{F1i}  \ee where $F_{1,i}$ is given by
\begin{equation}  F_{1,i} = \sum_{i_1+i_2 +i_3=i+1}\nabla^{i_1}\psibar^{i_2+1} \nabla^{i_3}\psibar + \frac{1}{\lvert u \rvert} \sum_{i_1+i_2 +i_3=i-1} \nabla^{i_1}\psibar^{i_2+1}\nabla^{i_3}\check{\sigma}   \end{equation}and the equation \be\nabla_3 \nabla^i \left( K-\frac{1}{\lvert u \rvert^2} \right) - \text{div}\nabla^{i}\betabar + \frac{3+i}{2}\tr\chibar\nabla^{i}\left(K - \frac{1}{\lvert u \rvert^2} \right) = F_{2,i}, \label{F2i}  \ee where $F_{2,i}$ admits the schematic form \begin{multline}
F_{2,i} = \sum_{i_1+i_2 +i_3=i+1} \nabla^{i_1}\psibar^{i_2+1}\nabla^{i_3}\psibar +\frac{1}{\lvert u \rvert} \nabla^{i}\mu \\ +\frac{1}{\lvert u \rvert}\sum_{i_1+i_2 +i_3=i} \nabla^{i_1}\psibar^{i_2+1} \nabla^{i_3} \psibar\\+ \frac{1}{\lvert u \rvert^2}\sum_{i_1+ i_2 + i_3 =i} \nabla^{i_1}\psibar^{i_2} \nabla^{i_3}\left( \tr\chibar + \frac{2}{\lvert u \rvert} \right)  \\+ \frac{1}{\lvert u \rvert^3}\sum_{i_1+ i_2 + i_3 =i} \nabla^{i_1}\psibar^{i_2} \nabla^{i_3}(1- \Omega^{-1}) \\ + \sum_{i_1+i_2 +i_3=i} \nabla^{i_1}\psibar^{i_2+1}\nabla^{i_3}\left( K-\frac{1}{\lvert u \rvert^2} \right)\\ + \sum_{i_1+i_2 +i_3=i-1}\frac{1}{\lvert u \rvert} \nabla^{i_1}\psibar^{i_2+1}\nabla^{i_3}\left( K-\frac{1}{\lvert u \rvert^2} \right)
\end{multline}and finally the equation

\begin{equation}\label{F3i}
\nabla_4 \nabla^i \betabar - \text{div}\nabla^{i}\left(K- \frac{1}{\lvert u \rvert^2}\right) - \Hodge{\nabla}\nabla^i \check{\sigma} = F_{3,i},
\end{equation}where $F_{3,i}$ admits the schematic form 
\begin{multline}
F_{3,i} =  \sum_{i_1+i_2 +i_3+i_4=i+1} \nabla^{i_1}\psibar^{i_2}\nabla^{i_3}\left( \tr\chibar + \frac{2}{\lvert u \rvert}, \chibarhat, \omegabar \right)\nabla^{i_4}\psi\\ + \sum_{i_1+ i_2 + i_3 =i} \nabla^{i_1}\psibar^{i_2+1} \nabla^{i_3}(K- \frac{1}{\lvert u \rvert^2},\check{\sigma}).
\end{multline}Using \eqref{F1i},\eqref{F2i} and \eqref{F3i}, we can obtain the energy estimates. Using Proposition \ref{energyprop1} and equation \eqref{F3i}, we get

\begin{multline}\label{energyeq1}
\frac{1}{2}\int_{\Hbar_{\ubar}^{(1,u)}} \left( u^{i+2} \nabla^i \betabar \right)^2 \\ = \frac{1}{2}\int_{\Hbar_{0}^{(1,u)}} \left( u^{i+2} \nabla^i \betabar \right)^2 + \iint_{\mathcal{D}{(u,\ubar)}} \langle u^{i+2}\betabar, u^{i+2}\nabla_4 \nabla^i \betabar \rangle_\gamma -\iint_{\mathcal{D}{(u,\ubar)}}(\omega- \frac{1}{2}\tr\chi)(u^{i+2}\nabla^i \betabar)^2\\ = \frac{1}{2}\int_{\Hbar_{0}^{(1,u)}} \left( u^{i+2} \nabla^i \betabar \right)^2 + \iint_{\mathcal{D}{(u,\ubar)}} \langle u^{i+2}\betabar, u^{i+2}F_{3i}\rangle_\gamma\\ + \iint_{\mathcal{D}{(u,\ubar)}} \Big \langle u^{i+2} \nabla^i \betabar,  u^{i+2} \left( \nabla \nabla^i \left(K-\frac{1}{\lvert u \rvert^2}+ \Hodge{\nabla} \nabla^i \check{\sigma} \right) \right)\Big \rangle_\gamma.
\end{multline}Now applying Proposition \ref{energyprop2} and \eqref{F1i} we have \begin{multline}\label{energyeq2}
\frac{1}{2}\int_{H_u^{(0,\ubar)}} \left( u^{i+2}\nabla^i \check{\sigma}\right)^2  =\frac{1}{2}\int_{H_1^{(0,\ubar)}} \left( u^{i+2}\nabla^i \check{\sigma}\right)^2\\ + \iint_{\mathcal{D}{(u,\ubar)}} \Big \langle u^{i+2} \check{\sigma}, u^{i+2}\left(\nabla_3 + \frac{i+3}{2}\tr\chibar \right)\nabla^i \check{\sigma} \Big \rangle_\gamma - \frac{1}{2}\iint_{\mathcal{D}{(u,\ubar)}} f \left( u^{i+2}\nabla^i  \check{\sigma}\right)^2 \\ = \frac{1}{2}\int_{H_1^{(0,\ubar)}} \left( u^{i+2}\nabla^i \check{\sigma}\right)^2  + \iint_{\mathcal{D}{(u,\ubar)}} \Big \langle u^{i+2} \check{\sigma}, u^{i+2}F_{1i} \Big \rangle_\gamma -\Big \langle u^{i+2} \check{\sigma}, u^{i+2}(\text{div}\Hodge{\nabla}^i \betabar) \Big \rangle_\gamma \\- \frac{1}{2}\iint_{\mathcal{D}{(u,\ubar)}} f\left( u^{i+2}\nabla^i \check{\sigma} \right)^2.
\end{multline}In the same spirit, using Proposition \ref{energyprop2} and \eqref{F2i} we get 
\begin{multline} \label{energyeq3} \frac{1}{2}\int_{H_u^{(0,\ubar)}} \left( u^{i+2} \nabla^i \left(K- \frac{1}{u^2}\right)\right)^2  = \frac{1}{2}\int_{H_1^{(0,\ubar)}} (u^{i+2}\nabla^i (K-1))^2 \\+ \iint_{\mathcal{D}{(u,\ubar)}} \Big \langle u^{i+2} \nabla^i \left(K- \frac{1}{u^2}\right), u^{i+2}F_{2,i} \Big \rangle_\gamma \\ + \iint_{\mathcal{D}{(u,\ubar)}}\Big \langle u^{i+2} \nabla^i \left(K- \frac{1}{u^2}\right), u^{i+2}(\text{div}\nabla^i \betabar)\Big \rangle_\gamma\\ - \frac{1}{2} \iint_{\mathcal{D}{(u,\ubar)}} f\left( u^{i+2} \nabla^i \left(K- \frac{1}{u^2}\right)\right)^2.
\end{multline}Now, we can integrate by parts on the spheres $S_{u,\ubar}$  using Proposition \ref{midenergyprop} to show that the sum of the terms with the highest order of angular derivatives in \eqref{energyeq1}, \eqref{energyeq2} and \eqref{energyeq3} cancels up to a lower order error term:

\begin{multline}\label{cancellationinenergytoporder}
\iint_{\mathcal{D}{(u,\ubar)}} \Big \langle u^{i+2} \nabla^i \betabar,  u^{i+2} \left( \nabla \nabla^i \left(K-\frac{1}{\lvert u \rvert^2}+ \Hodge{\nabla} \nabla^i \check{\sigma} \right) \right)\Big \rangle_\gamma \\ -\Big \langle u^{i+2} \check{\sigma}, u^{i+2}(\text{div}\Hodge{\nabla}^i \betabar) \Big \rangle_\gamma \\+ \iint_{\mathcal{D}{(u,\ubar)}}\Big \langle u^{i+2} \nabla^i \left(K- \frac{1}{u^2}\right), u^{i+2}(\text{div}\nabla^i \betabar)\Big \rangle_\gamma \\ \lesssim \Big \lVert u^{2i+4} \nabla^i \left(K-\frac{1}{u^2}, \check{\sigma} \right)\psibar \nabla^i \betabar \Big \rVert_{L_u^1 L_{\ubar}^1 L^1(S_{u,\ubar})}\\ \lesssim \frac{\delta^{\frac{3}{2}}\al b^{\frac{1}{4}} }{u^{\frac{3}{2}}} \Big \lVert  u^{i+2} \nabla^i \left(K-\frac{1}{u^2}, \check{\sigma} \right) \Big \rVert_{ L_u^\infty L_{\ubar}^2 L^2(S_{u,\ubar})}\lVert  u^{i+2} \nabla^i \betabar \rVert_{ L_{\ubar}^\infty L_{u}^2 L^2(S_{u,\ubar})},
\end{multline}where in the last line we have used the bootstrap assumption \eqref{bootstrap1}. Upon addition of \eqref{energyeq1}, \eqref{energyeq2} and \eqref{energyeq3}, using \eqref{cancellationinenergytoporder} and  the bound for $f$ in in Proposition \ref{energyprop2}, we obtain 

\begin{multline}\label{anotherenergyeq}
\Big \lVert u^{i+2} \nabla^i \left(K-\frac{1}{u^2},\check{\sigma} \right) \Big \rVert_{L_{\ubar}^2 L^2(S_{u,\ubar})}^2 + \lVert u^{i+2} \nabla^i \betabar \rVert_{L_{u}^2 L^2(S_{u,\ubar})}^2 \\ \lesssim \Big \lVert u^{i+2} \nabla^i \left(K-\frac{1}{u^2},\check{\sigma} \right) \Big \rVert_{L_{\ubar}^2 L^2(S_{1,\ubar})}^2 + \lVert u^{2i+4} \nabla^i \check{\sigma} F_{1,i} \rVert_{L_u^1 L_{\ubar}^1 L^1(S_{u,\ubar})} \\ + \Big \lVert u^{2i+4} \nabla^i \left(K- \frac{1}{u^2},\check{\sigma} \right) F_{2,i} \Big \rVert_{L_u^1 L_{\ubar}^1 L^1(S_{u,\ubar})} \\ + \Big \lVert \frac{\delta \al b^{\frac{1}{4}}}{u^2}\cdot u^{2i+4} \nabla^i \left(K- \frac{1}{u^2},\check{\sigma} \right)\nabla^i \left(K- \frac{1}{u^2},\check{\sigma} \right) \Big \rVert_{L_u^1 L_{\ubar}^1 L^1(S_{u,\ubar})}\\ + \frac{\delta^{\frac{3}{2}} \al b^{\frac{1}{4}}}{u^{\frac{3}{2}}} \lVert u^{i+2} \nabla^i \left(K- \frac{1}{u^2},\check{\sigma} \right) \rVert_{L_u^{\infty} L_{\ubar}^2 L^{2}(S_{u,\ubar})} \lVert u^{i+2} \betabar \rVert_{L_{\ubar}^\infty L_u^2 L^2(S_{u,\ubar})}.
\end{multline}For the second, third and fourth terms, we can apply Cauchy-Schwarz in either
the $H$ or the $\Hbar$ hypersurfaces, so that the terms

\[   \Big \lVert  u^{i+2} \nabla^i \left(K- \frac{1}{u^2},\check{\sigma} \right) \Big\rVert_{ L_{\ubar}^2 L^{2}(S_{u,\ubar})}          \]and $\lVert u^{i+2} \nabla^i \betabar \rVert_{L_u^2 L^2(S_{u,\ubar})}$ can be absorbed to the left.  For the fifth term, noticing that \[ \Big \lVert \frac{\delta \al b^{\frac{1}{4}}}{u^2} \Big \rVert_{L_u^1} \lesssim \frac{1}{b^{\frac{3}{4}}},       \]we can control it by Gr\"onwall's inequality. Finally, since \[  \Big \lVert \frac{\delta^{\frac{3}{2}} \al b^{\frac{1}{4}}}{u^{\frac{3}{2}}} \Big \rVert_{L_u^{\infty}} \lesssim \frac{1}{b^{\frac{5}{4}}},      \]the term can be absorbed to the left hand side after using Schwarz’s inequality. Therefore, \eqref{anotherenergyeq} implies that

\begin{multline}
\Big \lVert u^{i+2} \nabla^i \left(K-\frac{1}{u^2},\check{\sigma} \right) \Big \rVert_{L_{\ubar}^2 L^2(S_{u,\ubar})}^2 + \lVert u^{i+2} \nabla^i \betabar \rVert_{L_{u}^2 L^2(S_{u,\ubar})}^2 \\ \lesssim \Big \lVert u^{i+2} \nabla^i \left(K-\frac{1}{u^2},\check{\sigma} \right) \Big \rVert_{L_{\ubar}^2 L^2(S_{1,\ubar})}^2 +\lVert u^{i+2}F_{1,i} \rVert_{L_u^1 L_{\ubar}^2 L^2(S_{u,\ubar})}\\ +\lVert u^{i+2}F_{2,i} \rVert_{L_u^1 L_{\ubar}^2 L^2(S_{u,\ubar})}+ +\lVert u^{i+2}F_{3,i} \rVert_{L_{\ubar}^1 L_{u}^2 L^2(S_{u,\ubar})}.
\end{multline}Summing over $1\leq i \leq N$ and using the fact that \[\sum_{1\leq i \leq N} \Big\lVert \nabla^i \left(K-\frac{1}{u^2}\right),\check{\sigma} \Big \rVert_{L_{\ubar}^2 L^2(S_{1,\ubar})} \lesssim \delta^{\frac{1}{2}} \al,  \]we obtain

\begin{multline}\label{midwayenergyeq}
\sum_{1\leq i \leq N} \left( 	\Big \lVert u^{i+2} \nabla^i \left(K-\frac{1}{u^2},\check{\sigma} \right) \Big \rVert_{L_{\ubar}^2 L^2(S_{u,\ubar})}^2 + \lVert u^{i+2} \nabla^i \betabar \rVert_{L_{u}^2 L^2(S_{u,\ubar})}^2  \right) \\ \lesssim  \delta^{\frac{1}{2}}\al + \sum_{1\leq i \leq N} \left( \lVert u^{i+2}F_{1,i} \rVert_{L_u^1 L_{\ubar}^2 L^2(S_{u,\ubar})} +\lVert u^{i+2}F_{2,i} \rVert_{L_u^1 L_{\ubar}^2 L^2(S_{u,\ubar})}+ +\lVert u^{i+2}F_{3,i} \rVert_{L_{\ubar}^1 L_{u}^2 L^2(S_{u,\ubar})}\right).
\end{multline}We will estimate the right hand side of \eqref{midwayenergyeq} term by term. Let us recall the schematic form of $F_{1,i}$: \[  F_{1,i} = \sum_{i_1+i_2 +i_3=i+1}\nabla^{i_1}\psibar^{i_2+1} \nabla^{i_3}\psibar + \frac{1}{\lvert u \rvert} \sum_{i_1+i_2 +i_3=i-1} \nabla^{i_1}\psibar^{i_2+1}\nabla^{i_3}\check{\sigma}.    \]Notice, in particular, that it consists of a schematic term involving only Ricci coefficients and one involving curvature (in particular $\check{\sigma}$). We deal with these terms separately. For the first term, assume without loss of generality that $i_1 \leq i_3$. We bound separately the contributions where
there are at most $NH$ derivatives falling on any of the Ricci coefficients, where
$N+1$ derivatives fall on $(\tr\chibar, \chibarhat, \omegabar)$ and where $N+1$ derivatives fall on $(\eta,\etabar)$. In particular, we have

\begin{multline}\label{F1i1}
\sum_{i \leq N} \Big \lVert u^{i+2} \sum_{i_1+ i_2 + i_3 =i+1} \nabla^{i_1} \psibar^{i_2+1}\nabla^{i_3}\psibar \Big \rVert_{L_u^1 L_{\ubar}^2 L^2(S_{u,\ubar})} \\ \lesssim \sum_{\substack{i_1+i_2 \leq N+1 \\ i_1 \leq N-2}} \lVert u^{i_1+i_2+2} \nabla^{i_1} \psibar^{i_2+1} \rVert_{L_u^{\infty} L_{\ubar}^\infty L^{\infty}(S_{u,\ubar})}\\ \times \Big(  \delta^{\frac{1}{2}} \sum_{i_3 \leq N} \lVert u^{i_3+1} \nabla^{i_3}\psibar \rVert_{L_u^{\infty} L_{\ubar}^\infty L^{2}(S_{u,\ubar})} \lVert u^{-2} \rVert_{L_u^1} \\ + \delta^{\frac{1}{2}} \lVert u^{N+1} \nabla^{N+1} (\tr\chibar, \chibarhat, \omegabar) \rVert_{ L_{\ubar}^\infty L_{u}^2 L^2(S_{u,\ubar})}  \lVert u^{-1} \rVert_{L_u^2} \\ \lVert u^{N+2} \nabla^{N+1}(\eta,\etabar) \rVert_{ L_u^\infty L_{\ubar}^2 L^2(S_{u,\ubar})}\lVert u^{-2} \rVert_{L_u^1}         \Big)\\ \lesssim \delta \al b^{\frac{1}{4}} \left(\frac{\delta^{\frac{3}{2}}\al b^{\frac{1}{4}}}{u} + \frac{\delta^{\frac{3}{2}} a^{\frac{3}{4}}}{u}\right)\lesssim \frac{\delta^{\frac{5}{2}}a^{\frac{5}{4}}b^{\frac{1}{4}}}{u}.
\end{multline}For the remaining contributions in $F_{1,i}$, we will prove the slightly more general bound where we allow $(K-\frac{1}{u^2},\check{\sigma})$ in  place of $\check{\sigma}$. Using Sobolev embedding, we have

\begin{multline}\label{F1i2}
\sum_{i \leq N} \Big \lVert u^{i+2} \sum_{i_1+i_2+i_3=i} \nabla^{i_1} \psibar^{i_2+1}\nabla^{i_3} \left(K-\frac{1}{u^2},\check{\sigma} \right) \Big \rVert_{L_u^1 L_{\ubar}^2 L^2(S_{u,\ubar})}\\ \lesssim \sum_{\substack{i_1+i_2 \leq N \\ i_1 \leq N-2}} \lVert u^{i_1+i_2+1}\nabla^{i_1}\psibar^{i_2+1} \rVert_{L_u^{\infty}L_{\ubar}^\infty L^2(S_{u,\ubar}) }\cdot \sum_{i_3 \leq N} \Big \lVert u^{i_3+2} \nabla^{i_3}\left(K-\frac{1}{u^2},\check{\sigma} \right) \Big \rVert_{ L_{u}^\infty L_{\ubar}^2 L^2(S_{u,\ubar})}\cdot\lVert u^{-2} \rVert_{L_u^1} \\ \lesssim \frac{\delta^{\frac{5}{2}}a^{\frac{5}{4}}b^{\frac{1}{2}}}{u}.
\end{multline}The final term in $F_{1,i}$ can also be controlled as above:

\begin{multline}\label{F1i3}
\sum_{i \leq N}  \Big \lVert u^{i+1} \sum_{i_1+ i_2 + i_3 =i-1} \nabla^{i_1} \psibar^{i_2+1} \nabla^{i_3}\left(K-\frac{1}{u^2},\check{\sigma} \right) \Big \rVert_{L_u^1 L_{\ubar}^2 L^2(S_{u,\ubar})}\\ \lesssim \sum_{\substack{i_1+i_2 \leq N-1\\ i_1\leq N-2}} \lVert u^{i_1+i_2+1} \nabla^{i_1}\psibar^{i_2+1} \rVert_{L_u^{\infty} L_{\ubar}^\infty L^2(S_{u,\ubar})}    \\ \times \sum_{i_3 \leq N-1} \Big \lVert u^{i_3+2} \nabla^{i_3}\left(K-\frac{1}{u^2},\check{\sigma} \right) \Big \rVert_{ L_{u}^\infty L_{\ubar}^2 L^2(S_{u,\ubar})}\cdot\lVert u^{-2} \rVert_{L_u^1}\\ \lesssim \frac{\delta^{\frac{5}{2}}a^{\frac{5}{4}}b^{\frac{1}{2}}}{u}.
\end{multline}We move on to estimates for $F_{2,i}$. Notice that the first, sixth and
seventh terms are already estimated above in \eqref{F1i1}, \eqref{F1i2} and \eqref{F1i3}. For	the second term, we need to use the improved estimates for $\nabla^i \mu$ provided in the elliptic estimates section. In particular, we need to use the fact that $i\geq 1$. 

\begin{equation}
\sum_{1\leq i \leq N} \uubarSuu{u^{i+1} \nabla^i \mu}{1}{2}{2} \lesssim \delta^{\frac{1}{2}} \sum_{1\leq i \leq N}\uubarSuu{u^{i+3}\nabla^i\mu}{\infty}{\infty}{2} \lVert u^{-2} \rVert_{L_u^1} \lesssim \frac{\delta^{\frac{5}{2}}a^{\frac{5}{4}}b^{\frac{1}{4}}}{\lvert u \rvert}.
\end{equation}For the third term in $F_{2,i}$,  we have

\begin{multline}
\sum_{i \leq N} \uubarSuu{u^{i+1} \sum_{i_1+i_2+i_3=i} \nabla^{i_1} \psibar^{i_2+1} \nabla^{i_3}\psibar}{1}{2}{2} \\ \lesssim \delta^{\frac{1}{2
}} \sum_{\substack{i_1+i_2 \leq N \\ i_1 \leq N-2}}\uubarSuu{u^{i_1+i_2+2}\nabla^{i_1}\psibar^{i_2+1}}{\infty}{\infty}{\infty} \cdot \sum_{i_3 \leq N} \uubarSuu{u^{i_3+1} \nabla^{i_3}\psibar}{\infty}{\infty}{2}\cdot \lVert u^{-2}\rVert_{L_u^1}\\ \lesssim \frac{\delta^{\frac{5}{2}}a b^{\frac{1}{2}}}{\lvert u \rvert}.
\end{multline}For the fourth term in $F_{2,i}$, we need to use the improved bounds in $\nabla^i(\tr\chibar +\frac{2}{\lvert u \rvert})$ and the fact that $i\geq 1$.

\begin{multline}
\sum_{1\leq i \leq N} \bigubaruSuu{u^i \sum_{i_1+ i_2 + i_3 =i}     \nabla^{i_1}\psibar^{i_2} \nabla^{i_3}\left( \tr\chibar + \frac{2}{\lvert u \rvert}\right) }{1}{2}{2}\\ \lesssim \delta^{\frac{1}{2}} \sum_{1\leq i \leq N} \biguubarSuu{u^{i+\frac{3}{2}} \nabla^i \left(\tr\chibar+ \frac{2}{\lvert u \rvert}\right)}{\infty}{\infty}{2} \lVert u^{-\frac{3}{2}}\rVert_{L_u^1}\\ + \delta^{\frac{1}{2}}  \sum_{i_1+i_2 \leq N-2} \Big \lVert u^{i_1+i_2+2} \nabla^{i_1}\psibar^{i_2+1} \Big\rVert_{ L_u^\infty L_{\ubar}^\infty L^\infty(S_{u,\ubar})}\cdot \sum_{i_3 \leq N} \uubarSuu{u^{i_3+1} \nabla^{i_3}\psibar}{\infty}{\infty}{2}\cdot \lVert u^{-2}\rVert_{L_u^1}\\ \lesssim  \frac{\delta^2 a^{\frac{3}{4}}}{\lvert u \rvert^{\frac{1}{2}}} +\frac{\delta^{\frac{5}{2}}a b^{\frac{1}{2}}}{\lvert u \rvert}.
\end{multline}For the fifth term in  $F_{2,i}$, finally, we use the fact that $i\geq 1$ and the improved bounds on $\nabla^i (\Omega^{-1}-1)$. We have

\begin{multline}
\sum_{1\leq i \leq N} \biguubarSuu{u^{i-1}\sum_{i_1+ i_2 + i_3 =i}\nabla^{i_1}\psibar^{i_2+1}\nabla^{i_3}\left(\Omega^{-1}-1\right) }{1}{2}{2}\\ \lesssim \delta^{\frac{1}{2}} \sum_{1\leq i \leq N} \uubarSuu{u^{i+\frac{1}{2}} \nabla^i (\log \Omega)}{\infty}{\infty}{2} \lVert u^{-\frac{3}{2}}\rVert_{L_u^1}\\ + \delta^{\frac{1}{2}} \sum_{i_1+i_2 \leq N-2} \uubarSuu{u^{i_1+i_2+2}\nabla^{i_1}\psibar^{i_2+1}}{\infty}{\infty}{\infty}\cdot \sum_{i_3 \leq N-1} \uubarSuu{u^{i_3+1}\nabla^{i_3}\psibar}{\infty}{\infty}{2} \lVert u^{-2} \rVert_{L_u^1} \\ \delta^{\frac{1}{2}}  \sum_{i_1+i_2 \leq N} \uubarSuu{u^{i_1+i_2+1}\nabla^{i_1}\psibar^{i_2+1}}{\infty}{\infty}{2}\cdot \sum_{i_3 \leq N-1} \uubarSuu{u(1-\Omega^{-1})}{\infty}{\infty}{2\infty} \lVert u^{-2} \rVert_{L_u^1}\\ \lesssim \frac{\delta^2 a^{\frac{3}{4}}}{\lvert u\rvert^{\frac{1}{2}}} +\frac{\delta^{\frac{5}{2}}a b^{\frac{1}{2}}}{\lvert u \rvert}.
\end{multline}We move on to estimates for $F_{3,i}$. For the first term, we have

\begin{multline}
\sum_{i \leq N}  \bigubaruSuu{u^{i+2} \sum_{i_1+i_2 +i_3+i_4=i+1} \nabla^{i_1} \psibar^{i_2}\nabla^{i_3}\left( \tr\chibar+\frac{2}{\lvert u \rvert},\chibarhat,\omegabar \right)}{1}{2}{2}\\ \lesssim \delta  \sum_{\substack{i_1+i_2\leq N+1\\ i_1 \leq N-2}} \uubarSuu{u^{i_1+i_2+2}\nabla^{i_1}\psibar^{i_2+1}}{\infty}{\infty}{\infty} \cdot \sum_{i_3 \leq N} \uubarSuu{u^{i_3}\nabla^{i_3}\psibar}{\infty}{\infty}{2} \lVert u^{-1} \rVert_{L_u^2} \\ + \delta^{\frac{1}{2}} \uubarSuu{u^2 \psibar}{\infty}{\infty}{\infty}\uubarSuu{u^{N+1}\nabla^{N+1}\psibar}{\infty}{\infty}{2} \lVert u^{-1} \rVert_{L_u^2} \\ + \delta \sum_{\substack{i_1+i_2 \leq N+1 \\  i_1 \leq N}}\uubarSuu{u^{i_1+i_2+1}\nabla^{i_1}\psibar^{i_2+1}}{\infty}{\infty}{2} \sum_{i_3 \leq N-2} \uubarSuu{u^{i_3+1}\nabla^{i_3}\psi}{\infty}{\infty}{2} \lVert u^{-1} \rVert_{L_u^2} \\ + \delta \bigubaruSuu{u^{N+1} \nabla^{N+1}\left(\tr\chibar+\frac{2}{\lvert u \rvert}, \chibarhat,\omegabar\right)}{\infty}{2}{2} \uubarSuu{u\psi}{\infty}{\infty}{\infty} \lesssim \frac{\delta a^2 b^{\frac{1}{4}}}{\lvert u \rvert^{\frac{1}{2}}}.
\end{multline}For the second term, we use Sobolev embedding
to get \begin{multline}
\sum_{i \leq N}\bigubaruSuu{u^{i+2}\sum_{i_1+ i_2 + i_3 =i} \nabla^{i_1}\psibar^{i_2+1}\nabla^{i_3}(K,\check{\sigma})}{1}{2}{2} \\ \lesssim \delta^{\frac{1}{2}}\sum_{i_1+i_2 \leq N} \uubarSuu{u^{i_1+i_2+1}\nabla^{i_1}\psibar^{i_2+1}}{\infty}{\infty}{2} \cdot \sum_{i_3 \leq N} \biguubarSuu{u^{i_3+2}\nabla^{i}\left(K-\frac{1}{\lvert u \rvert^2},\check{\sigma}\right)}{\infty}{2}{2} \lVert u^{-2} \rVert_{L_u^2} \\ + \delta \sum_{i_1+i_2 \leq N} \uubarSuu{u^{i_1+i_2+1}\nabla^{i_1}\psibar^{i_2+1}}{\infty}{\infty}{2}\lVert u^{-1}\rVert_{L_u^2}\\ \lesssim \frac{\delta^{3}a^{\frac{5}{4}}b^{\frac{1}{2}}}{\lvert u \rvert^{\frac{3}{2}}} + \frac{\delta^2 \al b^{\frac{1}{4}}}{\lvert u \rvert^{\frac{1}{2}}}.
\end{multline}For the final term in $F_{3,i}$, we use the improved estimates for $\nabla^i \tr\chi$, where $i\geq 1$. More precisely, there holds

\begin{multline}
\sum_{i \leq N} \bigubaruSuu{u^{i+1} \sum_{i_1+ i_2 + i_3 =i+1} \nabla^{i_1}\psibar^{i_2+1}\nabla^{i_3}\tr\chibar}{1}{2}{2}\\ \lesssim \left( \delta^{\frac{1}{2}}\uubarSuu{u^{N+2}\nabla^{N+1}\tr\chi}{\infty}{2}{2} + \delta \sum_{i \leq N} \uubarSuu{u^{i+1}\nabla^i \tr\chi}{\infty}{\infty}{2} \right) \lVert u^{-1} \rVert_{L_u^2}\\ + \delta \sum_{\substack{i_1+i_2 \leq N+1 \\ i_1\leq N-2}}\uubarSuu{u^{i_1+i_2+2}}{\infty}{\infty}{\infty} \cdot \sum_{i_3 \leq N} \uubarSuu{u^{i_3}\nabla^{i_3}\psi}{\infty}{\infty}{2} \lVert u^{-1} \rVert_{L_u^2}\\ +  \delta \sum_{\substack{i_1+i_2 \leq N+1 \\ i_1\leq N}}\uubarSuu{u^{i_1+i_2+1}\nabla^{i_1}\psibar^{i_2+1}}{\infty}{\infty}{2} \cdot \sum_{i_3 \leq N-2} \uubarSuu{u^{i_3}\nabla^{i_3}\psi}{\infty}{\infty}{\infty} \lVert u^{-1} \rVert_{L_u^2} \\ \lesssim \frac{\delta^2 a b^{\frac{1}{4}}}{\lvert u \rvert^{\frac{1}{2}}}.
\end{multline}Collecting all the above estimates and using  $\lvert u \rvert \geq \delta \al b$, the result follows.

\end{proof}\par \noindent The rest of the energy estimates are carried out in the same way. We summarize the results of this section in the following way:

\begin{theorem}\label{higherorderenergyestimates}
There exists a universal constant $k$ such that the following holds, for all $N$. Consider the following characteristic initial value problem for the Einstein vacuum equations. The initial incoming null hypersurface $\Hbar_0$ is required to coincide with a backwards light cone in Minkowski space with $0\leq u \leq 1$. Given $\delta$, for every $B>0$, there exist $a_0(B)$ and $b_0 =b_0(B)$ such that the following holds. Pick any $a$ and $b$ satisfying $a_0 \leq a < \delta^{-1}$ and $b_0 \leq \al \leq b < a< \delta^{-1}$ and assume that, along $H_0$, the initial shear satisfies

\[ \sum_{i \leq N+k} \lVert \nabla^i \chihat_0 \rVert_{L_{\ubar}^{\infty}L^2(S_{u,\ubar})} \leq \al, \]for all $0\leq \ubar \leq \delta,$ then there exists a unique solution to the Einstein vacuum equations  in the region $\delta \al b \leq u \leq 1, \hspace{.5mm} 0 \leq \ubar \leq \delta$. Moreover, the solution obeys the following higher order energy estimates:

\[ \mathcal{O}_N + \tilde{\mathcal{O}}_{N+1,2} + \mathcal{R}_N \lesssim 1 .   \]Moreover, the following improved bounds hold:

\[ \sum_{i \leq N}   \big \lVert u^{i+1}\nabla^i \left( \tr\chi - \frac{2}{u} \right) \big \rVert_{L^2(\S)} \lesssim \delta a, \hspace{2mm} \sum_{i \leq N-1}\big \lVert u^{i+2}\nabla^i \left( K - \frac{1}{u^2} \right) \big \rVert_{L^2(\S)} \lesssim \delta \al. \]   The implicit constant depends only on the initial data.
\end{theorem}

\section{The transition region}
The goal here is to obtain estimates on the following pieces.
\begin{enumerate}
\item The sphere $S_{1,\delta}$ on the edge of the outgoing null cone of the initial data.
\item The (non-trivial) incoming null cone at the border of the region of existence.
\item The additional outgoing null cone extending beyond the region of existence, on which we impose no additional shear $\chihat=0$.  
\item The spacetime region developing from the characteristic initial value problem define by these outgoing and incoming null cones. This region is called the \textit{transition region}. Its existence is proved by using the estimates shown in 2 and 3. Such existence results have been shown, by \cite{L12}, and with an initial curvature singularity by \cite{LR14}, \cite{LR17}.
\end{enumerate}Throughout this paper we shall impose an extra condition on the initial data, namely \be \label{ssa} \int_0^{\delta} \lvert \chihat (1, \ubar, \theta) \rvert^2 \dubarprime =4m_0 , \ee for some positive constant $m_0$. We can think of this constant as representing the mass of a Schwarzschild spacetime, the geometry of which is close to the geometry of the entire transition region.

\begin{remark} \label{crucialremark}
Crucially, throughout this section, we assume  a bound on all the derivatives of $\chihat_0$:
\[  \sum_{i=0}^{\infty}\lVert \nabla^i \chihat_0 \rVert_{ L_{\ubar}^{\infty} L^{2}(S_{u,\ubar})}    \leq \al.   \]This means, thanks to Theorem \ref{higherorderenergyestimates}, that for any given $N$, the estimates
\[   \mathcal{O}_N + \tilde{\mathcal{O}}_{N+1,2}+\mathcal{R}_N \lesssim 1   \]hold in the region of existence.
\end{remark}
\subsection{The geometry of $S_{1,\delta}$}
The purpose of this section is to prove the following lemma
which, roughly speaking, says the geometry of the two sphere $S_{1,\delta}$ is close to the geometry of a given $2-$sphere in a Schwarzschild spacetime with mass $m_0$.

\begin{lemma}\label{spheregeoestimates}
On the sphere $S_{1,\delta}$ we have $\alpha\equiv 0$ and for all $k$

\[   \left(\lvert \nabla^k\left(\tr\chi- 2+4m_0\right) \rvert,\lvert  \nabla^k \beta \rvert,\lvert \nabla^k \sigma \rvert, \lvert \nabla^k (K-1) \rvert , \lvert \nabla^k (\rho + 2m_0) \rvert \right) \lesssim \delta a.       \]
\end{lemma}
\begin{proof}
To begin with, the fact that $\alpha \equiv 0 $ follows from \eqref{chihateq} and the fact that $\chihat$ has compact support in $C_1^{(0,\delta)}$. For $\tr\chi$ on $C_1$, we can recall the equation \eqref{trchieq} that reads

\[ \nabla_4 \tr\chi = - \frac{1}{2}(\tr\chi)^2 - \lvert \chihat \rvert^2 - 2\omega \hspace{.5mm} \tr\chi. \]Integrate this equation along $[0,\delta]$ to obtain

\begin{multline} \lvert \tr\chi(1, \delta) - 2 +4m_0 \rvert \lesssim  \int_0^{\delta}  \lvert \tr\chi(\ubar^{\prime},1)\rvert^2 \dubarprime + \int_0^{\delta} \lvert \omega(\ubar^{\prime},1)\rvert \lvert \tr\chi(\ubar^{\prime},1)\rvert \hspace{.5mm} \text{d}\ubar^{\prime} \\\lesssim  \delta \inftySu{\tr\chi}^2 + \delta \inftySu{\tr\chi}\inftySu{\omega}  \lesssim \delta \cdot \al \cdot \al = \delta a. \end{multline} Here we have used Theorem \ref{higherorderenergyestimates}. Moreover, for $k\geq 1$, the same Theorem implies that \begin{equation}\TwoSu{u^{k+1}\nabla^k \left(\tr\chi - \frac{2}{u} \right)} \lesssim \delta a.\end{equation}Consequently, for $u=1, \ubar= \delta$, there holds
\[   \big\lVert \nabla^k \left(\tr\chi - \frac{2}{u} \right)\big\rVert_{L^2(S_{1,\delta})} \lesssim \delta a, \]for all $k$. Using the Sobolev embedding Theorem \ref{Sobolevembedding}, there holds, for all $k\geq 1$,\be \label{sphereestimate1} \big\lVert \nabla^k \left(\tr\chi - 2+4m_0 \right)\big\rVert_{L^{\infty}(S_{1,\delta})} \lesssim \delta a.\ee 

\par \noindent For $\beta$, on the initial outgoing cone $C_{1}$, recall that we have \[ \text{div} \chihat = \frac{1}{2}\nabla \tr\chi - \chihat \cdot \eta +\frac{1}{2}\tr\chi \hspace{.5mm} \eta -\beta. \]We rewrite this as \[ \beta  = \frac{1}{2}\nabla \tr\chi - \chihat \cdot \eta +\frac{1}{2}\tr\chi \hspace{.5mm} \eta -\div \chihat \]Now notice that $\chihat \equiv 0$ on $S_{1,\delta}$, so that\[\beta = \frac{1}{2} \nabla \tr\chi + \frac{1}{2}\tr\chi \hspace{.5mm} \eta,   \]on $S_{1,\delta}$. Differentiating $k$ times by $\nabla$, we get \begin{multline} \lvert \nabla^k \beta \rvert(1,\delta) \lesssim \lvert \nabla^{k+1}\tr\chi\rvert(1,\delta) + \sum_{k_1+k_2=k} \lvert \nabla^{k_1} \eta \rvert(1,\delta) \cdot \lvert \nabla^{k_2} \tr\chi \rvert(1,\delta) \\\lesssim \big\lVert \nabla^{k+1} \left(\tr\chi - 2+4m_0 \right)\big\rVert_{L^{\infty}(S_{1,\delta})} + \sum_{k_1+k_2=k}\lVert \nabla^{k_1} \tr\chi \rVert_{L^{\infty}(S_{1,\delta})} \lVert \nabla^{k_2} \eta \rVert_{L^{\infty}(S_{1,\delta})}\lesssim \delta a + \al \cdot \delta \al \lesssim \delta a, \end{multline}where we have used \eqref{sphereestimate1} and the fact that $\lvert \nabla^{k_1} \eta \rvert(1,\delta) \lesssim \delta \al$ by Theorem \ref{higherorderenergyestimates}. \vspace{3mm}
\par \noindent For $\sigma$ on the initial cone $C_1$, we have

\[  \curl \eta = \sigma - \frac{1}{2}\chihat \wedge \chibarhat = \csigma. \]Since $\lvert \nabla^{k+1} \eta \rvert(1,\delta) \lesssim \delta \al$ and $\chihat \equiv 0$ on $S_{1,\delta}$ , taking $\nabla^k$ on both sides of the above yields the desired estimates. 

\vspace{3mm} \par \noindent  For $\nabla^k (K-1)$ on $S_{1,\delta}$,
Proposition 6.9  in \cite{AL17} along with Theorem \ref{higherorderenergyestimates} and an application of Sobolev's embedding yields \[  \lvert \nabla^k (K-1) \rvert \lesssim \delta \al (1+ \tilde{\mathcal{O}}_{k+1,2} + \mathcal{R}_k) \lesssim  \delta \al    . \]Moving on to estimates for $\rho+2m_0$, notice that there holds

	\begin{equation}\begin{split}
	\rho+ \frac{2m_0}{u^3}= &- \left(K- \frac{1}{u^2}\right) - \frac{1}{4}\left(\tr\chi- \frac{2}{u} +\frac{4m_0}{u^2} \right) \left(\tr\chibar + \frac{2}{u} \right)\\ &+ \frac{1}{2}\chihat \cdot \chibarhat + \frac{1}{2u} \left(\tr\chi- \frac{2}{u} +\frac{4m_0}{u^2} \right) - \frac{1}{2u}\left( \tr\chibar + \frac2u \right) + \frac{m_0}{u^2}\left( \tr\chibar + \frac2u \right). \end{split}
	\end{equation}Given the fact that $u=1$ and $\chihat=0$ on $S_{1,\delta}$, applying $\nabla^k$ to both sides of the above identity and using the estimates on $\tr\chibar+ \f2u$ from Theorem \ref{higherorderenergyestimates}, we arrive at 
	\begin{equation}
	    \lvert \nabla^k (\rho+2m_0) \rvert \lesssim \delta a.
	\end{equation}

\end{proof}


\subsection{The geometry of the incoming cone} \label{incomingconesection}

Throughout this section, our main goal will to establish the fact that the geometry of the cone $\underline{C}_{\delta}$
is close, in a rigorous sense to be given below, to the geometry of an incoming cone in a Schwarzschild spacetime with mass $m_0$, where $m_0$ is given by \eqref{ssa}. This will be achieved by propagating the appropriate null structure equations. This is where  equation \eqref{ssa} and the control on the geometry of $S_{1,\delta}$ will prove useful, as it is those that enable us to achieve the desired closeness. We shall prove the estimates in a series of propositions. We begin with the estimates for $\chihat$.

\begin{proposition}
	\label{chihatincomingprop} There holds \[ \lVert u^k \nabla^k \chihat \rVert_{L^2(S_{u,\delta})} \lesssim \frac{\delta a}{u}.    \]
\end{proposition}

\begin{proof} We begin with the structure equation
		\be  \nabla_3 \chihat + \frac{1}{2} tr\chibar \hspace{.5mm}\chihat = \nabla \widehat{\otimes} \eta + 2\hspace{.5mm} \omegabar\hspace{.5mm} \chihat - \frac{1}{2}\hspace{.5mm} tr\chibar \hspace{.5mm} \chihat + \eta \widehat{\otimes} \eta.  \ee Commuting this with $i$ angular derivatives, we have
	
	\begin{equation}\begin{split}
	\nabla_3 \nabla^i \chihat + \frac{i+1}{2}\tr\chibar \nabla^i \chihat &= \sum_{i_1+i_2+i_3=i} \nabla^{i_1} \psi^{i_2} \nabla^{i_3+1} \eta\\ &+ \sum_{i_1+i_2+i_3+i_4=i} \nabla^{i_1}\psi^{i_2} \nabla^{i_3} \omegabar \nabla^{i_4}\chihat  + \sum_{i_1+i_2+i_3+i_4=i} \nabla^{i_1}\psi^{i_2} \nabla^{i_3}\eta \nabla^{i_4} \eta \\ &+ \sum_{i_1+i_2+1=i} \nabla^{i_1+1}\tr\chibar \nabla^{i_2}\chihat + \sum_{i_1+i_2+i_3+i_4+1=i} \nabla^{i_1}\psi^{i_2+1}\nabla^{i_3}\tr\chibar \nabla^{i_4} \chihat \\ &+ \sum_{i_1+i_2+i_3+i_4=i} \nabla^{i_1}\psi^{i_2} \nabla^{i_3}(\psi,\chibarhat,\tildetr) \nabla^{i_4} \chihat .
	\end{split}\end{equation}

	\vspace{3mm}
	\par \noindent Let us look at Proposition \ref{riccihelpful}. We follow the proof of Proposition \ref{riccichihat}. There holds \[ \nabla_3 \chihat + \frac{1}{2}\tr\chibar\hspace{.5mm} \chihat = \nabla \eta + (\psi,\psibar)\psibar.     \]Using Proposition \ref{commutation} and commuting with $i$ angular derivatives, we have that for any $i$ there holds \begin{equation*} \begin{split}
	\nabla_3 \nabla^i \chihat + \frac{i+1}{2}\tr\chibar \nabla^i \chihat &= \nabla^{i+1}\eta +\sum_{i_1+i_2=i}\nabla^{i_1} \psibar^{i_2+2} + \sum_{i_1+ i_2 + i_3 =i} \nabla^{i_1}\psibar^{i_2+1}\nabla^{i_3}\psi \\&+ \sum_{i_1+ i_2 + i_3 =i-1} \frac{1}{u} \nabla^{i_1} \psibar^{i_2+1}\nabla^{i_3}\psi. \end{split}
	\end{equation*}We now apply Proposition \ref{transportprop2}, which tells us that the quantity $\lVert u^i \nabla^i \chihat \rVert_{L^2(S_{u,\delta})}$ can be controlled by the sum of $\lVert u^i \nabla^i \chihat \rVert_{ L^2(S_{1,\delta})}$ and the $\lVert u^i \cdot\rVert_{ L_{u}^1 L^2(S_{u,\delta})}$ norm of the right-hand side in the equation above. We now estimate each of the terms on the right-hand side:
	
	\begin{itemize}

		\item For the linear term in $\eta$, there holds
		
		\begin{equation}
		\lVert u^i \nabla^{i+1} \eta \rVert_{L_u^1 L^2(S_{u,\delta})} \lesssim \Big \lVert \frac{1}{u^2} \Big \rVert_{L_u^1}   \lVert u^{i+2}\nabla^{i+1}\psibar \rVert_{L_{\ubar}^\infty L^2(S_{u,\ubar})}  \lesssim \frac{\delta \al}{u}, 
		\end{equation} where we have used the improved bound on $ \lVert u^{i+2}\nabla^{i+1}\psibar \rVert_{L_{\ubar}^\infty L^2(S_{u,\ubar})}$ obtained in Theorem \ref{higherorderenergyestimates}.
		
		\item We then control the second and third terms together. Here, we use the
		estimates derived in Proposition \ref{riccihelpful}. There holds \begin{equation}\begin{split}
		&\hspace{2mm} \Big \lVert \sum_{i_1+i_2 +i_3=i} u^i \nabla^{i_1}\psibar^{i_2+1} \nabla^{i_3}(\psi,\psibar) \Big \rVert_{L_u^1 L^2(S_{u,\delta})}\\ &\lesssim  \sum_{i_1+i_2+i_3=i} \lVert u^{i_1+i_2-1} \nabla^{i_1}\psibar^{i_2+1} \rVert_{L_u^1 L^2(S_{u,\delta})}\cdot\lVert u^{i_3+1} \nabla^{i_3}(\psi,\psibar) \rVert_{L_u^\infty L^\infty(S_{u,\delta})} 
		\\ &\lesssim \frac{\delta \al }{u}\cdot \al \lesssim \frac{\delta a }{u} .
		\end{split}
		\end{equation}
		
		\item Finally, the last term is estimated as follows
		
		\begin{equation}\begin{split}
	\Big \lVert &\sum_{i_1+i_2 +i_3=i-1} u^{i-1} \nabla^{i_1}\psibar^{i_2+1} \nabla^{i_3}\psi \Big \rVert_{L_u^1 L^2(S_{u,\ubar})} \\ \lesssim &\sum_{i_1+i_2+i_3=i-1} \lVert u^{i_1+i_2-1}\nabla^{i_1}\psibar^{i_2+1} \rVert_{L_u^1 L^2(S_{u,\delta})}  \cdot \lVert u^{i_3+1} \nabla^{i_3}\psi \rVert_{L_{u}^{\infty}L^{\infty}(S_{u,\delta})}  \\\lesssim &\frac{\delta \al }{u} \cdot \al \lesssim \frac{\delta a }{u}.
	\end{split}	\end{equation}
	\end{itemize}
	
	\par \noindent
	Recalling that initially we have \[\lVert \nabla^i \chihat_0 \rVert_{ L^{2}(S_{1,\delta})}=0,     \]we get \[  \lVert u^i \nabla^i \chihat \rVert_{L_u^{\infty} L^2(S_{u,\delta})} \lesssim \frac{\delta a}{u}.    \]This concludes the proof for $\chihat$.
\end{proof}

\par\noindent  Notice that Remark \ref{crucialremark} at the start of this Section implies that we have $L^2$ and $L^{\infty}$--control on all derivatives of the Ricci coefficients, which allowed us to close the estimates in the above Proposition. We move on to estimates for $\csigma$ and $\sigma$:

\begin{proposition} \label{csigmaandsigmaincoming}
	There holds \begin{equation*}
	\twoSudelta{ u^{i+1} \nabla^{i} \csigma}+  \twoSudelta{ u^{i+1} \nabla^{i} \sigma} \lesssim \frac{\delta \al}{u}.
	\end{equation*}
\end{proposition}

\begin{proof} We begin with the estimates for $\check{\sigma}$. The fact that $\chihat=0$ on $S_{1,\delta}$ implies that $\csigma = \sigma$ on $S_{1,\delta}$. Consequently, the bounds for $\check{\sigma}$ on $S_{1,\delta}$ are the same as those for $\sigma$, meaning $\lVert \nabla^{k} \csigma \rVert_{L^{\infty}(S_{1,\delta})} \lesssim \delta \al $ for all $k$. We proceed by commuting \[  \nabla_3 \check{\sigma}+\frac{3}{2}\tr{\chibar}\check{\sigma}=-\div \Hodge{\betabar}+\zeta\wedge\betabar -2\etabar\wedge \betabar +\frac{1}{2}\chibarhat \wedge(\nabla\widehat{\otimes}\eta) +\frac{1}{2}\chibarhat \wedge (\eta \widehat{\otimes} \eta). \]  with $i$ angular derivatives. We obtain 
	
	\begin{equation}
	\nabla_3 \nabla^i \csigma  +\frac{3+i}{2}\tr{\chibar}\nabla^i\check{\sigma} =G \end{equation}
	where $G$ is given by 
	\begin{equation}
	G=\sum_{i_1+i_2+i_3=i+1}\nabla^{i_1}\psibar^{i_2+1}\nabla^{i_3}\psibar+\frac{1}{u}\sum_{i_1+i_2+i_3=i-1}\nabla^{i_1}\psibar^{i_2+1}\nabla^{i_3}\check{\sigma} + \text{div} \Hodge{\nabla^i} \betabar:= G_1 +G_2+G_3.
	\end{equation} 
	We now apply proposition \ref{transportprop2} with $\lambda_0 = \frac{i+3}{2}$, so that 
	\[ \lvert u \rvert^{i+2} \twoSudelta{\nabla^i \csigma} \lesssim \lVert\nabla^i \csigma \rVert_{L^2(S_{1,\delta})} + \int_{u}^{1} \lvert u^\prime \rvert^{i+2} \lVert G \rVert_{L^2(S_{u^{\prime} , \delta })} \duprime. \] 
	There holds  \[\int_{u}^1 \lvert u^\prime \rvert^{i+2} \lVert G \rVert_{L^2(S_{u^{\prime} , \delta}) } \duprime \lesssim \sum_{j=1}^3 \int_{u}^1 \lvert u^\prime  \rvert^{i+2} \lVert G_j \rVert_{L^2(S_{u^{\prime} , \delta })} \duprime. \]  We estimate these terms separately.

	\begin{itemize}
		\item There holds \begin{equation} \begin{split}&\hspace{1mm} \big\lVert \sum_{i_1+i_2+i_3=i+1} u^{i+2} \nabla^{i_1}\psibar^{i_2+1} \nabla^{i_3}\psibar  \big\rVert_{L_{u}^1 L^2(S_{u , \delta} )}\\ &\lesssim \sum_{i_1+i_2 +i_3=i+1 } \lVert u^{i_1+i_2 -1} \nabla^{i_1}\psibar^{i_2+1} \rVert_{L_u^1 L^2(S_{u,\delta})}\cdot  \lVert u^{i_3+2} \nabla^{i_3}\psibar \rVert_{L_u^\infty L^{\infty}(S_{u,\delta})} \\ &\lesssim \frac{\delta \cdot \al }{u} \cdot \delta \al \lesssim \frac{\delta^2 a}{u}. \end{split}
		\end{equation}\item For the second term, there holds, recalling that $\csigma = \curl \eta = \nabla \psibar$ schematically, 
		
		\begin{equation}
		\begin{split}
		&\hspace{3mm} \big\lVert \sum_{i_1+i_2+i_3=i-1} u^{i+1}\nabla^{i_1}\psibar^{i_2+1}\nabla^{i_3}\csigma \big\rVert_{L_u^1 L^2(S_{u,\delta })} \\ &\lesssim  \sum_{i_1+i_2+i_3=i-1} \lVert u^{i_1+i_2 -1} \nabla^{i_1}\psibar^{i_2+1} \rVert_{L_u^1 L^2(S_{u,\delta})} \lVert u^{i_3+3} \nabla^{i_3+1}\psibar \rVert_{L_u^\infty L^{\infty}(S_{u,\delta})} \\ &\lesssim \frac{\delta\al}{u}\cdot \delta \al + \frac{\delta \al}{u^2}\cdot u \cdot \delta\cdot \al \lesssim \frac{\delta^2a}{u}. 
		\end{split}
		\end{equation}
		
		\item Finally, we estimate the term involving $\betabar$, 
		\begin{equation} \begin{split} \big\lVert u^{i+2} \nabla^{i+1}\betabar \big\rVert_{L_u^1 L^2(S_{u,\delta})} &= \int_u^1 \lvert u^\prime \rvert^{i+2} \lVert \nabla^{i+1}\betabar \rVert_{L^2(S_{u^{\prime},\delta})} \duprime \\ &\lesssim \left( \int_u^1 \lvert u^{\prime}\rvert^{2i+6}\lVert \nabla^{i+1}\betabar \rVert_{L^2(S_{u^{\prime},\delta})}^2 \duprime     \right)^{\frac{1}{2}} \cdot \left( \int_u^1 \lvert u^{\prime} \rvert^{-2} \duprime  \right)^{\frac{1}{2}} \\  &\lesssim \frac{1}{u^{\frac{1}{2}}} \cdot \big\lVert u^{i+3} \nabla^{i+1}\betabar \big\rVert_{L_u^2 L^2(S_{u,\delta})} \lesssim \frac{1}{u^{\frac{1}{2}}} \cdot \big\lVert u^{i+3} \nabla^{i+1}\betabar \big\rVert_{L^2(\Hbar_{\delta})} \lesssim \frac{\delta^{\frac{3}{2} }a^{\frac{3}{4}}    }{u^{\frac{1}{2}}}. \end{split} \end{equation}
	\end{itemize}
Consequently, the worst term comes from the initial data, so that multiplying by $u^{-1}$ we get

\begin{equation}
\twoSudelta{ u^{i+1} \nabla^i \csigma} \lesssim \frac{\delta \al}{u}. 
\end{equation}For $\sigma$, there holds \[ \csigma = \sigma + \f12 \chibarhat \wedge \chihat.   \] Consequently, we have

\begin{equation} \begin{split}
\twoSudelta{u^{i+1} \nabla^i \sigma} &\lesssim \twoSudelta{u^{i+1} \nabla^i \csigma} + \TwoSudelta{ u^{i+1} \sum_{i_1+i_2=i} \nabla^{i_1}\chihat \nabla^{i_2} \cdot \chibarhat} \\ &\lesssim \frac{\delta \al}{u} +  \sum_{i_1+i_2=i} \twoSudelta{u^{i_1} \nabla^{i_1}\chihat} \cdot \inftySudelta{u^{i_2+1} \nabla^{i_2} \chibarhat} \\ &\lesssim \frac{\delta \al}{u} + \frac{\delta a}{u }\cdot \frac{\delta \al}{u} \lesssim \frac{\delta \al}{u}.
\end{split}
\end{equation}We recall here that we have chosen $\al \leq b$ so that $\delta a \leq u$. Moreover, we have used the improved estimates for $\chihat$ on the incoming cone from Proposition \ref{chihatincomingprop}. This concludes the result for $\sigma$.
\end{proof}\par \noindent We now bound $\beta$.

\begin{proposition}
	There holds  \[  \twoSudelta{u^{i+1} \nabla^i \beta} \lesssim \frac{\delta a }{u}.    \]
\end{proposition}
\begin{proof}
We begin with the Bianchi equation for $\beta$:
\begin{equation} \begin{split}
\nabla_3 \beta + \tr \chibar \beta = &-\nabla K+ \Hodge{\nabla} \check{\sigma}+2
\omegabar \beta-3(\eta K-\Hodge{\eta} \check{\sigma})+\frac{1}{2}(\nabla(\chihat  \cdot \chibarhat)+\Hodge{\nabla}(\hat{\chi}\wedge \chibarhat))\\&-\frac{3}{4}\hsp \eta \hsp \tr \chi \tr\chibar+ \frac{3}{2}(\eta \chihat \cdot \chibarhat+\Hodge{\eta} \chihat \wedge \chibarhat)-\frac{1}{4}(\nabla \tr \chi \tr \chibar + \tr \chi \nabla \tr \chibar). \end{split}
\end{equation}Commuting with $i$ angular derivatives and using the schematic representation, we have

\begin{equation}
\nabla_3 \nabla^i \beta +\frac{i+2}{2}\tr\chibar \nabla^i \betabar= G_i,
\end{equation}where
\begin{equation}\begin{split}
G_i&=     \Hodge{\mathcal{D}}_1\left(\nabla^i \left(K-\frac{1}{u^2}\right),\nabla^i \csigma \right) +  \psibar \nabla^{i+1}\psi + \frac{1}{u} \nabla^{i+1}\tr\chi   + \psi \nabla^{i+1}(\chibarhat,\tr\chibar)\\ &+ \sum_{i_1+i_2+i_3=i} \nabla^{i_1}\psibar^{i_2+1} \nabla^{i_3}\left( K-\frac{1}{u^2},\csigma \right) + \sum_{\substack{i_1+i_2+i_3=i+1\\ i_1,i_3\leq i}} \nabla^{i_1}\psibar^{i_2+1}\nabla^{i_3}\psi \\ 
&+\frac{1}{u}\sum_{i_1+i_2+i_3=i}\nabla^{i_1}\psibar^{i_2+1}\nabla^{i_3}\psi +\frac{1}{u^2}\sum_{i_1+i_2=i}\nabla^{i_1}\psibar^{i_2+1}.
\end{split}\end{equation}Applying Proposition \ref{transportprop2} with $\lambda_0 = \frac{i+2}{2}$, we can bound \[ \lVert u^{i+1} \nabla^i \beta \rVert_{L^2(S_{u,\delta })} \lesssim  \lVert \nabla^i \beta \rVert_{L^2(S_{1,\delta })} +  \lVert u^{i+1} G_i \rVert_{L_u^1 L^2(S_{u,\delta })}. \]

\begin{itemize}
	\item There holds \begin{equation} \begin{split}  &\big\lVert u^{i+1} \nabla^{i+1} \left( K- \frac{1}{u^2}, \csigma \right) \big\rVert_{L_u^{1} L^2(S_{u,\delta}) }  \\ \lesssim    &\big\lVert u^{i+2} \nabla^{i+1} \left( K- \frac{1}{u^2},\csigma \right) \big\rVert_{L_u^{2} L^2(S_{u,\delta}) }  \cdot  \big\lVert \frac{1}{u} \big\rVert_{L_u^2 } \lesssim \frac{\delta\al}{u}, \end{split} \end{equation}where we have used the improved  bounds on $K-\frac{1}{u^2}$ from Proposition \ref{higherorderenergyestimates} as well as the improved estimates on $\csigma$ from the previous proposition.
	
	\item There holds
	
	\begin{equation}\begin{split}
	&\big\lVert u^{i+1} \psibar \nabla^{i+1}\psi \big\rVert_{L_u^{1} L^2(S_{u,\delta}) } \\ \lesssim &\lVert u \hspace{.5mm} \psibar \rVert_{L_u^2 L^{\infty}(S_{u,\delta})} \lVert u^i \nabla^{i+1}\psi \rVert_{L_u^2 L^2(S_{u,\delta})} \\ \lesssim &\frac{\delta \al}{u^{\frac12}} \cdot \frac{\al}{u^{\frac{1}{2}}}  \lesssim \frac{\delta a}{u}.
	\end{split}
	\end{equation}
	\item There holds
	
	\begin{equation}   \lVert u^i \nabla^{i+1}\tr\chi \rVert_{L_u^{1} L^2(S_{u,\delta}) } =    \big\lVert u^i \nabla^{i+1}\left(\tr\chi - \frac{2}{u}\right) \big\rVert_{L_u^{1} L^2(S_{u,\delta}) } \lesssim \frac{\delta a}{u}. \end{equation}
	\item There holds 
	
	\begin{equation}\begin{split}
	&\big\lVert u^{i+1} \psi \nabla^{i+1}\left( \chibarhat,\tr\chibar       \right) \big\rVert_{L_u^{1} L^2(S_{u,\delta}) } \\ \lesssim  &\lVert \psi \rVert_{L_u^{2} L^{\infty}(S_{u,\delta})} \big\lVert u^{i+1}  \nabla^{i+1}\left( \chibarhat,\tr\chibar       \right) \big\rVert_{L_u^{2} L^2(S_{u,\delta}) }\\ \lesssim &\frac{\al}{u^{\frac12}} \cdot \frac{\delta \al}{u^{\frac12}}\lesssim \frac{\delta a}{u}.
	\end{split}
	\end{equation}
	\item There holds
	\begin{equation}
	\begin{split}
	&\big\lVert u^{i+1} \sum_{i_1+i_2+i_3=i} \nabla^{i_1}\psibar^{i_2+1}\nabla^{i_3}\left(K-\frac{1}{u^2},\csigma \right) \big\rVert_{L_u^{1}L^2(S_{u,\delta})} \\ &\lesssim \sum_{i_1+i_2+i_3=i}\lVert u^{i_1+i_2} \nabla^{i_1}\psibar^{i_2+1} \rVert_{L_u^{2}L^{\infty}(S_{u,\delta})} \big\lVert u^{i_3+1} \nabla^{i_3}\left(K-\frac{1}{u^2},\csigma\right) \big\rVert_{ L_u^{2}L^{2}(S_{u,\delta})}\\ &\lesssim \frac{\delta \al}{u} \cdot \delta^{\frac{1}{2}}\al  \lesssim \frac{\delta^{\frac{3}{2}}a}{u}.
	\end{split}
	\end{equation}
	
	\item There holds 	\begin{equation}\begin{split}&\big\lVert u^{i} \sum_{i_1+i_2+i_3=i} \nabla^{i_1}\psibar^{i_2+1}\nabla^{i_3} \psi \big\rVert_{ L_u^{1}L^2(S_{u,\delta})} \\ &\lesssim \sum_{i_1+i_2 +i_3= i} \lVert u^{i_1+i_2} \nabla^{i_1}\psibar^{i_2+1} \rVert_{ L_u^{2}L^2(S_{u,\delta})} \big\lVert u^{i_3} \nabla^{i_3} \psi \big\rVert_{ L_u^{2}L^{\infty}(S_{u,\delta})}\\ &\lesssim \frac{\delta \al}{u^{\frac12}} \cdot \frac{\al}{u^{\frac12}} \lesssim \frac{\delta a}{u}.\end{split} \end{equation}
	
	\item There holds
	
	\begin{equation}
	\sum_{i_1+i_2=i}\lVert u^{i_1+i_2-1} \nabla^{i_1}\psibar^{i_2+1} \rVert_{ L_u^1 L^2(S_{u,\delta})} \lesssim \frac{\delta \al}{u}.
	\end{equation}\end{itemize}
\par \noindent Consequently, there holds \begin{equation}
\twoSudelta{u^{i+1} \nabla^i \beta} \lesssim \frac{\delta a}{u}.
\end{equation}
\end{proof}
We proceed to control the term $\tr\chi- \f2u + \frac{4m_0}{u^2}$.

\begin{proposition}
	There holds  \[ \TwoSudelta{u^{i+1} \nabla^i \left(\tr\chi - \f2u + \frac{4m_0}{u^2}\right) } \lesssim \frac{\delta a}{u^{\f12}} + \frac{\delta \al}{u} \lesssim \frac{\dal \al}{u^{\f12}}. \]
\end{proposition}

\begin{proof}
	 Recall the transport equation
	\begin{equation}
	\nabla_3 tr\chi + \frac{1}{2}\hspace{.5mm} tr\chi \hspace{.5mm} tr\chibar = 2\hspace{.5mm} \omegabar \hspace{.5mm} tr\chi +2 \hspace{.5mm} \rho - \chihat\cdot \chibarhat + 2 \hspace{.5mm} \text{div}\hspace{.5mm}  \eta + 2\hspace{.5mm} \lvert \eta \rvert^2.
	\end{equation}	
	We can rewrite this using the Gauss equation $K=-\rho+\frac{1}{2}\chibarhat\cdot \chihat-\frac{1}{4}\tr{\chi}\tr{\chibar}$ to obtain 
	\begin{equation}
	\nabla_3 tr\chi + \hspace{.5mm} tr\chibar \hspace{.5mm} \tr\chi = - 2K+2\hspace{.5mm} \omegabar \hspace{.5mm} tr\chi + 2 \hspace{.5mm} \text{div}\hspace{.5mm}  \eta + 2\hspace{.5mm} \lvert \eta \rvert^2.
	\end{equation}	
	There holds
	\begin{equation}\begin{split}
	&\nabla_3\left( \tr\chi - \frac{2}{u}  + \frac{4m_0}{u^2}\right) + \tr\chibar \left(\tr\chi - \frac{2}{u}  + \frac{4m_0}{u^2}\right)  \\= &- 2\left(K-\frac{1}{u^2}\right) - \frac{2}{u}\left(\tr\chibar+ \frac{2}{u} \right) + \frac{2(1-\Omega^{-1})}{u^2} + \frac{4m_0}{u^2}\left(\tr\chibar+ \frac{2}{u}\right) \\+ &\frac{8m_0 (\Omega^{-1}-1)}{u^3} +2\hspace{.5mm} \omegabar \hspace{.5mm} tr\chi + 2 \hspace{.5mm} \text{div}\hspace{.5mm}  \eta + 2\hspace{.5mm} \lvert \eta \rvert^2 :=G_0. \end{split}
	\end{equation}Commuting with $i$ angular derivatives, we arrive at
	
	\begin{equation}\begin{split}
	\nabla_3 \nabla^i \left( \tr\chi - \frac{2}{u}  + \frac{4m_0}{u^2}\right) &+ \frac{i+2}{2}\tr\chibar \nabla^i \left( \tr\chi - \frac{2}{u}  + \frac{4m_0}{u^2}\right) \\ &= \sum_{i_1+i_2+i_3=i} \nabla^{i_1}\psibar^{i_2} \nabla^{i_3} G_0 \\ &+ \sum_{i_1+i_2+i_3=i} \nabla^{i_1}\psibar^{i_2} \nabla^{i_3}\left( \tr\chi - \frac{2}{u}  + \frac{4m_0}{u^2}\right)\\ &+ \sum_{\substack{i_1+i_2+i_3=i\\ i_3\leq i-1}} \tr\chibar \nabla^{i_1}\psibar^{i_2}\nabla^{i_3} \left( \tr\chi - \frac{2}{u}  + \frac{4m_0}{u^2}\right) := G_i.
	\end{split} \label{h}
	\end{equation} Proposition \ref{transportprop2} implies that we can control \begin{equation} \lVert u^{i+1} \nabla^{i}\left( \tr\chi-\frac{2}{u}+\frac{4m_0}{u^2} \right) \rVert_{L^2(S_{u,\delta})}  \lesssim \big\lVert \nabla^{i}\left( \tr\chi-\frac{2}{u}+\frac{4m_0}{u^2} \right) \big\rVert_{L^2(S_{1,\delta})} + \big\lVert u^{i+1} G_i \big\rVert_{L_u^1 L^2(S_{u,\delta})}. \end{equation}
	Before we continue, we remark the following: As can be seen from the last two terms in \eqref{h}, the estimates will have to be carried out using induction. To do this, we first obtain estimates of the form 
	
   \[ \big\lVert u^{i+1} \sum_{i_1+i_2+i_3=i} \nabla^{i_1} \psibar^{i_2} \nabla^{i_3} G_0 \big\rVert_{L_u^1 L^2(S_{u,\delta})}\lesssim T(\delta, a,u). \]As soon as we have obtained these estimates, we have essentially obtained the desired bounds for $i=0$.  For $i\geq 1$, we shall then use the inductive assumption 
   
   \[ \twoSudelta{u^{j+1} \nabla^j \left( \tr\chi - \f2u +\frac{4m_0}{u^2} \right)} \lesssim T(\delta, a,u), \hspace{2mm} \forall \hsp j<i,       \]which we will use to estimate the last two terms in \eqref{h}.
	
	\begin{itemize}
		\item There holds \begin{equation}\begin{split}
		&\hspace{4.5mm}\big \lVert \sum_{i_1+i_2+i_3=i} u^{i+1} \nabla^{i_1}\psibar^{i_2}\nabla^{i_3}\left(K- \frac{1}{u^2} \right) \big \rVert_{L_u^1L^2(S_{u,\delta})} \\ &\lesssim \big \lVert  u^{i+1} \nabla^{i}\left(K- \frac{1}{u^2} \right) \big \rVert_{L_u^1L^2(S_{u,\delta})} \\ &+\sum_{i_1+i_2+i_3+1=i} \lVert u^{i_1+i_2+1}\nabla^{i_1}\psibar^{i_2+1} \rVert_{L_u^2L^{\infty}(S_{u,\delta})} \cdot \lVert u^{i_3+1} \nabla^{i_3}\left(K-\frac{1}{u^2} \right) \big \rVert_{L_u^2L^2(S_{u,\ubar})}\\ &\lesssim \big \lVert  u^{i+1} \nabla^{i}\left(K- \frac{1}{u^2} \right) \big \rVert_{L_u^2L^2(S_{u,\delta})} \cdot \lVert 1 \rVert_{L_u^2} \\ &+\sum_{i_1+i_2+i_3+1=i} \lVert u^{i_1+i_2+1}\nabla^{i_1}\psibar^{i_2+1} \rVert_{L_u^2L^{\infty}(S_{u,\delta})} \cdot \lVert u^{i_3+1} \nabla^{i_3}\left(K-\frac{1}{u^2} \right) \big \rVert_{L_u^2L^2(S_{u,\ubar})} \\ &\lesssim \frac{\delta \al}{u^{\f12}}+ \frac{\delta \al}{u^{\f12}} \cdot \frac{\delta \al}{u^{\f12}} \lesssim \frac{\delta \al}{u^{\f12}}. \end{split}
		\end{equation}Here we have made use of the improved estimate $\twoSu{u^{i+2}\nabla^i \left(K - \frac{1}{u^2} \right) } \lesssim \delta \al$ from Theorem \ref{higherorderenergyestimates}.
		\item There holds \begin{equation}\begin{split}
		&\hspace{4.6mm}\Big \lVert \sum_{i_1+i_2+i_3=i} u^{i} \nabla^{i_1}\psibar^{i_2}\nabla^{i_3} \l \trubar \r \Big\rVert_{L_u^1L^2(S_{u,\delta})} \\&\lesssim \Big \lVert  u^{i} \nabla^{i}\left( \trubar \right) \Big \rVert_{L_u^1L^2(S_{u,\delta})} \\ &+\sum_{i_1+i_2+i_3+1=i} \lVert u^{i_1+i_2+1}\nabla^{i_1}\psibar^{i_2+1} \rVert_{L_u^{\infty}L^2(S_{u,\delta})} \cdot \big\lVert u^{i_3} \nabla^{i_3}\left(\trubar \right) \big \rVert_{L_u^1L^{\infty}(S_{u,\delta})}\\  &\lesssim \Big \lVert  u^{i} \nabla^{i}\left( \trubar \right) \Big \rVert_{L_u^2L^2(S_{u,\delta})} \cdot \lVert 1 \rVert_{L_u^2} \\ &+\sum_{i_1+i_2+i_3+1=i} \lVert u^{i_1+i_2+1}\nabla^{i_1}\psibar^{i_2+1} \rVert_{L_u^{\infty}L^2(S_{u,\delta})} \cdot \big\lVert u^{i_3} \nabla^{i_3}\left(\trubar \right) \big \rVert_{L_u^1L^{\infty}(S_{u,\delta})}\\ &\lesssim \frac{\delta \al}{u^{\frac12}} + \delta \al \cdot \frac{\delta \al}{u} \lesssim \frac{\delta \al}{u^{\frac12}} . \end{split}
		\end{equation}
		\item There holds \begin{equation}\begin{split}
		&\hspace{4.6mm}\Big \lVert \sum_{i_1+i_2+i_3=i} u^{i-1} m_0 \nabla^{i_1}\psibar^{i_2}\nabla^{i_3} \l \trubar \r \Big\rVert_{L_u^1L^2(S_{u,\delta})}\\ &\lesssim \Big \lVert \sum_{i_1+i_2+i_3=i} u^{i} \nabla^{i_1}\psibar^{i_2}\nabla^{i_3} \l \trubar \r \Big\rVert_{L_u^2L^2(S_{u,\delta})} \cdot \big\lVert \frac{m_0}{u}\big \rVert_{L_u^{2}}\\ &\lesssim \frac{\delta \al m_0}{u} \lesssim \frac{\delta \al}{u}.
		\end{split}
	    \end{equation}
		
		\item There holds \begin{equation}
		\big\lVert u^{i-1} \sum_{i_1+i_2+i_3=i} \nabla^{i_1}\psibar^{i_2} \nabla^{i_3} \left(1-\Omega^{-1}\right) \big\rVert_{L_u^1L^2(S_{u,\delta})} \\ \lesssim \frac{\delta \al}{u^{\frac12}},
		\end{equation}keeping in mind that $\omega = \frac{1}{2}\partial_{\ubar}\Omega^{-1}$ and that $\Omega^{-1}=1$ on $\Hbar_0$. Indeed, when $i=0$, there holds \[  \big\lVert u^{-1}   \left(1-\Omega^{-1}\right) \big\rVert_{L_u^1L^2(S_{u,\delta})}  \lesssim \lVert 1- \Omega^{-1} \rVert_{ L_{u}^{2}L^{2}(S_{u,\delta})} \cdot \lVert u^{-1} \rVert_{L_u^2 }  \lesssim \delta \al \cdot \frac{1}{u^{\f12}}. \]
		For $i \geq 1$, there holds \begin{equation}\begin{split}
		&\hspace{4.6mm} \big\lVert u^{i-1} \sum_{i_1+i_2+i_3=i} \nabla^{i_1}\psibar^{i_2} \nabla^{i_3} \left(1-\Omega^{-1}\right) \big\rVert_{L_u^1L^2(S_{u,\delta})} \\ &\lesssim \bigg\lVert u^{i-1} \nabla^i \left(\int_{0}^{\ubar} 2 \hs \omega(u,\ubar^{\prime}, \theta^1,\theta^2) \dubarprime \right) \bigg\rVert_{L_u^1L^2(S_{u,\delta})} \\&+ \bigg\lVert u^{i-1} \sum_{i_1+i_2+i_3=i} \nabla^{i_1}\psibar^{i_2} \nabla^{i_3} \left(\int_{0}^{\ubar} 2 \hs \omega(u,\ubar^{\prime}, \theta^1,\theta^2) \dubarprime \right) \bigg\rVert_{L_u^1L^2(S_{u,\delta})} \\ &\lesssim \bigg\lVert u^{i-1}  \int_{0}^{\ubar} \hs \nabla^i \omega(u,\ubar^{\prime}, \theta^1,\theta^2) \dubarprime \bigg\rVert_{L_u^1L^2(S_{u,\delta})} \\&+ \bigg\lVert u^{i-1} \sum_{i_1+i_2+i_3+1=i} \nabla^{i_1}\psibar^{i_2+1}  \left(\int_{0}^{\ubar}\nabla^{i_3} \omega(u,\ubar^{\prime}, \theta^1,\theta^2) \dubarprime \right) \bigg\rVert_{L_u^1L^2(S_{u,\delta})}\\ &\lesssim \lVert u^{-1} \rVert_{L_u^2} \cdot \delta \cdot \lVert u^{i} \nabla^i \omega \rVert_{L_u^2 L^{2}(S_{u,\delta})} \\ &+ \sum_{i_1+i_2+i_3=i} \lVert u^{i_1+i_2}\nabla^{i_1}\psibar^{i_2+1} \rVert_{L_u^2 L^{2}(S_{u,\delta})} \cdot \delta \cdot \lVert u^{i_3} \nabla^{i_3} \omega \rVert_{L_u^2 L^{\infty}(S_{u,\delta})}\\ &\lesssim \frac{1}{u^{\f12}} \cdot \delta \cdot \al + \frac{\delta \al}{u} \cdot \frac{\delta \al}{u^{\f12}} \lesssim \frac{\delta \al}{u^{\f12}}. \end{split} 
		\end{equation}
		
		\item Similarly, one bounds	\begin{equation}\begin{split}
		&\hspace{4.6mm} \lVert u^{i-2} m_0 \nabla^i \left(1-\Omega^{-1}\right) \rVert_{L_u^1L^2(S_{u,\delta})} \\ &\lesssim 	\lVert u^{i-1} \nabla^i \left(1-\Omega^{-1}\right) \rVert_{L_u^2L^2(S_{u,\delta})}\cdot \lVert u^{-1} m_0 \rVert_{L_{u}^{2} L^{\infty}(S_{u,\delta})}  \\& \lesssim \frac{\delta \al m_0}{u}\lesssim \frac{\delta \al}{u}. \end{split}
		\end{equation}	\item There holds \begin{equation}\begin{split}
	 &\hspace{4.6mm}	\big\lVert \sum_{i_1+i_2+i_3+i_4=i} u^{i+1} \nabla^{i_1}\psibar^{i_2}\nabla^{i_3}\omegabar \nabla^{i_4}\tr\chi \big \rVert_{L_u^1 L^2(S_{u,\delta})} \\ &=  \big\lVert \sum_{i_1+i_2+i_3=i} u^{i+1} \nabla^{i_1}\psibar^{i_2+1}\nabla^{i_3}\tr\chi \big \rVert_{L_u^1 L^2(S_{u,\delta})} \\ &\lesssim \sum_{i_1+i_2+i_3=i} \lVert u^{i_1+i_2} \nabla^{i_1}\psibar^{i_2+1} \rVert_{L_u^2 L^2(S_{u,\delta})} \lVert u^{i_3+1}\nabla^{i_3}\tr\chi \rVert_{L_u^{2} L^{\infty}(S_{u,\delta})}\\ &\lesssim \frac{\delta \al}{u^{\f12}}  \cdot \al \lesssim \frac{\delta a}{u^{\f12}}.
		\end{split}\end{equation}
		\item There holds \begin{equation}\begin{split}
	 &\hspace{4.6mm}	\big \lVert u^{i+1} \sum_{i_1+i_2+i_3=i} \nabla^{i_1}\psibar^{i_2}\nabla^{i_3+1}\eta \big\rVert_{L_u^1 L^2(S_{u,\delta})} \\ &= \lVert u^{i+1}\nabla^{i+1}\eta \rVert_{L_u^1 L^2(S_{u,\delta})}+  \lVert u^{i+1} \sum_{i_1+i_2+i_3+1=i} \nabla^{i_1}\psibar^{i_2+1}\nabla^{i_3+1}\eta \big\rVert_{L_u^1 L^2(S_{u,\delta})} \\ &\lesssim\lVert u^{i+1}\nabla^{i+1}\eta \rVert_{L_u^1 L^2(S_{u,\delta})} + \sum_{i-1} \hsp \lVert u^{i_1+i_2-1} \nabla^{i_1}\psibar^{i_2+1} \rVert_{L_u^1 L^2(\S)}\lVert u^{i_3+3}\nabla^{i_3+1} \eta \rVert_{L_u^{\infty}L^{\infty}(\S)}\\  &\lesssim \big \lVert \frac{1}{u} \big\rVert_{L_u^2} \lVert u^{i+2}\nabla^{i+1} \eta\rVert_{L_u^2 L^2(S_{u,\delta})} + \sum_{i-1} \hsp \lVert u^{i_1+i_2-1} \nabla^{i_1}\psibar^{i_2+1} \rVert_{L_u^1 L^2(S_{u,\delta})}\lVert u^{i_3+3}\nabla^{i_3+1} \eta \rVert_{L_u^{\infty}L^{\infty}(S_{u,\delta})} \\ &\lesssim \frac{1}{u^{\frac{1}{2}}}\cdot (\delta \al) + \frac{\delta \al}{u}\cdot \delta \al \lesssim \frac{\delta \al}{u^{\frac12}}. \end{split}
		\end{equation}
		\item There holds
		\begin{equation}
		\begin{split}
		&\hspace{4.6mm} \big \lVert u^{i+1} \sum_{i_1+i_2+i_3+i_4=i} \nabla^{i_1}\psibar^{i_2}\nabla^{i_3
		}\eta \nabla^{i_4} \eta \big\rVert_{L_u^1 L^2(S_{u,\delta})} = \big\lVert u^{i+1} \sum_{i_1+i_2+i_3=i} \nabla^{i_1}\psibar^{i_2+1} \nabla^{i_3}\eta \big\rVert_{L_u^1 L^2(S_{u,\delta}) }\\  &\lesssim \sum_{i_1+i_2+i_3=i} \lVert u^{i_1+i_2-1} \nabla^{i_1}\psibar^{i_2+1} \rVert_{L_{u}^{1}L^2(S_{u,\delta})} \lVert u^{i_3+2}\nabla^{i_3}\eta \rVert_{L_{u}^{\infty} L^{\infty}(S_{u,\delta})} \lesssim \frac{\delta^2 a}{u}.
    	\end{split}
		\end{equation}This concludes the terms of the form $\sum_{i_1+i_2+i_3=i} \nabla^{i_1}\psibar^{i_2} \nabla^{i_3}G_0$. The worst term that has appeared is $\frac{\delta a}{u^{\f12}} + \frac{\delta \al}{u}$. Consequently, the result holds for $i=0$, namely
		
		\[ \twoSudelta{u \hsp \tr\chi - 2 + \frac{4m_0}{u} } \lesssim  \frac{\delta a}{u^{\f12}} + \frac{\delta \al}{u}.    \] We can therefore set $T(a,u,\delta) = \frac{\delta a}{u^{\f12}}$ and make the inductive assumption \begin{equation}
		\TwoSudelta{u^{j+1}\nabla^j \left( \tr\chi - \f2u +\frac{4m_0}{u^2} \right)} \lesssim \frac{\delta a}{u^{\f12}} + \frac{\delta \al}{u},
		\end{equation} for all $j<i$.

		\item There holds 
		
		\begin{equation}\begin{split}
		\hspace{9.2mm} &\big \lVert u^{i+1} \sum_{i_1+i_2+i_3=i} \nabla^{i_1}\psibar^{i_2}\nabla^{i_3
		}\left( \tr\chi- \frac{2}{u} + \frac{4m_0}{u^2} \right) \big \rVert_{ L_{u}^{1}L^2(S_{u,\delta})}\\ &= \big\lVert u^{i+1} \nabla^{i}\left( \tr\chi- \frac{2}{u} + \frac{4m_0}{u^2} \right)  \big\rVert_{L^2(S_{u,\delta})}  \\ &+	\big \lVert u^{i+1} \sum_{i_1+i_2+i_3=i-1} \nabla^{i_1}\psibar^{i_2+1}\nabla^{i_3
		}\left( \tr\chi- \frac{2}{u} + \frac{4m_0}{u^2} \right) \big \rVert_{ L_{u}^{1}L^2(S_{u,\delta})}\\ &\lesssim \big\lVert u^{i+1} \nabla^{i}\left( \tr\chi- \frac{2}{u} + \frac{4m_0}{u^2} \right)  \big\rVert_{L_{u}^{1}L^2(S_{u,\delta})} \\&+ \sum_{\substack{i_1+i_2+i_3=i-1,\\ i_1 > i_3}} \lVert u^{i_1+i_2} \nabla^{i_1}\psibar^{i_2+1}\rVert_{L_u^1 L^2(S_{u,\delta})} \lVert u^{i_3+2} \nabla^{i_3}\left( \tr\chi- \frac{2}{u} + \frac{4m_0}{u^2} \right) \rVert_{ L_{u}^{\infty}L^\infty (S_{u,\delta})}\\ &+ \sum_{\substack{i_1+i_2+i_3=i-1, \\ i_3 > i_1}} \lVert u^{i_1+i_2} \nabla^{i_1}\psibar^{i_2+1}\rVert_{L_u^1 L^{\infty}(S_{u,\delta})} \lVert u^{i_3+2} \nabla^{i_3}\left( \tr\chi- \frac{2}{u} + \frac{4m_0}{u^2} \right) \rVert_{ L_{u}^{\infty}L^2(S_{u,\delta})}\\ &\lesssim \big\lVert u^{i+1} \nabla^{i}\left( \tr\chi- \frac{2}{u} + \frac{4m_0}{u^2} \right)  \big\rVert_{L_{u}^{1}L^2(S_{u,\delta})} + \frac{\delta \al}{u} \cdot T(a,u,\delta)    \end{split}	\end{equation}The first term in the last inequality is handled by Gr\"onwall's inequality. The second term is smaller than $\frac{\delta a}{u^{\f12}}$.
		
		\item We can also bound 
		
		\begin{equation}
		\begin{split}
		&\hspace{4.6mm} \big \lVert u^{i+1} \sum_{i_1+i_2+i_3+1=i}\tr\chibar \nabla^{i_1}\psibar^{i_2+1}\nabla^{i_3
		}\left( \tr\chi- \frac{2}{u} + \frac{4m_0}{u^2} \right) \big \rVert_{ L_{u}^{1}L^2(S_{u,\delta})} \\ &\lesssim 	\big \lVert u^{i+1} \sum_{i_1+i_2+i_3+1=i}\left(\trubar \right) \nabla^{i_1}\psibar^{i_2+1}\nabla^{i_3
		}\left( \tr\chi- \frac{2}{u} + \frac{4m_0}{u^2} \right) \big \rVert_{L_{u}^{1}L^2(S_{u,\delta})}\\   &+ 	\big \lVert u^{i} \sum_{i_1+i_2+i_3+1=i} \nabla^{i_1}\psibar^{i_2+1}\nabla^{i_3
		}\left( \tr\chi- \frac{2}{u} + \frac{4m_0}{u^2} \right) \big \rVert_{L_{u}^{1}L^2(S_{u,\delta})} \\ &\lesssim  \sum_{\substack{i_1+i_2+i_3+1=i, \\ i_1< i_3}} \lVert u^{i_1+i_2+1}\nabla^{i_1} \psibar^{i_2+2} \rVert_{L_u^1 L^{\infty}(\S)} \cdot \bigg\lVert u^{i_3+1}\nabla^{i_3}\left( \tru \right) \bigg\rVert_{L_u^{\infty}L^2(S_{u,\delta})}   \\&+ \sum_{\substack{i_1+i_2+i_3+1=i, \\ i_3< i_1}}\lVert u^{i_1+i_2} \nabla^{i_1}\psibar^{i_2+1} \rVert_{L_u^1 L^{2}(S_{u,\delta})} \cdot \bigg\lVert u^{i_3+1}\nabla^{i_3}\l \tru \r \bigg\rVert_{L_u^{\infty}L^{\infty}(S_{u,\delta})}   \\ &\lesssim \frac{\delta \al}{u} T(a,u,\delta).
	    \end{split}
		\end{equation}
		
	\end{itemize}
	So there holds \be \big\lVert u^{i+1} \nabla^i \left( \tr\chi - \frac{2}{u} + \frac{4m_0}{u^2}   \right) \big\rVert_{L_{\ubar}^{\infty} L_{u}^{\infty} L^2(S_{u,\delta})} \lesssim \frac{\delta a}{u^{\f12}} + \frac{\delta \al}{u} \lesssim  \frac{\dal \al}{u^{\f12}}.   \ee

\end{proof}
\par \noindent We now control $\omega + \frac{m_0}{2u^2}$.

\begin{proposition}
	For $i\geq 1$, there holds \[  \big\lVert u^i \nabla^i \left(\omega + \frac{m_0}{2u^2} \right)\big\rVert_{L^2(S_{u,\delta})} \lesssim \frac{\delta^{\frac{1}{2}}\al}{u^{\frac{1}{2}}}.  \] For $k=0$, there holds  \[  \big\lVert u  \left(\omega + \frac{m_0}{2u^2} \right)\big\rVert_{L^2(S_{u,\delta})} \lesssim \frac{\delta^{\frac{1}{2}}\al}{u^{\frac{1}{2}}}.  \] 
\end{proposition}

\begin{proof}
	We first prove the bound for $i\geq 1$, then for $i=0$. We begin with the schematic equation \[ \nabla_3 \omega = K+ \psi \hspace{.5mm} \psibar + \psibar \hspace{.5mm} \psibar + \tr\chi \hspace{.3mm} \tr\chibar. \]Commuting this equation with $i$ angular derivatives, we obtain \begin{equation}\begin{split}
	\nabla_3 \nabla^i \omega + \frac{i}{2} \tr\chibar \nabla^i \omega &= \sum_{i_1+i_2+i_3=i} \nabla^{i_1}\psibar^{i_2}\nabla^{i_3}K \\ &+\sum_{i_1+i_2=i}\nabla^{i_1}\psibar^{i_2+2} + \sum_{i_1+i_2+i_3=i} \nabla^{i_1}\psibar^{i_2+1}\nabla^{i_3}\psi \\ &+ \sum_{i_1+i_2+i_3=i-1} \frac{1}{u} \nabla^{i_1}\psibar^{i_2+1}\nabla^{i_3} \psi + \frac{1}{u}\nabla^i \tr\chi := G_i.
	\end{split}
	\end{equation} Using Proposition \ref{transportprop2} with $\lambda_0 = \frac{i}{2}$ and given that the initial data for $u^{i-1} \nabla^i \left(\omega + \frac{m_0}{2u^2}\right)$ for $i \geq 1$ vanish, we conclude that we can control \[\big\lVert u^{i-1} \nabla^i \left(\omega + \frac{m_0}{2u^2}\right) \big \rVert_{L_u^{\infty} L^2(S_{u,\delta})} \lesssim \big\lVert u^{i-1} G_i \big \rVert_{L_u^{1} L^2(S_{u,\delta})} . \]Define \[  F_i:= \sum_{i_1+i_2=i}\nabla^{i_1}\psibar^{i_2+2} + \sum_{i_1+i_2+i_3=i} \nabla^{i_1}\psibar^{i_2+1}\nabla^{i_3}\psi \\ + \sum_{i_1+i_2+i_3=i-1} \frac{1}{u} \nabla^{i_1}\psibar^{i_2+1}\nabla^{i_3} \psi .   \]We claim that \[ \lVert u^{i-1} F_i \rVert_{L_u^{1}L^2(S_{u,\delta})} \lesssim \frac{\delta^{\frac{1}{2}}\al}{\lvert u \rvert^{\frac32}}.   \]Indeed, there holds \[  \lVert u^{i-1} F_i \rVert_{L_u^{1}L^2(S_{u,\delta})} \lesssim \frac1u   \lVert u^{i} F_i \rVert_{L_u^{1}L^2(S_{u,\delta})}.     \]Moreover, we have \begin{equation}\begin{split}
	&\hspace{4.6mm}\big \lVert u^i \sum_{i_1+ i_2 + i_3 =i} \nabla^{i_1}\psibar^{i_2+1}\nabla^{i_3} (\psibar,\psi) \big\rVert_{L_u^1 L^2(S_{u,\delta})} \\ &\lesssim \sum_{i_1+i_2+i_3=i}\lVert u^{i_1+i_2-1} \nabla^{i_1}\psibar^{i_2+1} \rVert_{L_u^{1} L^2(S_{u,\delta})} \cdot \lVert u^{i_3+1} \nabla^{i_3}(\psibar,\psi) \rVert_{L_u^{\infty} L^{\infty}(S_{u,\delta})} \\ &\lesssim \frac{\delta \al}{u}\left( \al + \frac{\delta \al}{u} \right) \lesssim \frac{\delta a}{u},
	\end{split} \end{equation}while
	\begin{equation}\begin{split}
	&\hspace{4.6mm} \big \lVert u^{i-1} \sum_{i_1+ i_2 + i_3 =i-1} \nabla^{i_1}\psibar^{i_2+1}\nabla^{i_3} \psi \big\rVert_{L_u^1 L^2(S_{u,\delta})} \\ &\lesssim \sum_{i_1+i_2+i_3=i-1}\lVert u^{i_1+i_2-1} \nabla^{i_1}\psibar^{i_2+1} \rVert_{L_u^{1} L^2(S_{u,\delta})} \cdot \lVert u^{i_3+1} \nabla^{i_3}\psi \rVert_{L_u^{\infty} L^{\infty}(S_{u,\delta})} \\ &\lesssim \frac{\delta \al}{u}\left( \al + \frac{\delta \al}{u} \right) \lesssim \frac{\delta a}{u}. \end{split}\end{equation}We now focus on the two remaining terms, the one involving $K$ and the one with $\frac{1}{u}\nabla^i \tr\chi$. We
	first estimate the term containing the Gauss curvature. We split it as follows:
	
	\[  K = \left( K- \frac{1}{u^2}  \right) +\frac{1}{u^2}  .  \]For the term involving $K-\frac{1}{u^2}$, when $i_2=0$, we have \begin{equation}
	\big\lVert u^{i-1}\nabla^i \left( K- \frac{1}{u^2}  \right)\big\rVert_{L_u^1 L^2(S_{u,\delta})} \lesssim 	\big\lVert u^{i+1}\nabla^i \left( K- \frac{1}{u^2}  \right)\big\rVert_{L_u^2 L^2(S_{u,\delta})} \cdot \lVert u^{-2} \rVert_{L_u^2} \lesssim \frac{\dal \al}{u^{\frac32}}.
	\end{equation}In the above we have used the improved bounds on $\mathcal{R}$ by Theorem \ref{higherorderenergyestimates}. For $i_2 \geq 1$, we have the following improved bound \begin{equation} \begin{split}
	&\hspace{4.6mm}\big\lVert \sum_{i_1+i_2 +i_3=i-1} u^{i-1}\nabla^{i_1}\psibar^{i_2+1} \nabla^{i_3} \left(K	- \frac{1}{u^2}  \right)\big\rVert_{L_u^1 L^2(S_{u,\delta})}\\ &\lesssim \sum_{i_1+i_2+i_3 = i-1} \lVert u^{i_1+i_2+2}\nabla^{i_1}\psibar^{i_2+1} \rVert_{L_u^{\infty}L^{\infty}(S_{u,\delta})} \cdot \big\lVert u^{i_3+1}\nabla^{i_3}\left( K- \frac{1}{u^2}\right) \big\rVert_{L_u^2 L^{2}(S_{u,\delta})} \cdot \lVert u^{-3} \rVert_{L_u^2} \\ &\lesssim \delta \al \cdot \dal \al \cdot \frac{1}{u^{\f52}} \lesssim \frac{\delta^{\f32}a}{u^{\f52}}.
	\end{split}\end{equation} We now examine the contribution arising from the term $\frac{1}{u^2}$. The only possibility in this case for having $i_2=0$ is $i=0$, which is excluded as we are working with $i\geq 1$. For $i_2\geq 1$, we have \begin{equation}\begin{split}
	&\hspace{4.6mm}\big \lVert \sum_{i_1+i_2=i-1} u^{i-1}\nabla^{i_1}\psibar^{i_2+1} \frac{1}{u^2}\big \rVert_{L_u^1 L^2(S_{u,\delta})}\\ &\lesssim \sum_{i_1+i_2 \leq i-1} \lVert u^{i_1+i_2+1} \nabla^{i_1}\psibar^{i_2+1} \rVert_{L_u^{\infty} L^2(S_{u,\delta})} \cdot \lVert u^{-3} \rVert_{L_u^1} \lesssim \frac{\delta \al}{u^2}.
	\end{split}\end{equation}We now estimate the remaining term, given that $i\geq 1$,  by \be  \lVert u^{i-2}\nabla^i \tr\chi \rVert_{L_u^1 L^2(S_{u,\delta})} = \big\lVert u^{i-2}\nabla^i \left(\tr\chi - \frac2u \right) \big\rVert_{L_u^1 L^2(S_{u,\delta})}  \lesssim \frac{\delta a}{u^2}.  \ee Multiplying the above estimates by $u$, we arrive at 
	
	\be \sum_{k \geq 1} \big \lVert u^k \nabla^k \left(\omega+ \frac{m_0}{2 u^2}\right) \big \rVert_{L^2(S_{u,\delta})} \lesssim \frac{\dal \al}{u^{\frac{1}{2}}}.      \ee We finally focus on $k=0$.  Recall that there holds \[ \nabla_3 \hspace{.5mm} \omega = 2 \hspace{.5mm}\omega \hspace{.5mm} \omegabar - \eta \cdot \etabar +\frac12 \lvert \eta \rvert^2 + \frac{1}{2}\rho.    \]In particular, this implies \begin{equation} \label{omegapluseq}
	\nabla_3 \left( \omega + \frac{m_0}{2u^2} \right) =  2 \hspace{.5mm}\omega \hspace{.5mm} \omegabar - \eta \cdot \etabar +\frac12 \lvert \eta \rvert^2 + \frac{1}{2}\left(\rho + \frac{2 m_0}{u^3}\right) + \frac{\left(1- \Omega^{-1}\right) m_0}{u^3}.
	\end{equation}Notice the following identity:
	
	\begin{equation} \begin{split}
	\rho+ \frac{2m_0}{u^3}= &- \left(K- \frac{1}{u^2}\right) - \frac{1}{4}\left(\tr\chi- \frac{2}{u} +\frac{4m_0}{u^2} \right) \left(\tr\chibar + \frac{2}{u} \right) + \frac{1}{2}\chihat \cdot \chibarhat + \frac{1}{2u} \left(\tr\chi- \frac{2}{u} +\frac{4m_0}{u^2} \right) \\ &- \frac{1}{2u}\left( \tr\chibar + \frac2u \right) + \frac{m_0}{u^2}\left( \tr\chibar + \frac2u \right).
	\end{split}\end{equation} Looking at \eqref{omegapluseq} and using Proposition \ref{transportprop2}, we can bound  $\lVert u^{-1} \left( \omega+ \frac{m_0}{2u^2} \right) \rVert_{ L^2(S_{u,\delta})}$ by the initial data $\lVert  \omega+ \frac{m_0}{2}  \rVert_{L^2(S_{1,\delta})}$  and the $\lVert u^{-1} \cdot   \rVert_{L_u^{1} L^2(S_{u,\delta})}$--norm of the right-hand side of \eqref{omegapluseq}.
	
	\begin{itemize}
		\item There holds \begin{equation}
		\lVert u^{-1} \omega \hspace{.5mm} \omegabar   \rVert_{ L_u^{1} L^2(S_{u,\delta})} \lesssim \int_{u}^{1} \lvert u^{\prime} \rvert^{-1} \lVert \omega \rVert_{L^{\infty}(S_{u^{\prime},\delta})}  \lVert \omegabar \rVert_{L^{2}(S_{u^{\prime},\delta})} \duprime \lesssim \frac{\delta \hspace{.5mm}a}{u^2}.
		\end{equation}\item There holds
		\begin{equation}
		\lVert u^{-1} \eta \hspace{.5mm} \etabar   \rVert_{ L_u^{1} L^2(S_{u,\delta})} \lesssim \int_{u}^{1} \lvert u^{\prime} \rvert^{-1} \lVert \eta \rVert_{L^{\infty}(S_{u^{\prime},\delta})}  \lVert \etabar \rVert_{L^{2}(S_{u^{\prime},\delta})} \duprime \lesssim \frac{\delta^2  a}{u^3}.
		\end{equation}
		\item Similarly, there holds \begin{equation}
		\big	\lVert u^{-1} \lvert \eta \rvert^2  \big\rVert_{L_u^{1} L^2(S_{u,\delta})} \lesssim \frac{\delta^2 a}{u^3}.
		\end{equation}
		
		\item There holds \begin{equation}
		\bigg\lVert\frac{\left(1- \Omega^{-1}\right) m_0}{u^4} \bigg\rVert_{L_u^{1} L^2(S_{u,\delta})} \lesssim \frac{\delta \al m_0}{u^3} .
		\end{equation}
		\item There holds \begin{equation}
		\bigg	\lVert u^{-1} \left(\K \right)  \bigg\rVert_{L_u^{1} L^2(S_{u,\delta})} \lesssim \lVert u^{-2} \rVert_{L_u^2} \cdot \bigg\lVert u \left(\K \right)  \bigg\rVert_{L_{\ubar}^{\infty} L_u^{2} L^2(S_{u,\delta})} \lesssim \frac{\dal\al}{u^{\frac{3}{2}}}. 
		\end{equation}
		\item There holds \begin{multline}
		\bigg	\lVert u^{-1} \left(\tru \right) \left( \trubar \right) \bigg\rVert_{ L_u^{1} L^2(S_{u,\delta})}\\ \lesssim \int_{u}^{1} \big\lVert u^{-1} \big\rVert_{L^{\i}(S_{u,\delta})} \cdot \big\lVert \trubar \big\rVert_{L^{\i}(S_{u,\delta})} \cdot \big\lVert \tru \big\rVert_{L^2(S_{u,\delta})} \duprime \lesssim \int_{u}^{1} \lvert u^{\prime} \rvert^{-1}\cdot \frac{\delta \al}{\lvert u^{\prime} \rvert^2} \cdot \frac{\al}{b \cdot u^{\prime}} \lesssim \frac{\delta a}{b u^3}
		\end{multline}
		\item There holds \begin{multline}
		\lVert u^{-1}\chihat\cdot \chibarhat \rVert_{L_u^1 L^2(S_{u,\delta})}\\  \lesssim \int_{u}^{1} \lVert u^{-1} \rVert_{L^{\infty}(S_{u,\delta})} \cdot \lVert \chihat \rVert_{L^{\i}} \cdot \lVert \chibarhat \rVert_{L^2(S_{u,\delta})} \duprime \lesssim \int_{u}^{1} \lvert u^{\prime} \rvert^{-1} \cdot \frac{\al}{\uprime} \cdot \frac{\delta \al}{\uprime} \duprime \lesssim \frac{\delta a}{u^2}.
		\end{multline}
		
		\item There holds \begin{equation}
		\big \lVert u^{-2} \left(\tru \right) \big\rVert _{L_u^1 L^2(S_{u,\delta})} \lesssim\int_u^1  \frac{\dal \al}{u^{\frac{7}{2}}} \duprime \lesssim  \frac{\dal \al}{u^{\frac{5}{2}}}
		\end{equation} 
		
		\item There holds \begin{equation}
		\big\lVert u^{-2} \left(\trubar \right) \big\rVert _{L_u^1 L^2(S_{u,\delta})} \lesssim\int_u^1  \frac{\delta \al}{u^{3}}\duprime \lesssim  \frac{\delta\al}{u^2} . \end{equation}
		\item There holds \begin{equation}
		\big\lVert u^{-3} m_0 \left(\trubar \right) \big\rVert _{L_u^1 L^2(S_{u,\delta})} \lesssim\int_u^1  \frac{\delta \al \hs m_0}{u^{4}}\duprime \lesssim  \frac{\delta\al \hs m_0}{u^3} . \end{equation} 
	\end{itemize}Multiplying all the above estimates by $i$ and putting them together, there holds \begin{equation} \label{omeganewbound}
	\big \lVert u \hsp \omega + \frac{m_0}{2u} \big\rVert_{L_{\ubar}^{\i} L_u^{\infty}L^2(S_{u,\delta})} \lesssim \frac{\dal \al}{u^{\frac{1}{2}}}.
	\end{equation} \end{proof} \begin{remark}
	Notice that we need an extra $u$ in \eqref{omeganewbound}, but this will not affect our construction of the Transition Region.
\end{remark}
We proceed with $\rho+\frac{2m_0}{u^3}$.

\begin{proposition}
There holds

\[\TwoSudelta{u^{k+2} \nabla^k \left( \rho + \frac{2 m_0}{u^3}\right) } \lesssim \frac{\dal \al}{u^{\f12}}. \]
\end{proposition}

\begin{proof}
    Notice the following identity
    
    \begin{equation}
        \begin{split}
            \rho+ \frac{2m_0}{u^3}= &- \left(K- \frac{1}{u^2}\right) - \frac{1}{4}\left(\tr\chi- \frac{2}{u} +\frac{4m_0}{u^2} \right) \left(\tr\chibar + \frac{2}{u} \right) + \frac{1}{2}\chihat \cdot \chibarhat + \frac{1}{2u} \left(\tr\chi- \frac{2}{u} +\frac{4m_0}{u^2} \right) \\ &- \frac{1}{2u}\left( \tr\chibar + \frac2u \right) + \frac{m_0}{u^2}\left( \tr\chibar + \frac2u \right).
        \end{split}
    \end{equation}Multiplying both sides of the equation by $u^{k+2}\nabla^k$, we arrive at the following
    
    \begin{itemize}
        \item There holds \be \TwoSudelta{ u^{k+2}\nabla^k \left(K - \frac{1}{u^2} \right) } \lesssim \delta \al . \ee
        
        \item There holds 
        \be \begin{split}
             &\hspace{4.6mm}\TwoSudelta{ u^{k+2} \sum_{k_1+k+2=k} \nabla^{k_1}\left( \tru \right) \cdot \nabla^{k_2} \left( \trubar \right) } \\ &\lesssim \sum_{k_1+k_2=k} \twoSudelta{ u^{k_1+1} \nabla^{k_1} \left( \tru \right)} \cdot \lVert u^{k_2+1} \nabla^{k_2} \left( \trubar \right) \rVert_{L^{\infty}(S_{u,\delta})} \\&\lesssim \left( \frac{\delta a}{u^{\f12}} + \frac{\delta \al}{u} \right) \cdot \frac{\delta \al}{u}. \end{split} \ee
             
             \item Similarly, we can prove that the remaining terms, when controlled in the $L^2(S_{u,\delta})$--norm, are bounded above by $\frac{\dal \al}{u^{\f12}}$.
    \end{itemize}The result follows.
\end{proof}\par \noindent We finally bound $\alpha$ and $\alphabar$. We begin with estimates for $\alphabar$ on $\underline{C}_{\delta}$. 
\begin{proposition}
There holds

\[ \twoSudelta{u^{i+1} \nabla^i \alphabar} \lesssim \frac{\delta a}{u}. \]
\end{proposition}

\begin{proof}
Recall the Bianchi equation for $\alphabar$:

\[ \nabla_4 \alphabar + \f12 \hsp \tr\chi \hsp \alphabar = - \nabla \otimes \betabar + 4\hsp \omega\hsp \alphabar- 3(\chibarhat \hsp \rho - \Hodge{\chibarhat}\hsp \sigma) + (\zeta-4\hsp \etabar)\widehat{\otimes}\betabar. \]Using the constraint equation

\[ \betabar = \div \chibarhat - \f12 \nabla \hsp \tr\chibar -\f12\hsp (\eta-\etabar)\cdot \left(\chibarhat- \f12 \tr\chi \right), \]an application of Proposition \ref{commutationformulaeprop} yields 

\begin{equation} \label{alphabarcommutationformula}
    \begin{split}
        \nabla_4 \nabla^i \alphabar &= \sum_{i_1+i_2+i_3+i_4=i} \nabla^{i_1} \psibar^{i_2}\nabla^{i_3}\psi \nabla^{i_4}\alphabar \\ &+ \sum_{i_1+i_2+i_3+i_4=i} \nabla^{i_1} \psibar^{i_2}\nabla^{i_3} \chibarhat \nabla^{i_4}(\rho,\sigma) \\  &+ \sum_{i_1+i_2+i_3+i_4=i} \nabla^{i_1} \psibar^{i_2}\nabla^{i_3} (\eta,\etabar) \nabla^{i_4+1}(\chibarhat,\tr\chibar) \\ &+  \sum_{i_1+i_2+i_3+i_4+i_5=i} \nabla^{i_1} \psibar^{i_2}\nabla^{i_3} (\eta,\etabar)\nabla^{i_4}(\eta, \etabar) \nabla^{i_5} (\chibarhat, \tr\chi)\\ &+ \sum_{i_1+i_2+i_3=i} \nabla^{i_1}\psibar^{i_2}\nabla^{i_3+2}(\chibarhat,\tr\chibar) \\  &+ \sum_{i_1+i_2+i_3+i_4=i} \nabla^{i_1} \psibar^{i_2}\nabla^{i_3} (\chibarhat,\tr\chibar)\nabla^{i_4+1} (\eta,\etabar) .
    \end{split}\end{equation}Notice, furthermore, that $\nabla^i \alphabar = 0$ on the initial incoming cone, because it is a Minkowski null cone. We shall prove the estimates inductively. For $i=0$, using Proposition \ref{transportprop1}, we have
    
\begin{equation}
    \begin{split}
     \twoSudelta{\alphabar} \lesssim \int_{0}^{\delta} \twoSubarprime{(\omega, \tr\chi) \hspace{.5mm
    } \alphabar} \dubarprime &+\int_{0}^{\delta} \twoSubarprime{ \nabla^2(\chibarhat,\tr\chibar)} \dubarprime \\ &+ \int_{0}^{\delta} \twoSubarprime{(\eta,\etabar)\nabla(\chibarhat,\tr\chibar)}\dubarprime \\&+ \int_{0}^{\delta} \twoSubarprime{(\chibarhat,\tr\chibar)\nabla(\eta, \etabar)}\dubarprime  \\&+ \int_{0}^{\delta} \twoSubarprime{\chibarhat \cdot (\rho,\sigma)}\dubarprime  \\&+ \int_{0}^{\delta} \twoSubarprime{(\eta,\etabar)\cdot (\eta,\etabar) \cdot(\chibarhat,\tr\chibar)}\dubarprime. 
    \end{split}
\end{equation}We treat each term separately.\begin{itemize}
        \item There holds $\lVert (\omega, \tr\chi ) \hsp \alphabar \rVert_{L_{\ubar}^{1} L^2(S)} \lesssim \lVert (\omega, \tr\chi) \rVert_{L^{\infty}(S)} \lVert \alphabar \rVert_{L_{\ubar}^1 L^2 (S)} \lesssim \frac{\al}{u} \cdot \lVert  \alphabar \rVert_{L_{\ubar}^{1} L^2(S)}. $ This term is thus absorbed to the left by Gr\"onwall's inequality.
        \item There holds \[\lVert\nabla^2 (\chibarhat, \tr\chibar) \rVert_{L_{\ubar}^{1} L^2(S)} \lesssim \delta \hsp \lVert\nabla^2 (\chibarhat, \tr\chibar) \rVert_{L_{\ubar}^{\infty} L^2(S)} \lesssim \delta \cdot \frac{\delta \hsp \al}{u^3}. \]
        
        \item There holds 
        \[ \lVert (\eta,\etabar) \hsp \nabla(\chibarhat,\tr\chibar) \rVert_{L_{\ubar}^1 L^2(S)} \lesssim    \delta \lVert (\eta, \etabar) \rVert_{L_{\ubar}^{\infty} L^{\infty}(S)} \lVert \nabla ( \chibarhat, \tr\chibar) \rVert_{L_{\ubar}^{\infty} L^{2}(S)}   \lesssim  \frac{\delta^3 a}{u^4}. \]
        \item There holds  \[ \lVert (\chibarhat,\tr\chibar) \hsp \nabla (\eta,\etabar) \rVert_{L_{\ubar}^1 L^2(S)} \lesssim    \delta \lVert \nabla (\eta, \etabar) \rVert_{L_{\ubar}^{\infty} L^{\infty}(S)} \lVert ( \chibarhat, \tr\chibar) \rVert_{L_{\ubar}^{\infty} L^{2}(S)}   \lesssim \delta \cdot \frac{\delta \al}{u^3} \cdot  \]
        
        \item There holds \begin{equation}
            \begin{split}
                \lVert \chihat \cdot (\rho,\sigma) \rVert_{L_{\ubar}^1 L^2(S)} &\lesssim  \lVert \chihat \cdot \rho \rVert_{L_{\ubar}^1 L^2(S)} \\ &+ \lVert \chihat \cdot \nabla \eta \rVert_{L_{\ubar}^1 L^2(S)} + \lVert \chihat \cdot \chibarhat \cdot \chihat \rVert_{L_{\ubar}^1 L^2(S)} \\ &\lesssim \lVert \chihat \cdot \rho \rVert_{L_{\ubar}^1 L^2(S)} \\ &+  \delta\cdot \lVert \chihat \rVert_{L_{\ubar}^{\infty} L^{\infty}(S)} \lVert  \nabla \eta \rVert_{L_{\ubar}^{\infty} L^2(S)} + \delta \lVert \chihat \rVert_{L_{\ubar}^{\infty} L^{\infty}(S)}^2 \lVert \chibarhat \rVert_{L_{\ubar}^{\infty} L^{2}(S)} \\ &\lesssim \lVert \chihat \cdot \rho \rVert_{L_{\ubar}^1 L^2(S)} + \frac{\delta^2 a}{u^3} + \frac{\delta^2 a^{\f32}}{u^3}.
            \end{split}
        \end{equation}For $\chihat \cdot \rho$, we can rewrite this term as \[ \chihat \cdot \rho = \chihat\cdot (K-\f{1}{u^2}) +\f14 \left(\chihat \cdot \tr\chibar \cdot (\tr\chi- \f2u ) + \frac{2 \chihat}{u} ( \tr\chibar + \f2u ) \right) + \f12 \chihat \cdot \chibarhat \cdot \chihat .\]We thus have
     \begin{equation}
         \begin{split}
             \lVert \chihat \cdot \rho \rVert_{L_{\ubar}^1 L^2(S)} & \lesssim \lVert \chihat \rVert_{L_{\ubar}^{2} L^{\infty}(S)} \cdot \big\lVert K -\frac{1}{u^2} \big\rVert_{L_{\ubar}^{2} L^{2}(S)} \\ &+ \f1u \lVert \chihat \rVert_{L_{\ubar}^{\infty} L^{\infty}(S)} \lVert \tr\chibar + \f2u \rVert_{L_{\ubar}^{1} L^{2}(S)} \\ &+ \f1u \lVert \chihat \rVert_{L_{\ubar}^{\infty} L^{\infty}(S)} \lVert \tr\chi - \f2u \rVert_{L_{\ubar}^{1} L^{2}(S)} \\&+ \delta \lVert \chihat \rVert_{L_{\ubar}^{\infty} L^{\infty}(S)}^2 \lVert \chibarhat \rVert_{L_{\ubar}^{\infty} L^{2}(S)} \\ &\lesssim \frac{\delta a}{u^2} + \frac{\delta^2 a}{u^3} + \frac{\delta^2 a^{\f32}}{u^3} ++ \frac{\delta^2 a^{\f32}}{u^3} \lesssim \frac{\delta a}{u^2}. 
         \end{split}
     \end{equation}
     \item Finally, \be \lVert (\eta,\etabar)\cdot (\eta,\etabar)\cdot (\chibarhat, \tr\chibar) \rVert_{L_{\ubar}^1 L^2(S)} \lesssim \delta \lVert (\eta,\etabar ) \rVert^2_{L_{\ubar}^{\infty } L^{\infty }(S)}\lVert (\chibarhat,\tr\chibar ) \rVert_{L_{\ubar}^{\infty} L^2(S)} \lesssim \frac{\delta^3 a}{u^4}.  \ee
    \end{itemize}
Combining the above estimates together we see that there holds

\be \twoSu{ u \alphabar} \lesssim \frac{\delta a}{ u } \lesssim \frac{\dal \al}{u^{\f12}} . \ee A routine induction argument using \eqref{alphabarcommutationformula} yields, for all $i$ and for all $u \geq \f12$, that

\begin{equation} \label{alphabarbounds}
    \twoSu{ u^{i+1}\nabla^i \alphabar} \lesssim \frac{\delta a}{ u } \lesssim \frac{\dal \al}{u^{\f12}} .
\end{equation} \end{proof}
\par \noindent Finally, we move on to estimates for $\alpha$ on the incoming cone $\underline{C}_{\delta}$. 

\begin{proposition}
There holds \[\twoSudelta{u^{i+1} \nabla^i \alpha} \lesssim \frac{\delta a}{u}. \]

\end{proposition}
\begin{proof}
    Recall the Bianchi equation for $\alpha$, given by

\[ \nabla_3 \alpha + \f12 \tr\chibar \alpha =  \nabla \otimes \beta + 4 \omegabar \hsp \alpha - 3 (\chihat \rho + \Hodge{\chihat}\sigma + (\zeta +4 \etabar) \hat{\otimes} \beta .   \] Using Proposition \ref{transportprop2}, there holds \begin{equation}
    \begin{split}
        \lVert \alpha \rVert_{L^2 (S_{u,\delta})} &\lesssim \lVert \nabla \beta \rVert_{L_u^1 L^2(S)} \\ &+ \int_u^{1} \lVert \omegabar \hsp \alpha \rVert_{L^2(S_{u^{\prime}, \delta})} \duprime \\ &+ \int_u^{1} \lVert \chihat \cdot (\rho,\sigma) \rVert_{L^2(S_{u^{\prime}, \delta})} \duprime \\ &+ \int_u^{1} \lVert (\eta, \etabar) \cdot \beta \rVert_{L^2(S_{u^{\prime}, \delta})} \duprime  .
    \end{split}
\end{equation}At this point we crucially use the estimates for $\beta$ obtained in this section. We have \[ \lVert u^{i+1} \nabla^i \beta \rVert_{L^2(S_{u,\delta})} \lesssim \frac{\delta a}{u}. \] 
     \begin{equation} \begin{split}
           \lVert \alpha \rVert_{L^2 (S_{u,\delta})} &\lesssim \int_{u}^1 \frac{\delta a}{\lvert u^{\prime}\rvert^3} \duprime \\&+ \int_u^{1} \lVert \omegabar \rVert_{L^{\infty}(S_{u^{\prime}, \delta})} \lVert \alpha \rVert_{L^2(S_{u^{\prime } , \delta})} \duprime \\&+ \int_u^{1} \lVert (\eta,\etabar) \rVert_{L^{\infty}(S_{u^{\prime}, \delta})} \lVert \beta \rVert_{L^2(S_{u^{\prime } , \delta})} \duprime \\&+ \int_u^{1} \lVert \chihat \rVert_{L^{\infty}(S_{u^{\prime}, \delta})} \lVert (\rho,\sigma ) \rVert_{L^2(S_{u^{\prime } , \delta})} \duprime
          \\ &\lesssim \frac{\delta a}{u^2} + \int_{u}^1 \frac{\delta \al}{\lvert \upr \rvert^2} \lVert \alpha \rVert_{L^2(S_{u^{\prime } , \delta})} \duprime + \int_{u}^1 \frac{\delta \hsp \al}{\lvert \upr\rvert^2} \cdot \frac{\delta a}{\lvert \upr \rvert^2}\duprime \\ &\lesssim \frac{\delta a}{u^2} + \frac{\delta^2 a^{\f32}}{u^3} + \int_{u}^1 \frac{\delta \al}{\lvert \upr \rvert^2} \lVert \alpha \rVert_{L^2(S_{u^{\prime } , \delta})} \duprime + \int_{u}^1 \frac{\delta a}{\upr} \cdot \lVert (\rho,\sigma ) \rVert_{L^2(S_{u^{\prime } , \delta})} \duprime. \end{split}
     \end{equation} By using Gr\"onwall's inequality and the improved estimates on $\chihat$ from Proposition \ref{chihatincomingprop}, we arrive at 
     
     \begin{equation}
         \begin{split}
             \lVert a \rVert_{L^2(S_{u,\delta})} \lesssim \frac{\delta a }{u^2} &+ \int_{u}^1 \frac{\delta a}{\lvert \upr \rvert^2} \cdot \lVert (\rho,\sigma ) \rVert_{L^2(S_{u^{\prime } , \delta})} \duprime \\ \lesssim \frac{\delta a }{u^2} &+ \int_{u}^1 \frac{\delta a}{\lvert\upr\rvert^2} \cdot \left( \lVert \nabla \eta \rVert_{L^2(S_{u^{\prime } , \delta})} + \lVert \chihat \cdot \chibarhat  \rVert_{L^2(S_{u^{\prime } , \delta})} \right)  \duprime \\ &+ \int_{u}^1 \frac{\delta a} {\lvert \upr\rvert^2} \left( \lVert K- \frac{1}{\lvert \upr \rvert^2} \rVert_{L^2(S_{\upr, \delta})} +\lVert \tr\chibar + \frac{2}{\upr} \rVert_{L^2(S_{\upr, \delta})} \right) \duprime \\ &+ \int_{u}^1 \frac{ \delta a}{\lvert \upr\rvert^2} \cdot \frac{1}{\upr} \cdot \lVert \tr\chi - \f{2}{\upr} \rVert_{L^2(S_{\upr, \delta})} \duprime.
         \end{split}
     \end{equation}Putting everything together, there holds
         
         \begin{equation}
             \lVert a \rVert_{L^2(S_{u,\delta})} \lesssim \frac{\delta a}{u^2} + \frac{\delta^2 a^2}{u^3} \lesssim \frac{\delta a}{u^2}.
         \end{equation}A simple induction argument then yields
         
         \begin{equation}
             \lVert u^{i+1}\nabla^i \alpha \rVert_{L^2(S_{u,\delta})} \lesssim \frac{\delta a}{u}.
         \end{equation}

\end{proof}
\begin{remark}
The problem with using $\alpha$ and $\alphabar$ in \cite{AL17} is not that these quantities are singular, as in \cite{LR17}, but rather that they are large. So large, in fact, that with them it would be impossible to prove the existence of the spacetime up to $u= \delta \al b$. The reason we are able to obtain better estimates here is that we have started with $\chihat =0$ and gradually improved every Ricci coefficient and curvature component. At most steps through the incoming cone estimates, we have used the improved estimates on the incoming cone that arose from previous propositions in this section. In other words, the improvement on $\chihat$ and its implications is what ultimately gives us better control on these quantities.
\end{remark}

\subsection{The geometry of the outgoing cone $C_0^{[\delta,\delta+1]}$}

We extend the data onto an outgoing cone, extending the original data, so that $\hat{\chi}=0$ on this extension, now labelled $C_0^{[\delta,\delta+1]}$. $m_0$ is the mass parameter of a Schwarzschild spacetime, and the various curvature and connection terms will be renormalized with respect to the values they take in a certain patch of an $m_0$ Schwarzschild spacetime. The quantities $\eta$, $\etabar$, $\omega$, $\omegabar$, $\csigma_{m_0}$ and $K_{m_0}-\frac{1}{u^2}$ all vanish in any Schwarzschild spacetime.

\vspace{3mm}
 \par \noindent We begin with a few preliminary remarks on setting up the problem.

\begin{remark}
We extend the initial data on a cone $C_0^{[\delta,\delta+1]}$ with $\Omega \equiv 1$ on the cone. 
\end{remark}

\begin{remark}
The fact that $\chihat =0$ and $\omega =0$ on $C_0^{[\delta,\delta+1]}$ implies that, initially, $\text{tr}\chi$ satisfies an equation of the form \[ \nabla_4 \tr\chi = - \frac{1}{2}\left(\tr\chi\right)^2.\]
In particular, $\tr\chi$ is pointwise bounded by a uniform constant on this outgoing cone. Using the identity \[ \frac{\text{d}}{\text{d}\ubar} \int_{\S} f = \int_{\S}\left( \frac{\text{d}f}{\text{d}\ubar} + \Omega \hs \tr\chi\hs f \right) =   \int_{\S} \Omega \left( \text{e}_4(f) + \tr\chi \hs f \right),                      \]plugging in $f = \lvert \phi \rvert_{\gamma}^2$ and using Cauchy-Schwartz on the sphere along with the $L^{\infty}$ bounds on $\Omega$ and $\tr\chi$, we have for $\ubar \geq \delta$ that

\[ \lVert{\phi} \rVert_{L^2(S_{1,\ubar})} \lesssim \lVert \phi \rVert_{L^2(S_{1,\delta})} + \int_{\delta}^{\ubar} \lVert \nabla_4 \phi \rVert_{L^2(S_{1,\ubar^{\prime}})} \dubarprime.\] 
\end{remark}

\begin{remark}
The fact that the areas of the spheres $S_{1,\ubar}$ for $\delta \leq \ubar \leq  \delta +1$ are uniformly bounded above and below by positive constants implies a Sobolev embedding statement:

\[ \lVert \phi \rVert_{L^{\infty}(S_{1,\ubar})} \lesssim \sum_{i=0}^{2}  \lVert \nabla^i \phi \rVert_{L^2(S_{1,\ubar})}. \]
\end{remark}
\begin{proposition}
On the outgoing cone $C_0^{[\delta,\delta+1]}$, we have $\chihat=0$, $\omega=0$, and for all $k$ the following:

\[ \nabla^k \left( \emph{\tr} \chi-\emph{\tr} \chi_{m_0}, \eta, \beta, K-1, \csigma, \omegabar -\omegabar_{m_0}, \emph{\tr} \chibar-\emph{\tr} \chibar_{m_0}, \betabar \right) \lesssim \delta a. \] \label{outgoingconeestimates}	\end{proposition}

\begin{proof}
The proof will proceed via an induction argument on the number $k$ of derivatives. We will obtain these estimates in the order in which they are stated. This reflects how various terms come into the null structure and transport equations. \vspace{3mm} \par \noindent 
Since on the outgoing cone $\ubar\in [\delta,\delta+1]$ we have $\chihat=\alpha=\omega=0$, $\eta=-\etabar$, $\sigma=\csigma$, $\zeta=\eta$ the transport equations in the $\nabla_4$ direction may be written as follows. Note below that we are loose with notation: for `\: $\div$' we write $\nabla$, for $\eta \hat{\otimes} \eta $ we write $\eta^2$, and for $\Hodge{\nabla}\sigma$ we simply write $\nabla \sigma$.
\begin{gather}
    \nabla_4 \tr \chi =-\frac{1}{2}(\tr \chi)^2,\\ 
    \nabla_4 \eta=-2\tr \chi \eta -\beta,\\ 
    \nabla_4 \beta =-2\tr \chi \beta,\\ 
    \nabla_4 K=-\tr \chi K -\nabla \beta - \eta \beta +\frac{1}{2}\tr \chi \nabla \eta +\frac{1}{2}\tr \chi \eta^2,\\
    \nabla_4 \csigma =-\frac{3}{2}\tr \chi \csigma -\nabla \beta + \eta \beta ,\\
     \nabla_4 \omegabar =\frac{3}{2}\eta^2 +\frac{1}{2}\rho,\\
    \nabla_4 \tr \chibar =-\frac{1}{2}\tr \chi \tr \chibar + 2\rho -2\nabla \eta +2\eta^2,\\
    \nabla_4 \chibarhat =-\frac{1}{2} (\tr \chi) \chibarhat - \nabla \eta +\eta^2,\\
    \nabla_4 \betabar = -\tr \chi \betabar -\nabla \rho +\nabla \csigma+2\chibarhat \beta +3\eta(\rho-\csigma),
\end{gather}
Note here on the outgoing cone the constraint equations for $\rho$ and $\beta$ become
\begin{gather}
    \beta= \frac{1}{2}\nabla \tr \chi +\frac{1}{2}\tr\chi \eta,\\
    \rho=-K-\frac{1}{4}\tr \chi \tr \chibar,
\end{gather}
The proof will proceed via an induction argument on the number $k$ of derivatives. Before starting, we note that, since $\eta+\etabar=0$ on the cone, the commutation formula in Proposition 5.12 of [AL17] becomes. 
\begin{proposition} If $\nabla_4 \phi=F_0$ then $\nabla_4 \nabla^i \phi =F_i$ is given by  
\[F_i=\sum_i\nabla^i F_0 + \sum_{i_1+i_2=i}\nabla^{i_1}\chi \nabla^{i_2}\phi + \sum_{i_1+i_2=i-1} \nabla^{i_1}\beta \nabla^{i_2}\phi\]
\end{proposition}

We are now in a position to prove the desired statements.

\begin{lemma}
For all $k$ there holds $\lvert \nabla^{k+1}\tr\chi \rvert + \lvert \nabla^k \eta \rvert \lesssim \delta a$.
\end{lemma}

\begin{proof}
We begin with $k=0$ and with $\tr\chi- \tr\chi_{m_0}$. There holds \[ \nabla_4 (\tr\chi- \tr\chi_{m_0}) = -\frac{1}{2}(\tr\chi+\tr\chi_{m_0})(\tr\chi- \tr\chi_{m_0}). \] Since $\tr\chi+ \tr\chi_{m_0}$ is pointwise bounded on $C_0^{[\delta,\delta+1]}$, there holds

\[\lVert \tr\chi -\tr\chi_{m_0} \rVert_{L^{\infty}(S_{1,\ubar})} \lesssim \lVert \tr\chi - \tr\chi_{m_0} \rVert_{L^{\infty}(S_{1,\delta})} \lesssim \delta a, \]
where in particular we have made use of the sphere estimates from Proposition \ref{spheregeoestimates}. From now on, we shall make the following bootstrap assumption on $\eta$:

\begin{assumption}\label{bootstrapeta}
    There holds $\lvert \nabla^j \eta \rvert \lesssim 1$ for all $0\leq j \leq k$.
\end{assumption}Assume, as an inductive step, that $\lvert\nabla^j (\tr\chi-\tr\chi_{m_0})\rvert \lesssim \delta a$ for all $j < k$. Using Assumption \ref{bootstrapeta}, we notice that the following schematic commutation formula holds:

\begin{multline}
    \nabla_4 \nabla^{k+1} \left(\tr\chi-\tr\chi_{m_0} \right) = \sum_{i_1+i_2=k+1}\nabla^{i_1}(\tr\chi-\tr\chi_{m_0})\nabla^{i_2}(\tr\chi+\tr\chi_{m_0}) + \sum_{i_1+i_2=k+1}\nabla^{i_1}\tr\chi \nabla^{i_2}(\tr\chi-\tr\chi_{m_0}) \\ + \sum_{i_1+i_2+i_3=k} \nabla^{i_1}\tr\chi \nabla^{i_2}\eta \nabla^{i_3}(\tr\chi-\tr\chi_{m_0}).
\end{multline}

Consequently, there holds

\begin{multline}
    \lVert \nabla^{k+1}(\tr\chi-\tr\chi_{m_0}) \rVert_{L^2(S_{1,\ubar})} \lesssim \delta a \\ + \int_{\delta}^{\ubar} \big \lVert \sum_{i_1+i_2=k+1}\nabla^{i_1}(\tr\chi-\tr\chi_{m_0})\nabla^{i_2}(\tr\chi+\tr\chi_{m_0}) \big \rVert_{L^2(S_{1,\ubar^{\prime}})} \dubarprime\\ + \int_{\delta}^{\ubar} \big \lVert \sum_{i_1+i_2=k+1}\nabla^{i_1}\tr\chi\nabla^{i_2}(\tr\chi-\tr\chi_{m_0}) \big \rVert_{L^2(S_{1,\ubar^{\prime}})} \dubarprime\\+ \int_{\delta}^{\ubar} \big \lVert \sum_{i_1+i_2+i_3=k}\nabla^{i_1}(\tr\chi-\tr\chi_{m_0})\nabla^{i_2}\eta \nabla^{i_3}(\tr\chi-\tr\chi_{m_0}) \big \rVert_{L^2(S_{1,\ubar^{\prime}})} \dubarprime.
\end{multline}For the first two terms, whenever either $i_1$ or $i_2$ equals $k+1$, we can use Gr\"onwall's inequality to absorb the corresponding term. When neither of the derivatives are of top order, we bound the term using our inductive step. For the third term, we use the inductive step along with the auxiliary bootstrap assumption on $\eta$ to bound the integral by $\delta a$. The result follows for $\tr\chi$.\vspace{3mm}
\par \noindent For $\eta$, there holds 

\[ \nabla_4 \eta + \frac{1}{2} \tr\chi \eta = \frac{1}{2}\nabla \tr\chi. \] 

Commuting with $i$ angular derivatives, for all $j\leq k$ there holds

\begin{equation}
    \nabla_4 \nabla^{j} \eta = \nabla^{j+1} \tr\chi + \sum_{j_1+j_2=j} \nabla^{j_1}\tr\chi \nabla^{j_2}\eta \\ + \sum_{j_1+j_2+j_3=j-1} \nabla^{j_1} \tr\chi \nabla^{j_2} \eta \nabla^{j_3}\eta.
\end{equation}We see that there holds \begin{multline}
    \lVert \nabla^j \eta \rVert_{L^2(S_{1,\ubar})} \lesssim \lVert \nabla^j \eta \rVert_{L^2(S_{1,\delta})} + \int_{\delta}^{\ubar} \lVert \nabla^{j+1}\tr\chi \rVert_{L^2(S_{1,\ubar^{\prime}})} \dubarprime \\ + \int_{\delta}^{\ubar} \big\lVert \sum_{j_1+j_2=j} \nabla^{j_1}\tr\chi \nabla^{j_2}\eta \big \rVert_{L^2(S_{1,\ubar^{\prime}})} \dubarprime \\ + \int_{\delta}^{\ubar} \big\lVert \sum_{j_1+j_2+j_3=j-1} \nabla^{j_1}\tr\chi \nabla^{j_2}\eta \nabla^{j_3}\eta \big \rVert_{L^2(S_{1,\ubar^{\prime}})} \dubarprime .
\end{multline}Using the recently obtained estimates on $\nabla^{j+1}\tr\chi$ and the Sobolev embedding statement, along with Proposition \ref{spheregeoestimates} for the sphere $S_{1,\delta}$, we can bound $\nabla^{j} \eta$ in $L^{\infty}$ by $\delta a$. The result follows. 
\end{proof}

\par \noindent We move on to estimates for $\beta$. There holds, taking into account that $\omega, \chihat$ and $\alpha$ vanish from $\delta$ to $\delta+1$ initially, 

\[ \nabla_4 \beta = -2 \tr\chi \beta. \] A direct application of Gr\"onwall's inequality implies that $\lVert \beta \rVert_{L^{\infty}(S_{1,\ubar})} \lesssim \delta a$. Upon commutation with $i\geq 1$ angular derivatives, we get

\begin{equation}
    \nabla_4 \nabla^i \beta = \sum_{i_1+i_2=i} \nabla^{i_1}\tr\chi \nabla^{i_2}\beta  + \sum_{i_1+i_2+i_3=i-1} \nabla^{i_1}\tr\chi \nabla^{i_2}\eta \nabla^{i_3}\beta.
\end{equation}Using the estimates on $\tr\chi$ and $\eta$ obtained above as well as Gr\"onwall's inequality, we obtain that $\lvert \nabla^i \beta \rvert \lesssim \delta a$ for all $i$.
\end{proof}\par \noindent We move on to estimates for $K$ and $\csigma$. We have the following structure equation:

\[ \nabla_4 K = -\tr\chi K - \div \beta + \eta \cdot \beta -\frac{1}{2}\tr\chi \div \eta - \frac{1}{2}\tr\chi \lvert \eta \rvert^2, \] 
\[ \nabla_4 K_{m_0} = -\tr\chi_{m_0} K_{m_0}. \]At first instance, the equation for $K$ implies along with the bounds on $\tr\chi, \eta$ and $\beta$ implies the bound  $\lVert K-1 \rVert_{L^{\infty}(S_{1,\ubar})} \lesssim \delta a$, so that $K$ is uniformly bounded above and below. Subtracting those two equations, we arrive at
\[ \nabla_4 \left(K - K_{m_0} \right) = -\left(\tr\chi-\tr\chi_{m_0}\right) K -\tr\chi_{m_0} \left( K - K_{m_0} \right) - \div \beta + \eta \cdot \beta - \frac{1}{2}\tr\chi \text{div} \eta - \frac{1}{2}\tr\chi \lvert \eta \rvert^2. \] Given the fact that $K$ is uniformly bounded, along with the estimates on $\tr\chi-\tr\chi_{m_0}, \tr\chi_{m_0}, \beta$ and $\eta$, we can conclude using Gr\"onwall's inequality that 
\[   \lvert K - K_{m_0} \rvert \lesssim \delta a. \]An induction argument gives us the result for all higher orders of angular derivatives. Furthermore, the fact that $\csigma = \curl \eta$ implies the same bounds for $\csigma$ as those of $\eta$, so that $\lvert \nabla^k \csigma \rvert \lesssim \delta a$.

\vspace{3mm} 

\par \noindent We move on to estimates for $\chibarhat$. There holds
\[ \nabla_4 \chibarhat + \frac{1}{2}\tr\chi \chibarhat = - \nabla \hat{\otimes} \eta +\eta \hat{\otimes}\eta.     \]Using the uniform pointwise bound on $\tr\chi$ and the estimates on all derivatives of $\eta$, we can use the inequality

\[ \lVert \chibarhat \rVert_{L^2(S_{1,\ubar})} \lesssim \lVert \chibarhat \rVert_{L^2(S_{1,\delta})} + \int_{\delta}^{\ubar} \lVert \nabla_{4} \chibarhat  \rVert_{L^2(S_{1,\ubar^{\prime})}} \dubarprime \]   along with a standard application of Gr\"onwall's lemma to conclude that \[ \lVert \chibarhat \rVert_{L^2(S_{1,\ubar})} \lesssim \delta a.  \] An application of the commutation formula offers the same estimate for any number $k$ of angular derivatives acting on $\chibarhat$.    For $\tr\chibar - \tr\chibar_{m_0},$  there holds 

\[ \nabla_4 \left( \tr\chibar - \tr\chibar_{m_0} \right)  + \tr\chi \left(
\tr\chibar - \tr\chibar_{m_0} \right) = -\tr\chibar_{m_0}\left(\tr\chi-\tr\chi_{m_0}\right) +2 \left(K_{m_0}-K \right) - 2 \div \eta + 2\lvert \eta \rvert^2.           \]Using the pointwise bounds on $\tr\chi, \tr\chi-\tr\chi_{m_0}$, $K- K_{m_0}$ and on $\eta$ and its derivatives, an application of Gr\"onwall's inequality yields the desired bound  \[ \lVert \tr\chibar-\tr\chibar_{m_0} \rVert_{L^2(S_{1,\ubar})} \lesssim \delta a.  \]An application of the commutation formula offers the same estimate for any number $k$ of angular derivatives. 

\vspace{3mm}
\par \noindent For $\betabar$, we work with the constraint equation \[ \div\chibarhat = \frac{1}{2}\nabla \tr\chibar - \frac{1}{2}\tr\chibar \hs \eta  + \betabar.  \]Using the estimates on $\eta,\tr\chibar$ and $\chibarhat$ we arrive at the desired conclusion, which can also be obtained for higher numbers of derivatives via the commutation formula.

\vspace{3mm}
\par \noindent Finally, for $\omegabar- \omegabar_{m_0}$, there holds

\[      \nabla_4 \left(\omegabar -\omegabar_{m_0} \right) = \frac{3}{2} \lvert \eta\rvert^2 + \frac{1}{2} (K_{m_0}- K) + \frac{1}{8}(\tr\chi \tr\chibar -\tr\chi_{m_0}\tr\chibar_{m_0}) \] and using the bounds on $\eta, K-K_{m_0}, \tr\chi-\tr\chi_{m_0}, \tr\chibar-\tr\chibar_{m_0}$, the result follows. An application of the commutation formula offers the same estimate for any number of angular derivatives.


\subsection{Construction of the transition region}

We consider a characteristic initial data problem on $H_1^{[\delta,\delta+1]} \cup \Hbar_{\delta}^{[1,\frac{1}{2}]}$, where the data on $H_1^{[\delta,\delta+1]}$ are given by $\chihat=0$ and on the incoming and outgoing cone the data coincide with the estimates in the previous section. We shall continue, by a slight abuse of notation, to use the space-time metric $g$ to denote the solution of this problem.

\vspace{3mm}

\par \noindent Adapted to this framework, we introduce the following norms:

\[  \mathfrak{R}_k^{\ubar}(u) := \big\lVert \nabla^k\left( \alpha, \beta, \rho - \rho_{m_0} , \sigma, \betabar \right) \big\rVert_{L^2\left(H_u^{(\delta,\ubar)}\right)}, \hspace{2mm} \underline{\mathfrak{R}}_k^{\ubar}(u) := \big\lVert \nabla^k \left( \beta, \rho - \rho_{m_0}, \sigma, \betabar, \alphabar \right) \big\rVert_{L^2\left(\Hbar_{\ubar}^{(u,1)}\right)}        \] \[ \mathfrak{O}_k(\ubar,u) = \big\lVert \nabla^k \left( \chihat,\chibarhat, \eta, \etabar, \tr\chi- \tr\chi_{m_0},\tr\chibar- \tr\chibar_{m_0}, \omega- \omega_{m_0},\omegabar-\omegabar_{m_0}  \right)\big\rVert_{L^2(S_{u,\ubar})}.     \]   
\begin{remark}
Here and throughout we fix a maximum number of derivatives $k_{\text{max}}$ and only work with $k \leq k_{\text{max}}$. Since the number of derivatives is now bounded and $u \geq \frac12$ in our setting, every weighted norm involved in Subsection \ref{incomingconesection} is equivalent to its natural unweighted norm. Consequently, in what follows, we will drop the powers of $u$ that we have used so far.
\end{remark}




\par \noindent The results of Subsection  \ref{incomingconesection} and Proposition \ref{outgoingconeestimates} are then summarised in the following.

\begin{proposition} \label{summary}
There exists a small number $\epsilon \approx \delta^{\f12} a^{\f12}$ such that 

\[  \sum_{k \leq k_{\text{max}}} \left( \mathfrak{R}_{k}^{\delta+1}(1) +   \underline{\mathfrak{R}}_{k}^{\delta}(1/2) + \sup_{\delta \leq \ubar \leq \delta+1}  \mathfrak{O}_k(\ubar, 1) + \sup_{1/2\leq u \leq 1 }\mathfrak{O}_k(\delta, u)  \right) \lesssim \epsilon. \]
\end{proposition}Based on this Proposition, we shall prove the following Theorem.

\begin{theorem} \label{transitionthm}
There exists an $\epsilon_0 >0$ which is independent of $\delta,$ for $\delta$ sufficiently small, such that a solution to the Einstein equations exists in the slab \[\widetilde{\mathcal{D}} := \begin{Bmatrix}(u,\ubar) \in [1- \epsilon_0 ,1]\times [\delta, \delta+\epsilon_0 ] \end{Bmatrix}\] with the data satifying the bounds in Proposition \ref{summary}. Moreover, with $\epsilon$ as in Proposition \ref{summary},
there holds 

\[\sum_{k \leq k_{\text{max}}}\left( \mathfrak{R}_k^{\ubar}(u) + \underline{\mathfrak{R}}_k^{\ubar}(u) + \mathfrak{O}_k(\ubar,u) \right) \lesssim \epsilon.\]
\end{theorem}

A theorem of Rendall \cite{R} gives us a local solution in the neighbourhood of the sphere $S_{1,\delta}$. However, at this point, we follow the methodology of \cite{LY15} and make use of a result of Luk \cite{L12}. Apart from the existence in a neighbourhood of the two null cones, Luk's result provides us with a quantitative control on the solution depending on the initial data. In particular, we have the following

\begin{proposition}
There exists a smooth solution $(\mathcal{M}_{\epsilon_0}, g)$ corresponding to the slab $\mathcal{D}$. Moreover, for all $(u,\ubar) \in \mathcal{D}$, there exists a constant $C(\epsilon_0)$ such that there holds holds \begin{equation} \label{startingpointlikebootstrapassumption}
    \sum_{k \leq k_{\text{max}}}\left( \mathfrak{R}_k^{\ubar}(u) + \underline{\mathfrak{R}}_k^{\ubar}(u) + \mathfrak{O}_k(\ubar,u) \right) \lesssim C(\epsilon_0).
\end{equation} Moreover, the following Sobolev inequalities hold in $\mathcal{D}$:

\begin{gather} \label{startingpointSobolev1} \lVert \phi \rVert_{L^4(\S)} \leq C(\epsilon_0) \left( \twoSu{\nabla \phi} + \twoSu{\phi} \right), \\ \label{startingpointSobolev2} \inftySu{\phi} \leq C(\epsilon_0) \sum_{i \leq 2} \twoSu{\nabla^i \phi}.\end{gather}
\end{proposition}

\par \noindent Equations \eqref{startingpointlikebootstrapassumption} and \eqref{startingpointSobolev1}-\eqref{startingpointSobolev2} will be pivotal to what will follow.
\begin{remark}
Throughout this section, by $A \lesssim B$ we shall mean that there exists a constant $C$ depending only on $\epsilon_0$ for $\delta$ sufficiently small, such that $A \leq C \hs B$. 
\end{remark}

\par \noindent We begin by establishing the following lemma, which states that the connection coefficients can be controlled by the initial data and the curvature components.

\begin{lemma} \label{lemmastep}
There holds \[ \mathfrak{O}(\ubar,u) \lesssim  \sup_{\delta \leq \ubar^{\prime} \leq \ubar} \left( \mathfrak{O}(\ubar^{\prime}, 1) + \underline{\mathfrak{R}}^u(\ubar^{\prime}) \right) + \sup_{1\leq \upr \leq u} \left( \mathfrak{O}(\delta,\upr) + \mathfrak{R}^{\ubar}(\upr) \right). \]
\end{lemma}

\begin{proof}
The proof shall be carried out by integrating the null structure equations. We begin with the bounds on $\eta$ and $\etabar$.

\vspace{3mm}
\par \noindent Recall the following structure equations

\begin{gather}
    \nabla_4 \eta = - \left(\chihat +\frac{1}{2}\tr\chi \gamma\right) \cdot (\eta - \etabar) - \beta, \\ \nabla_3 \etabar = - \left( \chibarhat + \frac{1}{2} \tr\chibar \gamma \right) \cdot (\eta - \etabar) - \betabar.
\end{gather}\par \noindent Notice that, according to our notations, \eqref{startingpointlikebootstrapassumption} along with \eqref{startingpointSobolev1}-\eqref{startingpointSobolev2} imply that \[ \inftySu{\chihat} + \inftySu{\tr\hs \chi} \lesssim 1.\]Using this and integrating along the outgoing and incoming directions respectively, an application of Gr\"onwall's inequality yields \begin{gather}
\label{coupledetaone}    \twoSu{\eta} \lesssim \int_{\delta}^{\ubar} \left(
    \twoSubarprime{\etabar} + \twoSubarprime{\beta} \right) \hs \dubarprime, \\ \label{coupledetatwo} \twoSu{\etabar} \lesssim \lVert \etabar \rVert_{L^2(S_{u,0})} + \int_{u}^{1} \left( \twoSuprime{\eta} + \twoSuprime{\betabar} \right) \duprime.
\end{gather}Adding \eqref{coupledetaone} and \eqref{coupledetatwo} together and using Gr\"onwall's inequality,  we arrive at 

\begin{equation} \label{etaandetabarbounds}
    \twoSu{\eta}+\twoSu{\etabar} \lesssim \mathfrak{P}(u,\ubar),
\end{equation}where
\be \mathfrak{P}(u,\ubar) := \sup_{\delta \leq \ubar^{\prime} \leq \ubar} \left( \mathfrak{O}(\ubar^{\prime}, 1) + \underline{\mathfrak{R}}^u(\ubar^{\prime}) \right) + \sup_{u \leq \upr \leq 1} \left( \mathfrak{O}(\delta,\upr) + \mathfrak{R}^{\ubar}(\upr) \right). \ee Let us move on to $\chihat$ and $\chibarhat$.  There holds \[ \nabla_4 \chihat + \tr\chi \chihat = - 2\omega \chihat - \alpha. \]Consequently, 

\begin{equation}
\begin{split}
        \twoSu{\chihat} &\lesssim \lVert \chihat \rVert_{L^2(S_{u,\delta})} + \int_{\delta}^{\ubar} \twoSuubarprime{(\omega, \tr\chi) \chihat} \dubarprime + \int_{\delta}^{\ubar} \twoSuubarprime{\alpha} \dubarprime \\ &\lesssim \mathfrak{O}(\delta, u) + \int_{\delta}^{\ubar} \twoSuubarprime{\chihat} \dubarprime +\mathfrak{R}^{\ubar}(u) .
\end{split}
\end{equation}Using Gr\"onwall's inequality, there holds

 \be \label{chihattransitionregionbound} \twoSu{\chihat} \lesssim \mathfrak{O}(\delta, u) +\mathfrak{R}^{\ubar}(u) \lesssim \mathfrak{P}(u,\ubar). \ee We now continue with the induction hypothesis that for all $j < i$, there holds $\twoSu{\nabla^j \chihat} \lesssim \mathfrak{P}(u,\ubar)$.
Using the commutation formula from Proposition \ref{commutationformulaeprop}, there holds 

\be \begin{split} \nabla_4 \nabla^i \chihat &= \sum_{i_1+i_2 +i_3+i_4=i} \nabla^{i_1}\psi^{i_2} \nabla^{i_3}(\tr\chi, \omega) \nabla^{i_4} \chihat + \sum_{i_1+i_2+i_3=i} \nabla^{i_1}\psi^{i_2}\nabla^{i_3} \alpha \\&+ \sum_{i_1+i_2+i_3+i_4=i}\nabla^{i_1} \psi^{i_2}\nabla^{i_3}(\psi,\chihat) \nabla^{i_4}\chihat. \end{split} \ee

\par \noindent Using Proposition \ref{transportprop1} and bounding all terms apart from those involving $\nabla^{i_4}$  in $L^{\infty}$ by 1 using \eqref{startingpointlikebootstrapassumption} and using the induction hypothesis, we can also show

\begin{equation}
    \twoSu{\nabla^i \chihat} \lesssim \mathfrak{P}(u,\ubar).
\end{equation}

\par \noindent We move on to estimates for $\omegabar- \omegabar_{m_0}$ and $\omega-\omega_{m_0}$. The estimates for these two terms will be obtained in a similar way to those for $\eta$ and $\etabar$, meaning in a coupled way.  We recall the structure equation for $\nabla_4 \omega$, which reads

\[ \nabla_4 \omegabar = 2 \hsp \omega \hsp \omegabar - \eta \cdot \etabar + \f12 \lvert \eta \rvert^2 + \f12 \rho. \]Similarly, there holds \[ \nabla_4 \omegabar_{m_0} = 2\hs 
\omega_{m_0} \hs \omegabar_{m_0} + \frac{1}{2} \rho_{m_0}. \]Combining the two, there holds 

\begin{equation}
    \nabla_4 \left( \omegabar - \omegabar_{m_0} \right) = 2 \hsp \left( \omega \hsp \omegabar - \omega_{m_0} \hsp \omegabar_{m_0} \right) - \eta\cdot \etabar + \f12 \lvert \eta \rvert^2 + \f12 \left( \rho -\rho_{m_0}\right).
\end{equation}We can thus estimate, using $\lVert \omega \hsp \omegabar - \omega_{m_0} \hsp \omegabar_{m_0}  \rVert_{L^2(S_{u,\ubar})}  \lesssim \twoSu{ \omega-\omega_{m_0} } + \twoSu{\omegabar-\omegabar_{m_0}}  $, 

\begin{equation} \begin{split}
    \twoSu{\omegabar-\omegabar_{m_0}} &\lesssim \lVert \omegabar- \omegabar_{m_0} \rVert_{L^2(S_{u,\delta})} + \int_{\delta}^{\ubar} \left(  \twoSuubarprime{ \omega-\omega_{m_0} } + \twoSuubarprime{\omegabar-\omegabar_{m_0}} \right) \dubarprime \\ &+ \int_{\delta}^{\ubar}\left( \twoSuubarprime{\eta \cdot (\eta,\etabar)} + \twoSuubarprime{\chihat\cdot \chibarhat} + \twoSuubarprime{\rho- \rho_{m_0}} \right) \dubarprime. 
    \end{split}
\end{equation}The terms involving a product of two quantities are estimated by the product of one quantity in $L^{\infty}$ and the other in $L^2$. We contol each term separately.

\begin{itemize}
    \item There holds $\lVert \omegabar- \omegabar_{m_0} \rVert_{L^2(S_{u,\delta})} \lesssim \mathfrak{P}(u,\ubar).$

    \item In the term \[\int_{\delta}^{\ubar} \left(  \twoSuubarprime{ \omega-\omega_{m_0} } + \twoSuubarprime{\omegabar-\omegabar_{m_0}} \right) \dubarprime,  \]the second integrand is absorbed to the left by Gr\"onwall's inequality. 
    
    \item There holds \[ \twoSu{\eta \cdot (\eta,\etabar)} \lesssim \inftySu{\eta,\etabar} \cdot \twoSu{\eta} \lesssim 1 \cdot \mathfrak{P}(u,\ubar), \]where we have used \eqref{etaandetabarbounds}.
    
    \item Similarly, there holds  \[ \twoSu{\chihat \cdot \chibarhat} \lesssim \inftySu{\chibarhat} \cdot \twoSu{\chihat} \lesssim \mathfrak{P}(u,\ubar). \]Here we have used \eqref{chihattransitionregionbound}.
    
    \item Finally, there holds \[ \int_{\delta}^{\ubar} \twoSubarprime{\rho-\rho_{m_0}} \dubarprime \lesssim \left( \int_{\delta}^{\ubar} \twoSubarprime{\rho-\rho_{m_0}}^2 \dubarprime \right)^{\f12}  \lesssim \mathfrak{P}(u,\ubar). \]
\end{itemize}Putting all the above estimates together, we conclude that 

\begin{equation}
    \label{omegabartransitionregionbound} \twoSu{\omegabar-\omegabar_{m_0}} \lesssim \int_{\delta}^{\ubar} \twoSuubarprime{\omega-\omega_{m_0}} \duprime + \mathfrak{P}(u,\ubar). 
\end{equation}Through a similar procedure, we have  

\begin{equation}
    \label{omegatransitionregionbound} \twoSu{\omega-\omega_{m_0}} \lesssim \int_{u}^{1} \twoSuprime{\omegabar-\omegabar_{m_0}} \dubarprime + \mathfrak{P}(u,\ubar). 
\end{equation}Coupling \eqref{omegabartransitionregionbound} with \eqref{omegatransitionregionbound} and using Gr\"onwall's inequality, we arrive at \begin{equation}
    \label{omegaomegabarfinalboundtransitionregion} \twoSu{\omega-\omega_{m_0}} + \twoSu{\omegabar-\omegabar_{m_0}} \lesssim \mathfrak{P}(u,\ubar). 
\end{equation}Next, we move on to estimates for $\tr\chi-\tr\chi_{m_0}$ and $\tr\chibar-\tr\chibar_{m_0}$. There holds \begin{gather}
    \nabla_4 \tr\chi + \f12 (\tr\chi)^2 = - \lvert \chihat \rvert^2 -2 \hsp \omega \hsp \tr\chi, \\
    \nabla_4 \tr\chi_{m_0} + \f12 (\tr\chi_{m_0})^2 = - 2\omega_{m_0}\hsp \tr\chi_{m_0}.
\end{gather}Consequently, 

\begin{equation}
    \nabla_4 (\tr\chi- \tr\chi_{m_0}) = -\f12 (\tr\chi+\tr\chi_{m_0})(\tr\chi-\tr\chi_{m_0}) - \lvert \chihat \rvert^2 -2( \omega\tr\chi-\omega_{m_0}\tr\chi_{m_0}).
\end{equation}We notice the following things:

\begin{itemize}
    \item Equations \eqref{startingpointlikebootstrapassumption} and \eqref{startingpointSobolev1}-\eqref{startingpointSobolev2} imply that $\inftySu{\tr\chi+\tr\chi_{m_0}} \lesssim 1$.
    \item From \eqref{chihattransitionregionbound} and \eqref{startingpointlikebootstrapassumption}--\eqref{startingpointSobolev2}, there holds \[\twoSu{\lvert \chihat \rvert^2} \lesssim \inftySu{\chihat}\twoSu{\chihat} \lesssim \mathfrak{P}(u,\ubar).\]
        \item There holds $\twoSu{\omega\tr\chi-\omega_{m_0}\tr\chi_{m_0}} \lesssim \twoSu{\omega-\omega_{m_0}}+ \twoSu{\tr\chi-\tr\chi_{m_0}}$.
\end{itemize}

\par \noindent Integrating along the $\nabla_4$--direction, we have 

\begin{equation}
\begin{split}
    \twoSu{\tr\chi-\tr\chi_{m_0}} &\lesssim \lVert \tr\chi -\tr\chi_{m_0} \rVert_{L^2(S_{u,\delta})} \\ &+ \int_{\delta}^{\ubar} \left(\inftySuubarprime{ \tr\chi+\tr\chi_{m_0}}+1\right) \cdot \twoSuubarprime{\tr\chi-\tr\chi_{m_0}}  \dubarprime \\ &+ \int_{\delta}^{\ubar} \twoSuubarprime{\omega-\omega_{m_0}} \dubarprime + \mathfrak{P}(u,\ubar). 
\end{split}
\end{equation}The first integrand is absorbed to the left by Gr\"onwall's inequality. The second one is bounded by $\mathfrak{P}(u,\ubar)$, from the estimates obtained in \eqref{omegaomegabarfinalboundtransitionregion}. Finally,

\begin{equation}
    \label{trchitransitionregionbound}
    \twoSu{\tr\chi- \tr\chi_{m_0}} \lesssim \mathfrak{P}(u,\ubar). 
\end{equation}Similarly, using the structure equation for $\nabla_3\tr\chibar$, we can show that there holds 

\begin{equation}
    \label{trchibartransitionregionbound}
    \twoSu{\tr\chibar-\tr\chibar_{m_0}}\lesssim \mathfrak{P}(u,\ubar). 
\end{equation}The estimates for a higher number of derivatives are done through an induction argument. We conclude that

\[ \mathfrak{O}(u,\ubar) \lesssim \mathfrak{P}(u,\ubar). \]
\end{proof}

\par \noindent We proceed by rewriting the Bianchi equations to accommodate for the non--trivial curvature component $\rho_{m_0}$. Notice, to begin with, that $\rho_{m_0}$ satisfies the following equations:

\begin{gather}
    \nabla_4 \rho_{m_0} + \f32 \tr\chi_{m_0} \rho_{m_0} =0, \\  \nabla_3 \rho_{m_0} + \f32 \tr\chibar_{m_0} \rho_{m_0} =0.
\end{gather}Using these, we arrive at the following equations:

\begin{gather}
  \label{renormalizedrho1}      \nabla_4 (\rho - \rho_{m_0}) + \f32(\tr\chi \rho - \tr\chi_{m_0} \rho_{m_0}) = \div \beta - \f12 \chibarhat \cdot \alpha + (\zeta+2 \etabar) \cdot \beta , \\ \label{renormalizedrho2} \nabla_3 (\rho - \rho_{m_0}) + \f32(\tr\chibar \rho - \tr\chibar_{m_0} \rho_{m_0}) = - \div \betabar  - \f12 \chihat \cdot \alphabar + (\zeta-2 \etabar) \cdot \betabar . 
\end{gather}Moreover, since $\rho_{m_0}$ is constant on each $S_{u,\ubar}$, we can rewrite the Bianchi equations for $\beta$ and $\betabar$ as 

\begin{gather}
  \label{renormalizedbeta}  \nabla_3 \beta + \tr\chibar \hsp \beta = \nabla (\rho - \rho_{m_0}) + 2 \hsp \omegabar \hsp \beta  +\Hodge{\nabla} \sigma + 2 \chihat \cdot \betabar  + 3 (\eta \hsp \rho + \Hodge{\eta} \sigma), \\ \label{renormalizedbetabar} \nabla_4 \betabar + \tr\chi \hsp \betabar = - \nabla (\rho - \rho_{m_0}) + \Hodge{\nabla} \sigma + 2\omega \betabar + 2 \chibarhat \cdot \beta -3 (\etabar \rho - \Hodge{\etabar} \sigma).
\end{gather} We call $\eqref{renormalizedrho1}$--$\eqref{renormalizedrho2}$ along with \eqref{renormalizedbeta}--\eqref{renormalizedbetabar} and the six original Bianchi equations the \textit{renormalized Bianchi equations}. Using these Bianchi equations,  the following energy inequalities hold for all $(u,\ubar) \in \mathcal{D}$:

\begin{equation}
    \begin{split}
        &\sum_{R \in \lbrace
        \alpha, \beta, \rho-\rho_{m_0}, \sigma, \betabar  \rbrace } \int_{H_u^{(\delta, \ubar)}} \lvert R \rvert^2 \duprime + \sum_{\underline{R} \in \lbrace
         \beta, \rho-\rho_{m_0}, \sigma, \betabar, \alphabar \rbrace} \int_{\Hbar_{\ubar}^{(u,1)}} \lvert \underline{R} \rvert^2 \dubarprime   \\ \lesssim  &\sum_{R \in \lbrace
        \alpha, \beta, \rho-\rho_{m_0}, \sigma, \betabar  \rbrace } \int_{H_1^{(\delta, \ubar)}} \lvert R \rvert^2 \duprime + \sum_{\underline{R} \in \lbrace
         \beta, \rho-\rho_{m_0}, \sigma, \betabar, \alphabar \rbrace} \int_{\Hbar_{\delta}^{(u,1)}} \lvert \underline{R} \rvert^2 \dubarprime \\ + &\sum_{\substack{R \in \lbrace
        \alpha, \beta, \rho-\rho_{m_0}, \sigma, \betabar  \rbrace, \\ \underline{R} \in \lbrace
         \beta, \rho-\rho_{m_0}, \sigma, \betabar, \alphabar \rbrace, \\ \mathcal{R}_1 ,\mathcal{R}_2 \in \lbrace \alpha, \beta ,\rho-\rho_{m_0}, \sigma, \betabar,\alphabar \rbrace}} \iint_{\mathcal{D}(u,\ubar) } \big( \lvert \Gamma \cdot \mathcal{R}_1 \cdot \mathcal{R}_2 \rvert \\ &\hspace{45mm}+ \lvert \rho \rvert \lvert (\tr\chi- \tr\chi_{m_0}) (\rho-\rho_{m_0}) + \chihat \cdot \alpha + \eta \cdot \beta + \chibarhat \cdot \alphabar + \etabar \cdot \betabar \rvert \big) \duprime \dubarprime .
    \end{split}
\end{equation}Moreover, for $i=1,2,3$, the similar energy inequalities hold:

\begin{equation}
    \begin{split}
        &\sum_{R \in \lbrace
        \alpha, \beta, \rho-\rho_{m_0}, \sigma, \betabar  \rbrace } \int_{H_u^{(\delta, \ubar)}} \lvert \nabla^i R \rvert^2 \duprime + \sum_{\underline{R} \in \lbrace
         \beta, \rho-\rho_{m_0}, \sigma, \betabar, \alphabar \rbrace} \int_{\Hbar_{\ubar}^{(u,1)}} \lvert \nabla^i \underline{R} \rvert^2 \dubarprime   \\ \lesssim  &\sum_{R \in \lbrace
        \alpha, \beta, \rho-\rho_{m_0}, \sigma, \betabar  \rbrace } \int_{H_1^{(\delta, \ubar)}} \lvert \nabla^i R \rvert^2 \duprime + \sum_{\underline{R} \in \lbrace
         \beta, \rho-\rho_{m_0}, \sigma, \betabar, \alphabar \rbrace} \int_{\Hbar_{\delta}^{(u,1)}} \lvert \nabla^i  \underline{R} \rvert^2 \dubarprime \\ + &\sum_{\substack{R \in \lbrace
        \alpha, \beta, \rho-\rho_{m_0}, \sigma, \betabar  \rbrace, \\ \underline{R} \in \lbrace
         \beta, \rho-\rho_{m_0}, \sigma, \betabar, \alphabar \rbrace, \\ \mathcal{R}_1 ,\mathcal{R}_2 \in \lbrace \alpha, \beta ,\rho-\rho_{m_0}, \sigma, \betabar,\alphabar \rbrace}} \iint_{\mathcal{D}(u,\ubar) } \big(\sum_{j=0}^{i-1} \lvert \nabla^{j+1}\chihat \hsp  \nabla^{i-j-1}\underline{R} \hsp \nabla^i \underline{R} + \nabla^{j+1}\chibarhat \hsp  \nabla^{i-j-1}R \hsp \nabla^i R \rvert \\ &\hspace{46mm}+  \sum_{j=0}^{i-1} \lvert \nabla^{j}K \hsp  \nabla^{i-j-1}\underline{R} \hsp \nabla^i \underline{R} + \nabla^{j} K \hsp  \nabla^{i-j-1}R \hsp \nabla^i R \rvert \\ &\hspace{46mm}+ \lvert \nabla^i (\Gamma \cdot \mathcal{R}_1) \nabla^i \mathcal{R}_2\rvert  + \lvert \rho \rvert \lvert \nabla^i (\tr\chi- \tr\chi_{m_0}) \cdot \nabla^i \rho \rvert \\ &\hspace{46mm}+ \lvert \rho \rvert \lvert \nabla^i \chihat \cdot \nabla^i \alpha + \nabla^i \eta \cdot \nabla^i \beta + \nabla^i \chibarhat \cdot \nabla^i \alphabar + \nabla^i \etabar \cdot \nabla^i \betabar \rvert \big) \duprime \dubarprime.
    \end{split}
\end{equation}We notice at this point that in the terms of the form $\mathcal{R}_1 \cdot \mathcal{R}_2$ the term $\alpha \cdot \alphabar$ does not appear. We can, therefore, hence regard $\mathcal{R}_1 \cdot \mathcal{R}_2$ as either $\underline{R}_1 \cdot \underline{R}_2$ or $R_1 \cdot R_2$.  We now follow closely the approach of [LiYu] and obtain the following estimates using the Sobolev embedding theorem: 

\begin{equation}
    \sum_{j=0}^{1} \iint_{\mathcal{D}(u,\ubar)} \lvert \nabla^j \Gamma \cdot  \nabla^{i-j} R_1 \rvert^2 \duprime \dubarprime \lesssim \int_{u}^1 \int_{H_{\upr}^{(\delta,\ubar)}} \sum_{j=0}^i \lvert \nabla^j R_1 \rvert^2 \dubarprime \duprime \hspace{3mm} \text{for} \hspace{1mm} i\leq 3,  \end{equation}

\begin{equation}
    \iint_{\mathcal{D}(u,\ubar)} \lvert \nabla^2 \Gamma \cdot  \nabla^{i} R_1 \rvert^2 \duprime \dubarprime \lesssim \int_{u}^1 \int_{H_{\upr}^{(\delta,\ubar)}} \sum_{j=0}^{i+1} \lvert \nabla^j R_1 \rvert^2 \dubarprime \duprime \hspace{3mm} \text{for} \hspace{1mm} i\leq 1,  \end{equation}

\begin{equation}
    \iint_{\mathcal{D}(u,\ubar)} \lvert \nabla^3 \Gamma \cdot  R_1 \rvert^2 \duprime \dubarprime \lesssim \int_{u}^1 \int_{H_{\upr}^{(\delta,\ubar)}} \sum_{j=0}^{2} \lvert \nabla^j R_1 \rvert^2 \dubarprime \duprime.   \end{equation}Taking these into account and using the Cauchy--Schwartz inequality, we have
    
 \begin{equation}
     \sum_{i=0}^{3} \iint_{\mathcal{D}(u,\ubar)} \lvert \nabla^i (\Gamma \cdot R_1) \cdot \nabla^i R_2 \rvert^2 \duprime \dubarprime \lesssim \sum_{i=0}^3 \int_{u}^1 \int_{H_{\upr}^{(\delta,\ubar)}} \left( \lvert \nabla^i R_1 \rvert^2 + \lvert \nabla^i R_2 \rvert^2 \right) \duprime \dubarprime,  
 \end{equation}
 
 \begin{equation}
     \sum_{i=0}^{3} \iint_{\mathcal{D}(u,\ubar)} \lvert \nabla^i (\Gamma \cdot \underline{R}_1) \cdot \nabla^i \underline{R}_2 \rvert^2 \duprime \dubarprime \lesssim \sum_{i=0}^3 \int_{\delta}^{\ubar} \int_{H_{\ubar^{\prime}}^{(u,1)}} \left( \lvert \nabla^i \underline{R}_1 \rvert^2 + \lvert \nabla^i \underline{R}_2 \rvert^2 \right) \duprime \dubarprime.  
 \end{equation}The terms
\begin{equation}
    \begin{split}
        \sum_{j=0}^{2} \iint_{\mathcal{D}(u,\ubar)} \big( &\lvert \nabla^{j+1}\chihat \hsp  \nabla^{i-j-1}\underline{R} \hsp \nabla^i \underline{R} + \nabla^{j+1}\chibarhat \hsp  \nabla^{i-j-1}R \hsp \nabla^i R \rvert \\   &+\lvert \nabla^{j}K \hsp  \nabla^{i-j-1}\underline{R} \hsp \nabla^i \underline{R} + \nabla^{j} K \hsp  \nabla^{i-j-1}R \hsp \nabla^i R \rvert \big) \duprime \dubarprime
    \end{split}
\end{equation}can also be bounded by 
\[   \sum_{i=0}^3 \left( \int_{u}^1 \int_{H_{\upr}^{(\delta, \ubar)}}  \lvert \nabla^i R \rvert^2 \dubarprime \duprime + \int_{\delta}^{\ubar} \int_{\Hbar_{\ubar^{\prime}}^{(u,1)}} \lvert \nabla^i \underline{R} \rvert^2   \duprime \dubarprime \right). \]Finally, 

\begin{equation*}
 \sum_{i=0}^{3} \iint_{\mathcal{D}(u,\ubar)}   \lvert \rho \rvert \lvert \nabla^i (\tr\chi- \tr\chi_{m_0}) \cdot \nabla^i \rho \rvert + \lvert \rho \rvert \lvert \nabla^i \chihat \cdot \nabla^i \alpha + \nabla^i \eta \cdot \nabla^i \beta + \nabla^i \chibarhat \cdot \nabla^i \alphabar + \nabla^i \etabar \cdot \nabla^i \betabar \rvert \big)
\end{equation*}can be bounded, using Lemma \ref{lemmastep}, by

\begin{equation*}
    \begin{split}
        &\sup_{\delta \leq \ubar \leq \delta + \epsilon_0} \mathfrak{O}(\ubar,1)^2 + \sup_{1-\epsilon_0 \leq u \leq 1} \mathfrak{O}(\delta, u)^2 \\ &+\sum_{i=0}^{3} \left( \int_{u}^1 \sup_{1-\epsilon_0 \leq u^{\prime\prime} \leq u^{\prime} } \int_{H_{u^{\prime\prime} }^{(\delta,\ubar)}} \lvert \nabla^i R \rvert^2 + \int_{\delta}^{\ubar} \sup_{\delta \leq \ubar^{\prime\prime} \leq \ubar^{\prime}} \int_{\Hbar_{\ubar^{\prime\prime}}^{(u,1)}} \lvert \nabla^i \underline{R} \rvert^2  \right) .
    \end{split}
\end{equation*}We define

\[ \mathcal{E}(u) := \sum_{R \in \lbrace \alpha, \beta, \rho-\rho_{m_0}, \sigma, \betabar \rbrace} \sup_{u \leq \upr \leq 1} \sum_{i=0}^3 \int_{H_{\upr}^{(\delta, \delta+ \epsilon_0)}} \lvert \nabla^i R \rvert^2 ,  \] \[ \mathcal{F}(\ubar) := \sum_{R \in \lbrace  \beta, \rho-\rho_{m_0}, \sigma, \betabar, \alphabar \rbrace} \sup_{\delta \leq \ubar^{\prime} \leq \ubar } \sum_{i=0}^3 \int_{\Hbar_{\ubar^{\prime}}^{(1-\epsilon_0, 1)}} \lvert \nabla^i \underline{R} \rvert^2 .  \]The above estimates can then be summarized as

\begin{equation}
    \begin{split}
        \mathcal{E}(u)  + \mathcal{F}(\ubar)  &\lesssim  \mathcal{E}(1) + \mathcal{F}(\delta) \\ &+ \sup_{\delta \leq \ubar \leq \delta+ \epsilon_0} \mathfrak{O}(\ubar, 1)^2 + \sup_{1-\epsilon_0 \leq u \leq 1} \mathfrak{O}(\delta, \ubar)^2 + \left( \int_{u}^1 \mathcal{E}(\upr) \duprime + \int_{\delta}^{\ubar} \mathcal{F}(\ubar^{\prime}) \dubarprime \right).
    \end{split}
\end{equation}Applying Gr\"onwall's inequality, the result follows.

\subsection{Closeness to Schwarzschild}
With the transition region obtained, we now seek to show that some portion of the region is close to a patch of the Schwarzschild spacetime. The closeness is only obtained in some subregion of the L-shaped region of existence obtained in Theorem 4.1. The two crucial points here are that the size of the spacetime in which closeness is shown is independent of $\delta$, and that the closeness permits direct application of the Corvino-Schoen \cite{CS05} gluing arguments. The first point guarantees that when we take $\delta$ to be small, the region in which the gluing is performed remains of finite size, which is crucial for the argument in \cite{CS05}. 
\begin{theorem}[Geometry of the Transition Square]
In the region of existence obtained in Theorem 4.1, there is a  parameter $\epsilon_0>0$ depending only on $m_0$ and independent of $\delta$, such that the spacetime subregion $(M_{\epsilon_0},g)$ corresponding to $\delta\leq \ubar \leq \delta +\epsilon_0 $ and $1-\epsilon_0\leq u \leq 1$ is close to a compact subregion of a Schwarzschild spacetime with mass $m_0$ in the following sense 
\[ ||g-g_{m_0}||_{C^{k-3}(M_{\epsilon_0, g_{m_0}})} \lesssim \epsilon \]
\end{theorem}

\begin{proof} We recall  that we have managed to show, within $M_{\epsilon_0}$, the following bounds:

\be \label{closeness} \sup_{(u,\ubar) \in M_{\epsilon_0}} \sum_{k \leq \kmax} \left( \mathfrak{R}_k^{\ubar}(u) + \underline{\mathfrak{R}}_k^{\ubar}(u) + \mathfrak{O}_k(\ubar,u) \right) \lesssim \epsilon. \ee

\par \noindent We begin with the $C^{0}$ norms. In canonical double null coordinates $g$ can be written as follows

\[ g = -2 \hs \Omega^2  \left( \text{d}\ubar \otimes \text{d}u + \text{d}u \otimes \text{d}\ubar\right)    +\slashed{g}_{AB} \left( \text{d} \theta^A - b^A \text{d}\ubar \right) \otimes \left( \text{d} \theta^B - b^B \text{d}\ubar         \right). \]Using the bounds on $\omegabar-\omegabar_{m_0}$, $\chibar - \chibar_{m_0}$ and $\zeta$ due to \eqref{closeness} as well as the Sobolev inequalities, there holds \be \lVert g- g_{m_0} \rVert_{C^0(M_{\epsilon_0}, g_{m_0}) } \lesssim \epsilon. \ee Moving on to $C^1$ estimates, There holds $\nabla_{m_0}(g-g_{m_0}) = (\nabla_{m_0}-\nabla) g$.To obtain the $L^{\infty}$ bounds on $\nabla -\nabla_{m_0}$, we need the $L^{\infty}$ bounds on $\Gamma - \Gamma_{m_0}$ and $\slashed{\Gamma} - \slashed{\Gamma}_{m_0}$. Here $\Gamma$ refers to the null connection coefficients and $\slashed{\Gamma}$ to the Christoffel symbols of $\slashed{g}$. The bounds on \eqref{closeness} imply \[ \lvert \Gamma- \Gamma_{m_0} \rvert \lesssim \epsilon. \]  Moreover, to estimate $\slashed{\Gamma} -\slashed{\Gamma}_{m_0}$, given the variational formula $\mathcal{L}_{e_3}\slashed{g} = 2 \chibar$,  we need to control $\nabla \chibar$. But this also follows from \eqref{closeness}. Consequently, $\lvert \slashed{\Gamma} - \slashed{\Gamma}_{m_0} \rvert \lesssim \epsilon$ and hence

\be  \lVert g- g_{m_0} \rVert_{C^1(M_{\epsilon_0}, g_{m_0})} \lesssim \epsilon. \ee For $C^2$ estimates, there holds \[    \nabla_{m_0}^2(g-g_{m_0}) =  (\nabla^2 - \nabla_{m_0}^2) g =  \nabla \left(\nabla- \nabla_{m_0} g \right)  + (\nabla- \nabla_{m_0})\nabla_{m_0} g.\]The last term $ (\nabla- \nabla_{m_0})\nabla_{m_0} g$ has been controlled. Hence, we need the $L^{\infty}$ bounds of the following quantities:

\[   \nabla_4 \left(\Gamma- \Gamma_{m_0}\right), \nabla_4 \left(\slashed{\Gamma}- \slashed{\Gamma}_{m_0}\right),  \nabla_3 \left(\Gamma- \Gamma_{m_0}\right), \nabla_3 \left(\slashed{\Gamma}- \slashed{\Gamma}_{m_0}\right),  \nabla \left(\Gamma- \Gamma_{m_0}\right), \nabla \left(\slashed{\Gamma}- \slashed{\Gamma}_{m_0}\right) \]The estimates for the first four quantities can be obtained using the null structure equations,  provided we have  $L^{\infty}$ bounds of all first derivatives of null curvature components. We lack control at this stage, however, on \[ \nabla_4 \left(\omega-\omega_{m_0}\right), \nabla_3 \left(\omegabar- \omegabar_{m_0} \right). \]The estimates on the fifth quantity can be directly inferred from \eqref{closeness}. Finally, the bounds for the last quantity are also obtained by the propagation equation for $\slashed{\Gamma} - \slashed{\Gamma}_{m_0}$, provided we have an $L^{\infty}$ bound on $\chibarhat$. The goal of obtaining $C^2$ estimates on $g-g_{m_0}$ then reduces to proving bounds for \[  \lVert \nabla_4 \left(\omega-\omega_{m_0}\right) \rVert_{L^{\infty}},  \lVert \nabla_3 \left(\omegabar- \omegabar_{m_0} \right) \rVert_{L^{\infty}} \hs \hs \text{and} \hs \hs \lVert \nabla^2 \chibarhat \rVert_{L^{\infty}}. \]  The bounds on $\nabla^2 \chi$ follow from \eqref{closeness} and the Sobolev inequalities. For the first two quantities, notice that $\nabla_4 (\omega-\omega_{m_0})$ satisfies the following transport equation

\[    \nabla_3 \nabla_4 (\omega-\omega_{m_0}) = \nabla_4 \nabla_3 \left(\omega- \omega_{m_0}\right) + \text{l.o.t} = \nabla_4 (\rho- \rho_{m_0}) +\text{l.o.t} = \nabla \beta + \text{l.o.t} \]By \eqref{closeness} again, we have $\lVert \nabla \beta \lVert_{L^{\infty}} \lesssim \epsilon$. The same procedure can be done for $\nabla_3 (\omegabar-\omegabar_{m_0})$. It follows that \[     \lVert g-g_{m_0} \rVert_{C^2(M_{\epsilon_0},g_{m_0})} \lesssim \epsilon. \] An induction argument now proves the same control holds up to $C^{\kmax-1}$. The result follows.

\end{proof}

\section{Horizon Formation and the Penrose Inequality}
We split this section in two parts. \\ \indent 
We first treat the emergence of a unique smooth MOTS, for each $\ubar \in [\gamma \frac{a^{1/2}}{b}\delta,2\delta]$ for some $\epsilon$ to be specified. Together, these MOTS are leaves of a smooth topological hypersurface $S^2\times \mathbb{R}$ in the spacetime. That this hypersurface is in fact everywhere \textit{spacelike} can be shown under an additional assumption on the initial data; something which we explain but not assume for our initial data. \\ \indent 
We then show how this leads to a dynamical setting in which to test the spacetime Penrose inequality.\\ \indent 
The full set of details for the arguments below are contained in the works \cite{CS05}, \cite{LY15}, \cite{A17}, and as such our task here is to describe how to apply these in the current setting. 
\subsection{Horizon Emergence}
Here we show how An's argument \cite{A17} leads to the existence of a horizon with the desired properties. Our goal is to prove the following. 
\begin{theorem}
Take as starting point the existence theorem of \cite{AL17}. Suppose that the initial data $|\chihat_0|$ is prescribed as in Section 1.4. Then, for each $\ubar \in [\gamma \delta  \frac{a^{\frac{1}{2}}}{b}, 2\delta]$, there is a unique spherical MOTS $M_{u,\ubar}$ and together these form a smooth topological hypersurface $S^2\times \mathbb{R}$ in the spacetime. 
\end{theorem} 
Since the argument \cite{A17} is long and computationally heavy, we will simply describe the necessary adjustments.\\ \\ 
To begin with, let us quickly recall what is done in \cite{AL17}. The setting there is $b\leq a$ and $\delta a^{1/2}b<1$, and existence is shown for $[\delta a^{1/2}b \leq u\leq 1]$, $[0\leq \ubar \leq \delta]$ with $u$ decreasing towards the future. They then show that after imposing a bound on the incoming shear of the initial outgoing null hypersurface
\begin{equation}
\inf_{\omega\in S_{0,0}} \int^\delta_0 |\chihat_0(\ubar',\omega)|^2 d\ubar' \geq 4b\delta a^{1/2}
\end{equation}
the sphere $S_{b\delta a^{1/2},\delta}$ is trapped.\footnote{$\omega$ previously denoted a Ricci component. In this section it also denotes the angular co-ordinates on the spheres foliating the initial null hypersurface $H_{u=1}$.}\\ \\
An takes as starting point the existence theorem of \cite{AL17} but constructs the initial shear differently to (5.1), yielding extra control permitting to prove the formation of the relevant MOTS hypersurface. To explain his work we first briefly recall the relevant previous work on which it is based.  \\ \\
For what follows we will use the $u$ co-ordinate and the gauge function $\Omega$ as they appear in \cite{KLR14} and \cite{A17}. $u$ in \cite{KLR14} and \cite{A17} corresponds to $1-u$ in the current set-up, which also what is employed in \cite{AL17}.  \\ \\
In \cite{KLR14} the authors prescribe the co-ordinates of a 2-sphere embedded in the $\ubar=\delta$ hypersurface by $(\ubar=\delta,u=1-R(\omega),\omega)$ where $\omega$ denotes the angular co-ordinates on the initial sphere at $\ubar=0$. The $u$ co-ordinate of this sphere is allowed to vary with $\omega$ and so this sphere is not one of the $S^2$ that appears in the original double null foliation (i.e. it is not defined by $H_u \cap \underline{H}_{\ubar}$). To compute the null expansion of this sphere, one transforms to an adapted frame $\{e\} \to \{e'\}$ where
$e_3=e_3'$, $e'_a=e_a-\Omega e_a(R)e_3$, $e'_4=e_4 -2\Omega e^a(R)e_a+\Omega^2|\nabla R|^2e_3$, where by definition $e_3(R)=\Omega^{-1}$ and consequently $g(e'_a,e'_b)=g(e_a,e_b)=\delta_{ab}$, $g(e'_4,e'_a)=g(e'_4,e'_4)=0$, $g(e'_3,e'_4)=-2$. In this frame, the authors in \cite{KLR14} show, owing to the pre-existing estimates in the slab, that this new sphere is trapped, $\tr \chi'<0$, provided there holds a certain elliptic inequality on the initial sphere, which they then show can be satisfied. \\ \\
The method of An \cite{A17} is entirely analogous. Since the aim is to identify a MOTS for each $\ubar$, one aims to solve the elliptic PDE corresponding to the MOTS condition $\tr \chi'=0$. Combining with the computation in \cite{KLR14}, he shows that the null expansion of the sphere located at $(u=1-R(\ubar,\omega), \ubar, \omega)$ is given by the following.
\begin{proposition}[A18]
\[\tr \chi'=\tr \chi -2\Omega \Delta' R-4\Omega \eta \cdot \nabla R-\Omega^2 \tr \chibar |\nabla R|^2 -8\Omega^2 \omegabar |\nabla R|^2\]
\end{proposition}
Here, for each $\ubar$, $\Delta'$ denotes the Laplace-Beltrami operator on the sphere $(u=1-R(\omega),\ubar,\omega)$ and, as before, $\nabla$ denotes the induced covariant derivative on $S_{u,\ubar}$.\\ \\
This suggests defining the following operator
\begin{equation}
L(R)\equiv \Delta'_M R+2 \eta \cdot \nabla R-\Omega \tr \chibar |\nabla R|^2 -8 \omegabar |\nabla R|^2-\frac{\Omega^{-1}}{2}\tr \chi
\end{equation}
so that finding a MOTS on each $\ubar$ reduces to solving the following.\footnote{Note here that $\omega$ denotes angular co-ordinates of the initial sphere and that $R(\omega)$ is a function of these co-ordinates.}
\begin{equation}
L(R(\omega))=0
\end{equation}
To do so, An makes the choice $b\leq a^{1/2}$ and imposes the following condition on the initial data.
\begin{equation}
\int^{\ubar}_0 |\chihat_0(\ubar',\omega)|^2d\ubar'=f(\ubar,\omega)\ubar a \: \: \: \text{for each} \: \: \: \frac{b\delta}{a^{1/2}} \leq \ubar \leq \delta 
\end{equation}
after having made the choice $b\leq a^{1/2}$, where $f(\ubar,\omega)$ is a smooth function with properties $\frac{20}{21}\leq f(\ubar,\omega)\leq \frac{22}{21}$ and $|\partial^i_{\omega} f(\ubar,\omega)|\lesssim 1$ for all $i\in \mathbb{N}$ and $\omega\in \mathbb{S}^2$. He then shows that along every $\underline{H}_{\ubar}$ for $\frac{b\delta}{a^{1/2}}\leq \ubar \leq \delta$ there is a unique MOTS. He then also shows that if one takes $i,j\to \infty$ in the initial data of \cite{AL17}, then the MOTS are strung together in a smooth topological hypersurface in the spacetime.\\ \\
\textbf{Remark}. In \cite{A17} the result just mentioned appears as Theorem 1.4. \cite{A17} also contains Theorems 1.5 and 1.7. In Theorem 1.5 the integral condition is extended to $0 \leq \ubar \leq \delta$, and in Theorem 1.7 he considers an additional null hypersurface data $\ubar \in [\delta,2\delta]$ along which $\chihat_0=0$. \\ \indent 
We note here that the integral condition becomes increasingly difficult to construct as the lower bound of $\ubar$ approaches $0$. Since $\chihat_0$ is allowed to have angular freedom (its zero set moves from sphere to sphere), the angular derivatives of $\chihat_0$ become increasingly large as $\ubar\to 0$. It is not clear to us, from Appendix B in \cite{A17}, whether $\chihat_0$ is $C^1$ or even $C^0$ in the angular directions on approach of $\ubar=0$. We completely avoid such issues by setting integral condition for the interval $\ubar\in [\gamma \frac{a^{\frac{1}{2}}}{b}\delta,\lambda \delta]$. \\ \indent 
As for Theorem 1.7, the conclusion he obtains is a \textit{piecewise} smooth dynamical horizon with a $C^1$ discontinuity at $\ubar=\delta$. This discontinuity is due to a $C^0$ discontinuity in $\chihat_0$ at $\ubar=\delta$. This discontinuity is not present in our data and we will show that the dynamical horizon obtained is entirely smooth. We avoid discontinuous data to make the connection with the Final State and Weak Cosmic Censorship conjectures more direct (i.e. global well posedness ought not to be expected for discontinuous initial data).  \\ \\ 
As in \cite{KLR14}, the \cite{A17} argument relies on pre-established dynamical estimates, which in this case come from \cite{AL17} and are as follows. 
\begin{gather}
    |\chibarhat, \omegabar, \eta, \tr \chibar +\frac{2}{R(\omega)}| \leq \frac{\ubar a^{1/2}}{R(\omega)^2}\\
    |\Omega-1| \leq \frac{\ubar a^{1/2}b^{1/4}}{R(\omega)} \\
    |\tr \chi -\frac{2}{R(\omega)}+\frac{1}{R(\omega)^2} \int^{\ubar}_0 |\chihat_0|^2(\ubar',\omega)d\ubar' | \leq \frac{\ubar a^{1/2}b^{1/4}}{R(\omega)^2}
\end{gather}
Note that owing to a different orientation for $u$ in \cite{AL17}, they appear in a form where $R$ is replaced with $u$ and $\Omega$ with $\Omega^{-1}$. To de-clutter, we drop the  $\omega$ dependence in $R$. \\ \\
With the operator (5.2) in hand, An considers two PDEs and uses a method of continuity argument to prove the existence of a solution to (5.3). Given the initial data (5.4), the appropriate PDEs are 
\begin{equation}
    \Delta' R+\frac{1}{2}\Omega \tr \chibar |\nabla R|^2 -\frac{1}{R}+\frac{a\ubar }{2R^2}[1+(f(\ubar,\omega)-1)\lambda] = 0
\end{equation}
and
\begin{equation}
    \Delta' R+\frac{1}{2} \Omega \tr \chibar |\nabla R|^2 -\frac{1}{R}+\frac{a\ubar}{2R^2} + \lambda[2\eta_b \nabla^b R+4\Omega\omegabar |\nabla R|^2 -\frac{\Omega^{-1}}{2}\tr \chi +\frac{1}{R} -\frac{a \ubar f(\ubar,\omega)}{2R^2}] = 0
\end{equation}
Using the estimates (5.5-7), and absorbing lower order terms into $c_1,c_2,c_3$, these equations become equivalent, for any $\lambda\in [0,1]$, to the following PDE
\begin{equation}
     \Delta' R -\frac{1}{R}|\nabla R|^2-\frac{1}{R}+ \frac{a \ubar f(\ubar,\omega)}{2R^2} + \ubar a^{1/2}c_{1,a} \frac{\nabla^a R}{R^2} + \ubar a^{1/2}c_{2bc} \frac{\nabla^b R\nabla^c R}{R^2}+\frac{\ubar a^{1/2}c_3}{R^2} = 0
\end{equation}
where $\frac{20}{21}\leq f(\ubar,\omega)\leq \frac{22}{21}$, $|c_1,c_2,c_3|\leq b^{1/4}$ and $c_1,c_2,c_3$ do not depend on $\nabla R$, $\nabla ^2 R$ but only on $R$, and $\zeta(\ubar) (1-\frac{1}{c_2})\leq \zeta(\ubar,\omega)\leq \zeta(\ubar)(1+\frac{1}{c_2})$. We label (5.10) by $H=0$.\\ \\
This equation straightforwardly leads to a $C^0$ estimate, a different version of which is shown below.
\\ \\
For the rest of what follows let $F(R,\lambda)$ denote the left hand side of (5.8) and $G(R,\lambda)$ that of (5.9). Note that $F(R,\lambda=0)=G(R,\lambda=0)$, and that $\tr \chi'=0$ is equivalent to $G(u,\lambda=1)=0$.\\ \\
The first observation is that $F(R,\lambda=0)=0$ has an explicit solution $R= \frac{a\ubar}{2}$. Thus, to solve $F(R,\lambda)=0$, the method of continuity tells us that it suffices to show that $F(\tilde{R},\tilde{\lambda})[W]\equiv \lim_{\epsilon\to 0} \frac{1}{\epsilon}(F(\tilde{u}+\epsilon W,\tilde{\gamma})-F(\tilde{u},\tilde{\gamma}))$ is invertible for $W$. This will follow from \textit{a priori} estimates gathered for (5.8). \\ \\
Solving $G(R,\lambda)=0$ is done analogously, where we note that the explicit solution for $F(R,\lambda=0)=0$ is also a solution to $G(R,\lambda=0)=0$. By the method of continuity, it suffices to show that $G(\tilde{R},\tilde{\lambda})[W]\equiv \lim_{\epsilon\to 0} \frac{1}{\epsilon}(G(\tilde{R}+\epsilon W,\tilde{\lambda})-G(\tilde{R},\tilde{\lambda}))$ is invertible for $W$ when $\tilde{\lambda}$ is close to $\lambda$. One can show this explicitly from \textit{a priori} estimates for (5.9) and by using the analysis that was done for $F(R,\lambda)$. Note that the role of $F(R,\lambda)$ is simply to facilitate the analysis of $G(R,\lambda)$. In particular, the invertibility of $G(\tilde{R},\tilde{\lambda})$ turns out to involve only modest modifications to the analysis of $F(\tilde{R},\tilde{\lambda})$.\\ \\
The uniqueness and regularity parts of An's argument proceed using what has already established. The argument is computationally demanding but the ideas proceed naturally. Supposing there are two solutions $R$ and $\tilde{R}$, one writes two versions of (5.3), which in turn leads to an equation for $\Delta_R (\tilde{R}-R)$. Expanding the Laplacian $\Delta_R (\tilde{R}-R)$ into its components and using the estimate 
\begin{equation}
    \frac{\partial^2 R}{ \partial \theta_i \partial \theta_j} \ll \ubar a
\end{equation} 
coming from the $C^1$ estimate 
\begin{equation}
    |\nabla R|\ll 1
\end{equation} 
then yields estimates for each of these components. Using a null structure equation and a coarse version of the $C^0$ estimate $\frac{3}{8}\ubar \leq R,\tilde{R} \leq \frac{5}{8}\ubar a$, one obtains a further estimate 
\begin{equation}
    \nu(\omega)\geq \ubar a \frac{1}{(\frac{3\ubar a}{8})^3} =\frac{64}{81\ubar^2 a^2}
\end{equation}
where $\nu(\omega)$ is a term defined in terms of the decomposition for the leading term in the equation for $\tilde{R}-R$. \\ \\ 
Based on these, the equation for the difference $\tilde{R}-R$ becomes
\begin{equation}
    \Delta_R(\tilde{R}-R)-\nu(\omega)(\tilde{R}-R)+\frac{1}{\ubar^2 a^2}(R-\tilde{R})o(1)+\frac{1}{\ubar^2 a^2}\frac{\partial}{\partial \theta_i}(\tilde{R}-R)o(1)=0
\end{equation}
and using the estimate for $\nu(\omega)$ along with the maximum principle yields $\tilde{R}=R$, as desired. \\ \\
Regularity proceeds somewhat analogously. One starts by writing two versions of equation (5.3) for $\ubar$ and $\ubar'$ and one writes the equation for 
\begin{equation}
    \Delta_R \left(\frac{R(\ubar',\omega)-R(\ubar,\omega)}{\ubar'-\ubar}\right)
\end{equation} 
which is then split into two components, each of which are estimated. Expanding (5.15) into its components and using the estimates (5.12) and the $C^0$ estimate (5.11), when $\ubar$ is close to $\ubar'$, one obtains uniform (independent of $\ubar$) upper and lower bounds on the quantity $\nu(\ubar,\omega;\ubar)$, defined analogously to $\nu(\omega)$ above. This leads to the following equation 
\begin{equation}
\begin{split}
    \Delta_R\left(\frac{R(\ubar',\omega)-R(\ubar,\omega)}{\ubar'-\ubar}\right) - \frac{\nu(\ubar,\omega;\ubar')}{\ubar^2 a^2}\left(\frac{R(\ubar',\omega)-    R(\ubar,\omega)}{\ubar'-\ubar}\right)\\ +\tilde{\nu}(\ubar,\omega;\ubar')\frac{a}{\ubar^2a^2}+\frac{1}{\ubar^2a^2}\frac{R(\ubar',\omega)-R(\ubar,\omega)}{\ubar'-\ubar}o(1)\\
    +\frac{1}{\ubar^2 a^2}\frac{\partial}{\partial \theta_i}\frac{R(\ubar',\omega)-R(\ubar,\omega)}{\ubar'-\ubar}o(1)=\frac{a}{\ubar^2 a^2}o(1)
\end{split}
\end{equation}
The next step is to define a function $\mathfrak{h}(\ubar,\omega;\ubar')$ via the following equation
\begin{equation}
    \Delta_R \mathfrak{h}(\ubar,\omega;\ubar') -\frac{\nu(\ubar,\omega;\ubar')}{\ubar^2 a^2}\mathfrak{h}(\ubar,\omega;\ubar')+\tilde{\nu}(\ubar,\omega;\ubar')\frac{a}{\ubar^2 a^2}=0
\end{equation}
where $\tilde{\nu}(\ubar,\omega;\ubar')$ is a quantity defined in terms of the components one obtains in the expansion of the components of (5.15). \\ \\
Given the bounds on $\nu$ and $\tilde{\nu}$, what has already been shown makes it clear that there is a unique smooth solution $\mathfrak{h}(\ubar,\omega;\ubar')$. One uses it to rewrite (5.16) as follows.
\begin{equation}
\begin{split}
    \Delta_R\left(\frac{R(\ubar',\omega)-R(\ubar,\omega)}{\ubar'-\ubar}-\mathfrak{h}(\ubar,\omega;\ubar') \right) - \frac{\nu(\ubar,\omega;\ubar')}{\ubar^2 a^2}\left(\frac{R(\ubar',\omega)-R(\ubar,\omega)}{\ubar'-\ubar}-\mathfrak{h}(\ubar,\omega;\ubar')\right)\\ +\frac{1}{\ubar^2a^2}\left(\frac{R(\ubar',\omega)-R(\ubar,\omega)}{\ubar'-\ubar}-\mathfrak{h}(\ubar,\omega;\ubar') \right)o(1)\\
    +\frac{1}{\ubar^2 a^2}\left(\frac{\partial}{\partial \theta_i}\frac{R(\ubar',\omega)-R(\ubar,\omega)}{\ubar'-\ubar}-\mathfrak{h}(\ubar,\omega;\ubar')\right)o(1)=\frac{a}{\ubar^2 a^2}o(1)
\end{split}
\end{equation}
Working now with (5.16), standard elliptic theory eventually leads to 
\begin{equation}
    ||\frac{R(\ubar',\omega)-R(\ubar,\omega)}{\ubar'-\ubar}-\mathfrak{h}(\ubar,\omega)||_{L^\infty(M_{\ubar})} \leq ao(1)
\end{equation}
where $\mathfrak{h}(\ubar,\omega)$ is defined as the solution to an equation of the form (5.17) but with $\nu$ and $\tilde{\nu}$ replaced with their respective limits as $\ubar\to \ubar'$.\\ \\ Further standard elliptic theory combined with the estimates obtained yields 
\begin{equation}
    \frac{\partial R}{\partial \ubar}\in C^\infty(M_{\ubar})
\end{equation}
One then studies 
\begin{equation}
    \frac{\frac{\partial R}{\ubar}(\ubar',\omega)-\frac{\partial R}{\partial \ubar}(\ubar,\omega)}{\ubar'-\ubar}
\end{equation}
which by the same argument yields that 
\begin{equation}
    \frac{\partial^2 R}{\partial \ubar^2} \in C^{\infty}(M_{\ubar})
\end{equation}
Iterating this argument finally yields 
\begin{equation}
    \frac{\partial^k R}{\partial \ubar^k} \in C^{\infty}(M_{\ubar}) 
\end{equation}
for all $k\in \mathbb{Z}^+$, as desired.\\ \\ 
To establish the spacelike nature of the hypersurface $u=1-R(\ubar,\omega)$, An makes the additional assumption that on small discs $D_{\ubar} \subset S^2$ the initial shear satisfies
\begin{gather}
0\leq |\chihat_0(\ubar,\omega)|\leq a \hspace{0.5in} \text{for} \hspace{0.2in} \omega\in D_{\ubar} \\
|\chihat_0(\ubar,\omega)|^2=a \hspace{0.5in} \text{for} \hspace{0.2in} \omega \in D_{\ubar} \backslash S^2 \\
\int_{D_{\ubar}} d\omega \sim \frac{1}{c^2}
\end{gather}
where $1\ll c \ll b \leq a^{1/2}$. \\ \\
By the structure equation for $\nabla_3\chihat$ and dynamical estimates from [AL17] for the terms appearing in the $\nabla_3\chihat$ equation, we have
\begin{equation}
u^2|\chihat(u,\ubar,\omega)|\sim |\chihat(0,\ubar,\omega)|^2 
\end{equation}
Combining with previously established bounds and the assumption on the discs $D_{\ubar}$ eventually leads to 
\begin{equation}
    |\frac{\partial R(\ubar,\theta_1,\theta_2)}{\partial \ubar}-h(\theta_1,\theta_2)|\leq ao(1)
\end{equation}
and 
\begin{equation}
    h(\ubar,\theta_1,\theta_2)=(\frac{1}{2}+o(1))a
\end{equation}
where $\theta_1,\theta_2$ are angular co-ordinates on the sphere at $\ubar, u=1-R$. \\ \\
Given that $a$ is a fixed large positive constant, the tangent vectors $\frac{\partial}{\partial \ubar}$, $\frac{\partial}{\partial \theta_1}$, $\frac{\partial}{\partial \theta_2}$ are all spacelike. Letting $\lambda_{1,2,3}$ be any real numbers, and using the estimate $h(\ubar,\theta_1,\theta_2)=(\frac{1}{2}+o(1))a$, one computes
\begin{equation}
    \begin{split}
    g'(\lambda_1\frac{\partial }{\partial \theta_1}+\lambda_2\frac{\partial }{\partial \theta_2}+\lambda_3\frac{\partial }{\partial \ubar},\lambda_1\frac{\partial }{\partial \theta_1}+\lambda_2\frac{\partial }{\partial \theta_2}+\lambda_3\frac{\partial }{\partial \ubar})\\
    =\lambda_1^2g_{\theta_1\theta_1}+\lambda^2g_{\theta_2\theta_2}+4\lambda_1\lambda_3\frac{\partial R}{\partial \theta_1}+4\lambda_2\lambda_3\frac{\partial R}{\partial \theta_2}+\lambda_3^2h(\ubar,\theta_1,\theta_2)(1+o(1))
\end{split}
\end{equation}
Seeking to show $g'(\cdot,\cdot)$ is always $>0$, the interesting term is 
\begin{equation}
    \lambda_3\left(4\lambda_1\frac{\partial R}{\partial \theta_1}+4\lambda_2\frac{\partial R}{\partial \theta_2}+\lambda_3h(\ubar,\theta_1,\theta_2)(1+o(1))\right) 
\end{equation}
The mixed terms $\frac{\partial R}{\partial \theta_{1,2}}$ involve a power of $\delta$ compared to $h(\ubar,\theta_1,\theta_2)$, and so these are lower order. If $\lambda_{1,2}$ are large compared to $\lambda_3$ by an order $\delta^{-1}$ or more, then this term could be negative but in that case at least one of first two terms in (5.25) is very large and positive. \\ \\
So $g'(\cdot,\cdot)>0$ for any spacelike vector $\cdot$ and the horizon is spacelike. \\ \\
In summary, the main argument in \cite{A17} goes as follows. 
\begin{enumerate}
    \item Use the idea of \cite{KLR14} to write a MOTS equation. Use it to write two further equations, $F=0$ and $G=0$, and use the \textit{a priori} estimates of \cite{AL17} to derive a third equation, $H=0$, from $F=0$ and $G=0$.   
    \item Derive a $C^0$ estimate, with $\alpha$ coming from the bounds on $f(\ubar,\omega)$.
    \begin{equation}
        (1-\frac{1}{\alpha})(\frac{1}{2}+o(1))\ubar a \leq R(\ubar,\theta) \leq (1+\frac{1}{\alpha})(\frac{1}{2}+o(1))\ubar a
    \end{equation}
    \item Use (5.32) to obtain a $W^{1,p}$ estimate 
    \begin{equation}
        \int_{M_{\ubar}} |\nabla R|^2 \ll \ubar a
    \end{equation}
    \item Define a function $h(R)=1+\frac{8}{\ubar^2a^2}(R-\frac{\ubar a}{2})^2$, and use (5.32), Bochner's formula, and an estimate for the Ricci curvature, to obtain the $C^1$ estimate
    \begin{equation}
        |\nabla R|\ll 1
    \end{equation}
    \item Use (5.34) and standard elliptic theory to obtain 
    \begin{equation}
        |\frac{\partial^2}{\partial \theta_i \partial \theta_j} R | \ll \ubar a  
    \end{equation}
    \item Use (5.34/5) and other \textit{a priori} estimates to work through two continuity arguments for $F$ and $G$ to yield existence of solution for the MOTS equation. 
    \item Use (5.32/4) and standard elliptic theory to show uniqueness. 
    \item Use (5.32/35) and standard elliptic theory to show 
    \begin{equation}
        \frac{\partial^k R}{\partial \ubar^k} \in C^\infty(M_{\ubar})
    \end{equation}
    for all $k$.
    \item Combine dynamical and elliptic estimates so that the assumption on the discs $D_{\ubar}$ yields a uniform estimate for $h(\ubar,\theta_1,\theta_2)$ to prove that the horizon is spacelike. 
\end{enumerate}
An also uses this proof to show that if one places trivial data $\chihat_0=0$ on a subsequent null hypersurface $\ubar\in [\delta,2\delta]$, then there remains a dynamical horizon satisfying the relevant properties. He can no longer prove that it is spacelike however, and the differences from the above argument are as follows.
\begin{itemize}
\item The equation $H=0$ derived from $F=0$ and $G=0$ is different owing to the fact that $\chihat_0=0$ is on this null hypersurface. This leads to a $C^0$ estimate that is independent of $\ubar$. 
\item The $C^{1}$ estimate is the same, but it is obtained by defining $h(R)=1+\frac{8}{\delta a}(R-\frac{\delta a}{2})^2$. 
\item The proof of existence and uniqueness proceeds as above, but there is a difference for regularity. In working out the analog of (5.29), one obtains
\begin{equation}
    |h(\ubar,\theta_1,\theta_2)| \leq o(1) a
\end{equation}
whereas previously that had been
\begin{equation}
  h(\ubar,\theta_1,\theta_2) = (\frac{1}{2}+o(1))a
\end{equation}
Since (5.28) holds in both cases, the quantity $\frac{\partial R}{\partial \ubar}$ must jump at $\ubar=\delta$. This $C^1$ discontinuity in the dynamical horizon is caused by a $C^0$ discontinuity in $\chihat_0$ on the initial null hypersurface, which we avoid in our setting.
\end{itemize}
\newpage 
Our initial data differs from \cite{A17} in the following sense.  
\begin{itemize}
\item First note that $\delta a^{1/2}b<1$, $\delta a^{1/2} b^{\mu}<1$ and that $a^{1/2} < b$.
    \item Let $\lambda$ be a constant satisfying $1>\lambda>\frac{a^{\frac{1}{2}}}{b}>0$, $\mu>1$ a constant to be specified, and $\gamma$ a free $o(1)$ parameter $>0$. Then require, for all $\ubar \in [\gamma \frac{a^{\frac{1}{2}}}{b}\delta,\lambda \delta]$ the following
\begin{equation}
\int_0^{\ubar} | \chihat_0|^2 (\ubar', \omega) \hspace{.5mm} \text{d}\ubar' =a^{1/2}b^\mu f(\ubar,\omega) \ubar 
\end{equation} 
for a smooth (in $\ubar$ and $\omega$) function $f(\ubar,\omega)$ such that $1-\frac{1}{c_1} \leq f(\ubar,\omega) \leq 1+\frac{1}{c_1}$ for a constant $c_1\geq 20$, and moreover $| \partial_{\omega}^i f(\ubar,\omega)| \lesssim 1$ for all $i \in \mathbb{N}$ and all $\omega \in \mathbb{S}^2$.

\item For $\ubar \in [\lambda \delta, \lambda' \delta]$ where $\lambda'$ is a constant such that $\lambda <\lambda'<1$, we have 
    \begin{equation}
       |\chihat_0(\ubar,\omega)|^2 = \mathcal{A}(\ubar, \omega) |\chihat_0(\ubar=\lambda \delta,\omega)|^2 
    \end{equation}
where $\mathcal{A}(\ubar,\omega)$ is a smooth (in $\ubar$ and $\omega$) cut-off function with $\mathcal{A}(\lambda \delta,\omega)=1$ and $\mathcal{A}(\lambda'\delta,\omega)= 0$. 
\item The total shear from $0$ to $\ubar > \lambda' \delta$ is dominated by the contribution from $\ubar \in [0,\lambda \delta]$ 
\begin{equation}
\begin{split}
\int_0^{\ubar} |\chihat_0(\ubar',\omega)|^2 d\ubar' = \int_0^{\lambda \delta} |\chihat_0(\ubar',\omega)|^2d\ubar' + \int_{\lambda \delta}^{\lambda'\delta} |\chihat_0(\ubar',\omega)|^2d\ubar' + \int_{\lambda'\delta}^{\ubar}|\chihat_0(\ubar',\omega)|^2d\ubar'\\
= M^*(\omega) +\epsilon(\lambda'\delta,\omega) +0=4m_0
\end{split}
\end{equation}
where $M^*(\omega)=\int_0^{\lambda \delta}|\chihat_0(\ubar',\omega)|^2d\ubar' \approx d_0 \epsilon(\lambda'\delta,\omega)$ for some universal large constant $d_0$ and 
\begin{equation}
    \int_0^{\ubar}|\chihat_0(\ubar',\omega)|^2 d\ubar' = b^\mu a^{1/2}\mu \ubar f(\ubar,\omega)\zeta(\ubar,\omega) + (1-\zeta(\ubar,\omega))4m_0
\end{equation}
where $\zeta(\ubar,\omega)$ is a smooth (in both $\omega$ and $\ubar$) function resulting from integrating $\mathcal{A}(\ubar,\omega)$, which is $=1$ for $\ubar \leq \lambda \delta$, and $=0$ for $\ubar \geq \lambda'\delta$, and moreover which satisfies $\zeta(\ubar) (1-\frac{1}{c_2})\leq \zeta(\ubar,\omega)\leq \zeta(\ubar)(1+\frac{1}{c_2}) $ for a constant $c_2\geq 20$, $|\partial^i_\omega \zeta(\ubar,\omega)|\lesssim 1$ for $ i \in \mathbb{N}$ and $\omega\in S^2$, with $\zeta(\ubar)$ a smooth cut-off function in $\ubar$ such that $\zeta(\lambda \delta)=1$ and $\zeta(\lambda'\delta)=0$. 
\end{itemize}
We now describe how this initial data permits the above argument in \cite{A17} to proceed more or less unchanged. \\ \\
First, the analogs to $F=0$, $G=0$ are as follows.
\begin{equation}
     \Delta' R+\frac{1}{2}\Omega \tr \chibar |\nabla R|^2 -\frac{1}{R}+\left(\frac{b^\mu a^{1/2} \ubar f(\ubar,\omega)\zeta(\ubar,\omega) + (1-\zeta(\ubar,\omega))4m_0 }{2R^2}\right) [1+(f(\ubar,\omega)-1)\lambda] = 0 
\end{equation}
\begin{equation}
\begin{split}
     \Delta' R+\frac{1}{2} \Omega \tr \chibar |\nabla R|^2 -\frac{1}{R}+\left(\frac{b^\mu a^{1/2} \ubar f(\ubar,\omega)\zeta(\ubar,\omega) + (1-\zeta(\ubar,\omega))4m_0}{2R^2}\right) +\\ \lambda \left[ 2\eta_b \nabla^b R+4\Omega\omegabar |\nabla R|^2 -\frac{\Omega^{-1}}{2}\tr \chi +\frac{1}{R} -\left(\frac{b^\mu a^{1/2} \ubar f(\ubar,\theta)\zeta(\ubar,\omega) + (1-\zeta(\ubar,\omega))4m_0)}{2R^2}\right)\right] = 0 
\end{split}
\end{equation}
Substituting the dynamical estimates from \cite{AL17} onto these equations whilst re-absorbing the contributions by lower order terms into $c_1,c_2,c_3$ leads to the following analog of $H=0$.
\begin{equation}
\begin{split}
     \Delta' R -\frac{1}{R}|\nabla R|^2-\frac{1}{R}+ \left(\frac{b^\mu a^{1/2} \ubar f(\ubar,\omega)\zeta(\ubar,\omega) + (1-\zeta(\ubar,\omega))4m_0}{2R^2}\right) + \\
     \ubar a^{1/2}c_{1,a} \frac{\nabla^a R}{R^2} + \ubar a^{1/2}c_{2bc} \frac{\nabla^b R\nabla^c R}{R^2}+\frac{\ubar a^{1/2}c_3}{R^2} = 0
     \end{split}
\end{equation}
where $(1-\frac{1}{\alpha_1})\leq f(\ubar,\omega)\leq (1+\frac{1}{\alpha_1})$, $|c_1,c_2,c_3|\leq b^{1/4}$ and $c_1,c_2,c_3$ do not depend on $\nabla R$, $\nabla ^2 R$ but only on $R$, and $\zeta(\ubar) (1-\frac{1}{c_2})\leq \zeta(\ubar,\omega)\leq \zeta(\ubar)(1+\frac{1}{c_2})$. \\ \\
For the $C^0$ estimate, we use (5.45) directly as follows.\\ \\
Let $M_0(\ubar,\omega)\equiv b^\mu a^{1/2} \ubar f(\ubar,\omega)\zeta(\ubar,\omega) + (1-\zeta(\ubar,\omega))4m_0$, and $R_{\text{max}}\equiv \max_\omega R(\omega)$ (and equivalently for $R_{\text{min}}$). At $R_{\text{max}}$ we have $\nabla R_{\text{max}}=0, \Delta R_{\text{max}}\leq 0$, and so (5.22) yields 
\[0\leq \frac{-2R_{\text{max}}+{M_0}_{\text{max}} +2\ubar a^{1/2}c_3}{2R^2_{\text{max}}} \] with $|c_3|\leq b^{1/4}$ which leads to 
\[R_{\text{max}}\leq (\frac{1}{2}+o(1)){M_0}_{\text{max}}\]
and
\[R_{\text{min}}\geq (\frac{1}{2}+o(1)){M_0}_{\text{min}}\]
which when combined gives 
\begin{equation}
\begin{split}
(\frac{1}{2}+o(1))(1-\frac{1}{\alpha_1})b^\mu a^{1/2} \ubar (1-\frac{1}{\alpha_2})\zeta(\ubar) + (1-\zeta(\ubar))(1-\frac{1}{\alpha_2})2m_0 \leq R \\ 
\leq (\frac{1}{2}+o(1))(1+\frac{1}{\alpha_1})b^\mu a^{1/2} \ubar (1+\frac{1}{\alpha_2})\zeta(\ubar) + (1-\zeta(\ubar))(1+\frac{1}{\alpha_2})2m_0 
\end{split}
\end{equation}
The terms involving $\alpha_1$ and $\alpha_2$ come from the the angular bounds on $f(\ubar,\omega)$ and $\zeta(\ubar,\omega)$ as in \cite{A17}, the $o(1)$ term comes from $\ubar a^{1/2}c_3$ which is $\ll 1$ since $\ubar\leq \delta$, $c_3\leq b^{1/4}$ and yet $\delta a^{1/2}b<1$ with $b,a$ both large constants. \\ \\
So long as $\alpha_1,\alpha_2$ are chosen suitably large, integrating (5.21) yields the following $W^{1,2}$ estimate.
\begin{equation}
\int_M |\nabla R|^2\lesssim  a^{1/2} b^\mu \ubar \zeta(\ubar)+(1-\zeta(\ubar))4m_0 
\end{equation}
The next and key estimate in \cite{A17} is the $C^1$ estimate $|\nabla R| \ll 1$. To obtain this, An considers a function $h(R)=1+\frac{8}{\ubar^2 a^2}(R-\frac{\ubar a}{2})^2$, computes $\Delta'(h(R)|\nabla R|^2) $ with the help of Bochner's formula, and estimates each term using the slab estimates derived from \cite{AL17}. In our case, we use
\begin{equation}
    h(R)=1+\left(\frac{8}{(\ubar a^{1/2}b^\mu\zeta(\ubar)+(1-\zeta(\ubar))4m_0)^2}\right)\left[R-\left(\frac{\ubar a^{1/2}b^\mu\zeta(\ubar)}{2}+(1-\zeta(\ubar))2m_0\right)\right]^2
\end{equation}$\ubar a$ with $\ubar a^{1/2}b^\mu\zeta(\ubar)+(1-\zeta(\ubar))4m_0$. Note that for suitable $\alpha_{1,2}$, many functions like $h(R)$ will do and in particular the factor `$8$' in \cite{A17} is to some extent arbitrary and can be replaced by any other real number with lower bound depending on how close $R$ is to $\frac{1}{2}\ubar b^\mu \delta a^{1/2}$.  \\ \\
For the first term estimated, one uses the transport equations for $\chi$, $\chihat$, $\tr \chi$ and the $\nabla_4$ equation for $\crho=\rho-\frac{1}{2}\chihat \cdot \chibarhat$ which reads
\begin{equation}
    \nabla_4 \crho =\frac{-3}{2}\tr \chi \crho +\div \beta +\zeta\cdot \beta +2\etabar \cdot \beta -\frac{1}{2}\chihat \cdot \nabla \hat{\otimes} \etabar -\frac{1}{2}\chihat\cdot(\etabar \hat{\otimes} \etabar)+ \frac{1}{4}\tr \chibar |\chihat|^2 
\end{equation}
Using once more the estimates in \cite{AL17}, one gets 
\begin{equation}
    \crho|_{S_{1-R,\ubar}} = \frac{-b^\mu a^{1/2} f(\ubar,\omega)\zeta(\ubar,\omega)+(1-\zeta(\ubar,\omega))4m_0}{2R^3}+\frac{\ubar a^{1/2}c_3}{R^3}
\end{equation}
and upon combining with the $\nabla_3$ equation for $\tr \chi$ this leads to  
\begin{equation}
    \begin{split}
        \nabla_3 \tr{\chi}|_{S_{1-R,\ubar}}=\frac{1}{R}(\frac{2}{R}-\frac{\ubar b^\mu a^{1/2} f(\ubar,\omega)\zeta(\ubar,\omega)+(1-\zeta(\ubar,\omega))4m_0}{R^2}) -\\
        \frac{b^\mu \ubar f(\ubar,\omega) a^{1/2}\zeta(\ubar,\omega)+(1-\zeta(\ubar,\omega))4m_0}{R^3} + \frac{ \ubar a^{1/2}c_3}{R^3}
    \end{split}
\end{equation}
where $c_3$ is used to mean a quantity $\leq b^{1/4}$. The analog of the quantity on pg.23 of \cite{A17} for which one needs a lower bound is 
\begin{equation}
    -2R+2 b^\mu a^{1/2}f(\ubar,\omega)\zeta(\ubar,\omega)+(1-\zeta(\ubar,\omega))8m_0
\end{equation}
which in \cite{A17} satisfies $>\frac{1}{2}R$ and so too in our case by our assumption on the initial data.  \\ \\
As in \cite{A17}, the $C^1$ estimate for $R$ will come from term-by-term analysis of (4.7) on p.22 of \cite{A17}. From the $C^0$ estimate (5.42) we get \begin{equation}
    |R-\left(\frac{\ubar b^\mu a^{1/2}\zeta(\ubar)+(1-\zeta(\ubar))4m_0}{2} \right)|\ \leq o(1)\left((1+\frac{1}{\alpha_1})(1+\frac{1}{\alpha_2}) \ubar b^\mu a^{1/2}\zeta(\ubar)+(1+\frac{1}{\alpha_2})(1-\zeta(\ubar))4m_0\right)
\end{equation}
With this it is clear that the estimates in the middle of p.24 \cite{A17} go through in their analogous form here, which in turn suffices to guarantee that $|\nabla R|\ll 1$.\\ \\
The other estimates - $W^{2,p}$, $C^{1,p}$, and $C^{2,q}$ - proceed as in \cite{A17} and involve no significant modification and so one finally obtains the following analog of (4.14) on p.27 of \cite{A17}
\begin{equation}
    |\frac{\partial^2}{\partial \theta_i \partial \theta_j}R|\ll b^\mu a^{1/2}\ubar\zeta(\ubar)+4m_0(1-\zeta(\ubar)) 
\end{equation}
\textbf{Existence}. The invertibility of $F(\tilde{R},\tilde{\gamma})[W]\equiv \lim_{\epsilon\to 0} \frac{1}{\epsilon}(F(\tilde{R}+\epsilon W,\tilde{\gamma})-F(\tilde{R},\tilde{\gamma}))$ proceeds virtually identically. Since we share the estimates for dynamical quantities as in  \cite{AL17}, the computation and estimates which lead to an expression for the coefficients of $I_1+I_2+I_3$ on p.30 of \cite{A17} will be identical bar the relevant replacements of $\ubar a$ with $\ubar  a^{1/2}b^\mu \zeta(\ubar)+(1-\zeta(\ubar))4m_0$, and since the estimate $|\nabla R|\ll 1$ holds, the invertibility of $F(R,\lambda)$ follows. The second continuity argument in section 6 of \cite{A17} proceeds identically, with $\ubar a \to \ubar a^{1/2}b^\mu\zeta(\ubar) +(1-\zeta(\ubar))4m_0$ and with $|c|\leq \frac{1}{a^{1/2}}\to \frac{1}{b^{1/4}}$ in the lower part of p.32 of \cite{A17}, but again the invertibility of $G(\tilde{R},\tilde{\lambda})$ follows straightforwardly. \\ \\
\textbf{Uniqueness}. This proceeds with the same kind of superficial modifications: $\ubar a \to \ubar b^\mu a^{1/2}\zeta(\ubar)+(1-(\zeta(\ubar))4m_0$. One key estimate is for a quantity $\nu(\omega)$ defined on p.35 and shown to obey the lower bound $\nu(\omega)\geq \frac{64}{81 \ubar^2a^2}$. This guarantees that uniqueness via a maximum principle argument. In our case the estimate becomes 
\begin{equation}
    \nu(\omega)\geq \frac{64}{81 (a^{1/2}b^\mu \ubar\zeta(\ubar)+(1-\zeta(\ubar))^2)4m_0)}
\end{equation} 
and the same maximum principle argument yields uniqueness. \\ \\
\textbf{Regularity}. As described, one starts by writing out the equation for $\Delta_{R(\ubar,\omega)}\frac{(R(\ubar',\omega)-R(\ubar,\omega)}{\ubar'-\ubar}$. The long computation yields terms which are estimated via the estimate $\frac{\partial^2 R(\ubar',\omega)}{\partial \theta_i \theta_i} \approx o(1)\ubar a $, a similar version of which holds in our case. Using our $C^0$ estimate (5.46) and (5.54), the argument p.41-4 of \cite{A17} proceeds.    \\ \\ 
\textbf{Spacelike}. Conditions (5.21)-(5.23) lead to the uniform estimate on $h(\ubar,\theta_1,\theta_2)$ which eventually permits proving that the horizon is spacelike. Although we could do something similar, we choose not to since the assumption is rather strong. Another issue is that our condition that 
\begin{equation}
    \int_0^\delta |\chihat_0(\ubar,\omega)|^2d\ubar'=4m_0
\end{equation}
be independent on $\omega$ implies that whatever happens in the small discs has to cancel appropriately. If we made such an assumption, $\chihat_0$ would have to compensate for this in the region $\ubar\in(0,\frac{a^{1/2}}{b}\delta)$, i.e., before the integral condition on $\chihat_0$ is imposed, and, presumably also after, i.e., from $\lambda \delta$ to $\lambda'\delta$. In that case, schematically, one obtains an estimate of the form 
\begin{equation*}
    h(\ubar,\theta_1,\theta_2)\sim [\frac{1}{2}(\zeta(\ubar) +\ubar\zeta'(\ubar) -\zeta'(\ubar))+o(1)]a^{1/2}b^\mu
\end{equation*}
which in turn would permit proving that the horizon is spacelike, at least for $\ubar\leq \lambda\delta$. We omit the details.

\subsection{Gluing}
The gluing procedure that eventually permits us a test of the Penrose inequality follows from Theorem 4.3, which gave the estimate for the spacetime metric $g$ in the region $\ubar\in [\delta+\epsilon_0]$ and $u\in [1,1-\epsilon_0]$
\begin{equation}
||g-g_{m_0}||_{C^{k-3}} 
\end{equation}
with $g_{m_0}$ the Schwarzschild metric with mass $m_0$. This is the same as that obtained in \cite{LY15} and since the spacetime is smooth, the gluing argument is exactly as in \cite{LY15}, which itself is based on \cite{CS05}. \\ \\
The result is simple to state. One can glue the $t=0$ slice of the Kerr spacetime onto the region $[\ubar,u]\in [\delta+\epsilon_0,1-\epsilon_0]$ with ADM mass and angular momentum satisfying
\begin{equation}
    |m-m_0|+|\textbf{a}|\lesssim \epsilon
\end{equation}
where $\epsilon$ is as in Theorem 1.8. The initial data of \cite{AL17} and estimates obtained in this region permits picking the $\delta$ dependence of $\epsilon$. The best possible choice gives $\epsilon =Ca^{1/2}\delta^{1/2}$ for a constant $C$ indepedent of $a$, $\delta$. \\ \\
Finally, the absence of trapped surfaces or MOTS on this slice follows by an elementary maximum principle argument and the fact that the slice is asymptotically flat.

\subsection{Area Estimate}
Here we give the areal estimate for the MOTS obtained in the region $\ubar\in [\gamma \frac{a^{1/2}}{b}\delta,\delta]$. On each MOTS $M_{\ubar}$, the induced metric is given by 
\begin{equation*}
    g'_{\theta_i \theta_j}=g_{\theta_i \theta_j}+\frac{\partial (1-R)}{\partial \theta_i} \frac{\partial (1-R)}{\partial \theta_j} g(\frac{\partial}{\partial u},\frac{\partial}{\partial u})=g_{\theta_i \theta_j}
\end{equation*}
So along $\Hbar_{\ubar}$ 
\begin{equation*}
    \sqrt{\text{det}g'}=\sqrt{\text{det}g}
\end{equation*}
By the first variation formula and properties of the double null foliation we have
\begin{equation*}
\frac{\partial }{\partial \ubar}\sqrt{\text{det}g}=\sqrt{\text{det}g}\Omega \text{tr}\chi
\end{equation*}
which in turn leads to
\begin{equation*}
\frac{\partial }{\partial \ubar}\sqrt{\text{det}g}=\sqrt{\text{det}g}\Omega \text{tr}\chi 
\end{equation*}
and
\begin{equation*}
|\frac{\sqrt{\text{det}g(u,\ubar,\theta_1,\theta_2)}}{\sqrt{\text{det}g(u,0,\theta_1,\theta_2)}}-1|\leq \frac{\ubar a^{1/2}b^{1/4}}{|u|}\ll 1
\end{equation*}
So for some constant $f_0\gg 1$. 
\begin{equation*}
    (1-\frac{1}{f_0})(\sqrt{\text{det}g(u,0,\theta_1,\theta_2)})\leq \sqrt{\text{det}g(u,\ubar,\theta_1,\theta_2)} \leq (1+\frac{1}{f_0})\sqrt{\text{det}g(u,0,\theta_1,\theta_2)}
\end{equation*}
which in turn gives 
\begin{gather*}
    |M_{\ubar}|=\int \int _{S^2} \sqrt{\text{det}g(u,\ubar,\theta_1,\theta_2)} d\theta_1 d\theta_2 \\
    \leq (1+\frac{1}{f_0}) \int \int _{S^2} \sqrt{\text{det}g(u,0,\theta_1,\theta_2)} d\theta_1 d\theta_2\\
    =(1+\frac{1}{f_0})4\pi |1-u|^2 
\end{gather*}
Using the $C^0$ estimate (5.46) for $R=1-u$ for  $\ubar\leq \lambda \delta$ and including the lower bound we obtain
\begin{equation}
    (\frac{1}{4}-o(1))b^\mu a^{1/2}\ubar \leq  \sqrt{\frac{|M_{\ubar}|}{16\pi}} \leq (\frac{1}{4}+o(1))b^\mu a^{1/2}\ubar 
\end{equation}
for some $o(1)\ll \frac{1}{4}$.

\end{document}